\documentclass[10pt]{article}
\usepackage{amsmath}
\usepackage{amsthm}
\usepackage{amsfonts}
\usepackage{amssymb}
\usepackage{latexsym}

\usepackage[matrix,tips,graph,curve]{xy}

\makeatletter

\def\@maketitle{%
  \newpage
  \null
  \let \footnote \thanks
    {\normalfont\sffamily\bfseries\Large\noindent\@title \par}%
    \vskip 1em%
    {\normalfont\sffamily\large
        \noindent
        \@author
        \par}
  \par
  \vskip 4em}
\def\@seccntformat#1{\csname the#1\endcsname{.\ }}
\renewcommand\section{\@startsection {section}{1}{\z@}%
                                   {-3.0ex \@plus -1ex \@minus -.2ex}%
                                   {1.5ex \@plus.2ex}%
                                   {\normalfont\large\bfseries}}
\renewcommand\subsection{\@startsection{subsection}{2}{\z@}%
                                     {-2.75ex\@plus -1ex \@minus -.2ex}%
                                     {1.5ex \@plus .2ex}%
                                   {\normalfont\normalsize\bfseries}}
\makeatother
\makeatletter 
\@addtoreset{equation}{section}
\makeatother
\def\abstract{\topsep=0pt\partopsep=0pt\parsep=0pt\itemsep=0pt\relax
\trivlist\item[\hskip\labelsep
{\bfseries\abstractname}.]\if!\abstractname!\hskip-\labelsep\fi}
\if@twocolumn

\else
   
\fi

\theoremstyle{plain}
\newtheorem{theorem}[equation]{Theorem}
\newtheorem{corollary}[equation]{Corollary}
\newtheorem{lemma}[equation]{Lemma}
\newtheorem{proposition}[equation]{Proposition}

\newtheorem*{pmt}{Positive Mass Theorem}
\newtheorem*{theoremA}{Theorem A}
\newtheorem*{theoremB}{Theorem B}
\newtheorem*{theoremC}{Theorem C}
\newtheorem*{theoremD}{Theorem D}

\theoremstyle{definition}
\newtheorem{definition}[equation]{Definition}

\newtheorem{remark}[equation]{Remark}

\newcommand{\sA}{\mathcal{A}}
\newcommand{\sB}{\mathcal{B}}
\newcommand{\sC}{\mathcal{C}}

\newcommand{\sH}{\mathcal{H}}

\newcommand{\sO}{\mathcal{O}}

\newcommand{\sU}{\mathcal{U}}

\newcommand{\sW}{\mathcal{W}}

\newcommand{\IR}{\mathbb{R}}

\newcommand{\IZ}{\mathbb{Z}}

\newcommand{\Hom}{\mathrm{Hom}}

\newcommand{\spin}{\rm Spin}
\newcommand{\Spin}{\rm Spin}
\newcommand{\SO}{\rm SO}

\newcommand{\Ker}{\mathrm{Ker \,}}
\newcommand{\Coker}{\mathrm{Coker \,}}

\renewcommand\dim{{\rm dim\,}}

\renewcommand\Im{\mathrm{Im \,}}
\newcommand\Range{\mathrm{Range}}

\newcommand\iso{{\cong}}

\newcommand\union{\bigcup} 
\newcommand\onehalf{\frac{1}{2}}

\newcommand\widebar{\overline}

\newcommand{\<}{\langle}
\renewcommand{\>}{\rangle}

\def\d/{/\mspace{-6.0mu}/}

\newcommand{\db}{\bar{\partial}}

\newcommand{\p}{\partial}

\newcommand{\supp}{\mathrm{supp}}
\newcommand{\sing}{\mathrm{sing}}
\newcommand{\Dmin}{\mathrm{Dom}_{\mathrm{min}}}
\newcommand{\Dmax}{\mathrm{Dom}_{\mathrm{max}}}

\newcommand{\oner}{\frac{1}{r}}

\newcommand{\mass}{\operatorname{mass}}

\newcommand{\ti}{\pitchfork} 
\newcommand\TRC[1]{\operatorname{TRC}#1} 
\def\e{\epsilon}
\renewcommand{\div}{\operatorname{div}}

\begin{document}

\title{Witten spinors on nonspin manifolds}
\author{Anda Degeratu\footnote{University of Freiburg, Mathematics Institute; email: anda.degeratu@math.uni-freiburg.de}
~and Mark Stern\footnote{Duke University, Department of Mathematics;  
email: stern@math.duke.edu}}
\date{\today}

\maketitle

\setcounter{section}{0}

\medskip
\begin{abstract}
  Motivated by Witten's spinor proof of the positive mass theorem, we
  analyze asymptotically constant harmonic spinors on complete
  asymptotically flat nonspin manifolds with nonnegative scalar curvature.
\end{abstract}

\medskip
\medskip
\section{Introduction}\label{sec:introduction}
Techniques from spin geometry have proved very powerful in the study
of positive scalar curvature. However, not all manifolds are spin.
The topological obstruction to the existence of a spin structure is
the second Stiefel-Whitney class.  In this article we explore, in the
context of the positive mass theorem, whether it is possible to graft
some of the machinery of spin geometry onto the study of nonspin
manifolds.

The fundamental idea is to excise a representative of the Poincar\'e
dual of the second Stiefel-Whitney class and apply spin techniques on
the complement. We explore this idea by trying to adapt Witten's proof
of the positive mass theorem to the nonspin case.  The difficulties in
executing this approach arise from the fact that this complement is
incomplete, greatly complicating analytic arguments.
Ultimately, we do not succeed in this endeavor, but  we hope that our
analysis of  harmonic spinors on  incomplete spin manifolds  may prove
useful in other contexts.
A prior examination of incomplete spin structures
appeared in~\cite{baldwin}.

\subsection{The positive mass theorem}\label{sec:pmt}
\begin{definition}\label{def:asy_flat}
 A complete non-compact Riemannian manifold $(M^n, g)$ is called
 {\em asymptotically flat} of order $\tau>0$ if there exists a compact
 set, $K \subset M$, whose complement is a disjoint union of
 subsets $M_1, \ldots, M_L$ -- called the {\em ends} of $M$ -- such that
 for each end there exists a diffeomorphism
 $$
   Y_l: \IR^n \setminus B_{T} (0) \to M_l,
 $$
 so that 
 $ Y_l^* g =: g_{ij}dx^idx^j$ satisfies for $\rho = |x|$, 
\begin{equation*}
 g_{ij} = \delta_{ij} + \sO(\rho^{-\tau}),\quad
 \p_k g_{ij} = \sO(\rho^{-\tau-1}), \quad
 \p_k\p_l g_{ij} = \sO(\rho^{-\tau-2}),
\end{equation*}
with $\delta$ the Euclidean metric on $\IR^n$, and $T>1$.
We call such a coordinate chart $(M_l,Y_l)$ {\em asymptotically flat}.
On each asymptotically flat end $(M_{l},Y_l)$ we have the induced
coordinate system $\{x_i\}$ and the corresponding radial coordinate
$\rho (x) = |x|$ obtained by pulling-back the cartesian coordinates on
$\IR^n$ under $Y_l$. We extend $\rho$ smoothly over the interior of
$M$ so that it is bounded from below by $1$. 
\end{definition}

For an end $(M_l,g_l)$ of the asymptotically flat manifold $(M,g)$,
the {\em mass} is defined to be
 \begin{equation}\label{eq:adm_mass}
   \mass(M_l, g): 
    =\frac{1}{c(n)} 
        \lim_{L \to \infty}
        \int_{S^{n-1}_{L}}  \frac{x^j}{\rho}(\partial_i g_{ij} - \partial_j g_{ii}) d\sigma, 
 \end{equation}
if the limit exists. Here $c(n)$ is a normalizing constant depending
only on the dimension of the manifold, and
$S_{L}^{n-1}$ is the sphere of radius $L$ in the asymptotically
flat coordinate chart $(M_l, Y_l)$.
The {\em mass} of the manifold $(M^n, g)$ is the sum of the masses of its ends
 $$
   \mathrm{mass}(M, g): = \sum_{l=1}^L \; m(M_l, g).
 $$
Bartnik showed that if $(M,g)$ satisfies the {\em mass
decay conditions},
\begin{equation}\label{eq:wdm}
 \tau > \frac{n-2}{2}
 \qquad \text{and} \qquad 
 R \in L^1(M),
\end{equation}
then the mass is well-defined  and is a Riemannian
invariant,~\cite{bartnik}. Here $R$ denotes the scalar curvature of
$g$.
\begin{pmt}
 Let $(M,g)$ be an asymptotically flat Riemannian manifold of
 dimension $n\geq 3$ satisfying the mass decay
 conditions~\eqref{eq:wdm} and having nonnegative scalar curvature. Then
 the mass is nonnegative. Furthermore $\mass(M,g) =0$ if and
 only if $(M,g)$ is isometric to the Euclidean space.
\end{pmt}
The positive mass theorem has a long history. Arnowitt, Deser, and
Misner introduced the notion of mass of an asymptotically flat
spacelike hypersurface in space-time and conjectured its positivity
for $3$-dimensional spacelike hypersurfaces. This conjecture was
proved by Schoen and Yau, using minimal surface techniques.  Their
proof extends readily to dimensions $n\leq 7$.  
In fact, it gives the stronger result that the mass of {\em each
  asymptotically flat end} is nonnegative and that it vanishes if and
only if $(M,g)$ is the Euclidean space.
Witten subsequently gave a different proof using spinors, which
yields this strengthened form of the theorem
for all {\em spin} manifolds of arbitrary dimension.  The requisite
analysis was provided by Parker and Taubes~\cite{pt} and
Bartnik~\cite{bartnik}. After this work, the positive mass theorem was
open for higher dimensional {\em nonspin} manifolds.

Recently, Schoen~\cite{SchoenGergen} and Lohkamp~\cite{Lohkamp} have
each announced programs for extending to higher dimension the minimal
surface approach to proving the positive mass theorem.
On the other hand, our analysis of harmonic spinors on nonspin manifolds is motivated by
Witten's proof, which we now briefly recall.
\subsection{Witten's proof of the positive mass theorem}\label{sec:WittenProof}
Let $(M,g)$ be an asymptotically flat Riemannian spin manifold of
dimension $n$ which satisfies the mass decay conditions~\eqref{eq:wdm}
and has nonnegative scalar curvature. For simplicity, we assume here
that $M$ has only one end.
Let $\psi_0$ be a smooth spinor on $M$ which is constant near infinity with
respect to the chosen asymptotically flat coordinate system and normalized by
$|\psi_0|^2 \to 1$ at infinity. Then there exists a unique spinor $u$,
with  $Du\in L^2(M,S)$ and $\frac{u}{\rho}\in L^2(M,S)$ solving
\begin{equation}\label{eq:u}
  D^2 u = - D \psi_0.
\end{equation}
Let 
\begin{equation}\label{eq:psi}
  \psi : = Du + \psi_0.
\end{equation}
Applying the Lichnerowicz formula to the harmonic spinor $\psi$
one computes (see the proof of Proposition~\ref{prop:mass}) that 
\begin{equation}
 \int_M |\nabla \psi|^2 + \frac{R}{4} |\psi|^2 - |D\psi|^2 
 = \frac{c(n)}{4} \mass(M,g).
\end{equation}
Since the spinor $\psi$ is harmonic and the scalar curvature $R$ is
nonnegative and $L^1$, the mass of $(M, g)$ is finite and  nonnegative.
 
\subsection{Our main results}
In this work we study a natural extension of Witten's argument to
nonspin manifolds.  Suppose that $(M,g)$ is a nonspin Riemannian
manifold that is asymptotically flat of order $\tau >0$.  By passing
to an oriented double cover when necessary, it suffices to consider
the case of orientable manifolds. Therefore in this article $M$ is
orientable and nonspin. Moreover, since every orientable $3$-manifold
is spin, it follows that the dimension of the manifolds we are
considering here is $n \geq 4$.
In Theorem~\ref{thm:V} we show that we
can choose a compact subset $V$ in the compact part $K$ of $M$,
stratified by smooth submanifolds $V^{k_b}$ of codimension $k_b =
b(b-1)$ with $b\geq 2$, so that $M \setminus V$ admits a spin
structure, and this spin structure does not extend over the lowest
codimension stratum $V^2$. We fix such a $V$ and such a spin structure on $M\setminus V$. 
This spin structure restricts to the trivial spin structure on each of
the asymptotically flat ends of $M$.  We denote by $S$ the
corresponding spinor bundle on $M\setminus V$.  Similar to the spin
case, we say that a spinor on $M\setminus V$ is {\em constant near
  infinity} if it is constant with respect to each of the chosen
asymptotically flat coordinate systems on the ends of $M$.

To implement Witten's proof, we first need to construct an
asymptotically constant harmonic spinor, which we call a ``Witten
spinor''. The properties of this spinor are exactly those properties
of the spinor $\psi$ constructed above in Witten's proof on a
complete spin manifold.
\begin{definition}\label{def:ws}
  Let $\psi_0$ be a spinor on $M\setminus V$, constant near infinity
  on each of the asymptotically flat ends of $M$.
  We say that a spinor $\psi$ on $M\setminus V$ is a {\em Witten
    spinor} asymptotic to $\psi_0$, if the following conditions are satisfied :
 \begin{enumerate}
 \item $\frac{\psi-\psi_0}{\rho}\in L^2(M\setminus V,S)$,
  \item $\psi$ is strongly harmonic, i.e. $D\psi =0$, and
 \item $\nabla (\psi-\psi_0) \in L^2(M_l, S\lvert_{M_l})$ for each
 asymptotically flat end $M_l$ of $M$.
 \end{enumerate}  
 In particular, we say that the Witten spinor is {\em associated to the end
 $M_l$} if it is asymptotically nonvanishing in $M_l$ and asymptotically vanishing for all ends $M_i$
 with $i \neq l$.
\end{definition}
We show that Witten spinors exist on $M\setminus V$. 
\begin{theoremA}
  Let $(M,g)$ be a nonspin Riemannian manifold which is asymptotically
  flat of order $\tau >\frac{n-2}{2}$ and which has nonnegative scalar
  curvature.  Given a smooth spinor $\psi_0$ on $M\setminus V$ that is
  constant near infinity on each of the asymptotically flat ends of
  $M$ and vanishes in a neighborhood of $V$, there exists a Witten
  spinor on $M\setminus V$ asymptotic to $\psi_0$.
\end{theoremA}

Next one needs to use the integral form of
the Lichnerowicz formula on the {\em incomplete }manifold $M\setminus
V$. In this case, integration by parts may introduce unwanted boundary
terms from the ideal boundary $V$ into the formula. This is a problem
familiar in the study of the Hodge theory of $L^2$-cohomology of
singular varieties (see~\cite{ps2}), which is resolved in that context
by proving that, as they approach the singularities, harmonic forms
decay sufficiently rapidly to introduce no extra boundary terms when
integrating by parts.

To analyze the behavior of the Witten spinors near $V$, we study the growth of $\psi$ near each of the strata $V^{k_b}$ of $V$ separately. Unless
$V^{k_b}$ is a closed stratum, there are no tubular neighborhoods of uniform radius over
the entire stratum. Hence, for uniform estimates involving separation of variables,  we formulate our estimates in tubular
neighborhoods over relatively compact subsets of $V^{k_b}$ that do
not intersect the higher codimension strata. We denote by
$\TRC(V^{k_b})$ the set of all these good neighborhoods around points
in $V^{k_b}$. Letting $r$ denote the distance to $V^2$ the lowest
codimension stratum, and $r_b$ denote the distance to the higher
codimension strata $V^{k_b}$ of $V$, we have:
\begin{theoremB}
 Let $\psi$ be a Witten spinor constructed as in Theorem A. Then
 \begin{enumerate}
  \item for all $W\in\TRC(V^2)$
     \begin{equation}\label{eq:r2}
        \frac{\psi}{r^{1/2}\ln^{1/2+a}(\frac{1}{r})} \in L^2
        (W\setminus V, S\lvert_{W\setminus V}),
             \quad \text{for all $a>0$},
      \end{equation}
  \item for all $W\in \TRC(V^{k_b})$ with $k_b > 2$
      \begin{equation}\label{eq:rb}
        \frac{\psi}{r_b^{(k_b-2)/2}\ln^{1/2+a}(\frac{1}{r_b})} \in L^2
        (W\setminus V, S\lvert_{W\setminus V}),
             \quad \text{for all $a>0$}.
      \end{equation}
 \end{enumerate}
\end{theoremB}
However, the decay estimates in~\eqref{eq:r2} are borderline for our
purposes. For any class of manifolds for which we could set $a = 0$
in~\eqref{eq:r2}, Witten's proof of the positive mass theorem extends.
\begin{theoremC}
  Let $(M,g)$ be an asymptotically flat nonspin manifold which
  satisfies the hypothesis of the Positive Mass Theorem.  Let $M_l$ be
  an asymptotically flat end of $M$, and let $\psi_{0l}$ be a spinor 
  constant at infinity which is supported on $M_l$ and is asymptotically nonvanishing.
Let $\psi$ be the Witten spinor associated to $M_l$ and
  asymptotic to $\psi_{0l}$ constructed in Theorem A. If this spinor
  satisfies
  \begin{equation}\label{eq:wish}
   \frac{\psi}{r^{1/2}\ln^{1/2}(\frac{1}{r})} \in L^2(W\setminus V, S\lvert_{W\setminus V})
 \end{equation}
 for all $W \in \TRC(V^2)$, then the mass of the end $M_l$ is
 positive. 
 Moreover, if there exists a Witten spinor whose norm is asymptotic to
 $1$ on every end and which satisfies (\ref{eq:wish}), then the mass
 of $(M,g)$ is positive.
\end{theoremC}

To construct the Witten spinor in Theorem A, we proceed similarly to
Witten's proof: we find $u$ in the minimal domain of the Dirac
operator solving $D^2 u = - D\psi_0$, and set $\psi = \psi_0 + Du$.
Since the spin structure on $M\setminus V$ does not extend over $V^2$,
spinors have nontrivial holonomy around small circles normal to $V^2$.
The $L^2$-harmonic spinors near $V^2$ have a Fourier decomposition in
these normal circles whose leading order modes in polar coordinates
may behave like $r^{-1/2}e^{\pm i\theta/2}.$ Such modes prevent direct
application of the Lichnerowicz formula.  If a spinor satisfies the
hypotheses~\eqref{eq:wish} of Theorem C, these modes vanish, giving
that its product with any element of $C_0^\infty(M)$ is in the minimal
domain of the Dirac operator. However, the decay obtained in Theorem B
near $V^2$ is not sufficient to remove them. In fact in our
construction of the Witten spinor $\psi$, the spinor $u$ satisfies the
decay conditions in~\eqref{eq:wish}, but $Du$ and $\psi$ need
not. Therefore we cannot conclude the positive mass theorem this way.

On the other hand, if we knew that there are no nontrivial strongly
harmonic $L^2$-spinors, then it would be possible to prove the
positive mass theorem in the following manner: Find $w$ sufficiently
regular near $V$ solving $D^2 w = - D^2\psi_0$ and so that $Dw \in
L^2(M\setminus V, S)$. Then the spinor $\Phi: = \psi_0 + w$ is weakly
harmonic (by which we mean $D^2 \Phi =0$), while $D\Phi$ is strongly
harmonic and in $L^2(M\setminus V, S)$.  If zero was the only strongly
harmonic $L^2$-spinor, it would follow that $D\Phi =0$. Then $\Phi$
would be a Witten spinor satisfying the estimate~\eqref{eq:wish}, and
the positive mass theorem would follow (see Remark~\ref{adumbration}
for more details).  However, nonzero strongly harmonic $L^2$-spinors
not only exist but, in fact, form an infinite-dimensional space.

\begin{theoremD}
  Let $(M,g)$ be an asymptotically flat nonspin manifold which
  satisfies the hypothesis of the Positive Mass Theorem. Assume $V^2
  \neq \emptyset$. Then, the space of strongly harmonic square
  integrable spinors is infinite dimensional. In particular, Witten
  spinors are not unique.
\end{theoremD}

\subsection{Plan of the Paper}

It is well-known that the obstruction to having a spin structure on an
orientable manifold $M$ is the second Stiefel-Whitney class $w_2(M)
\in H^2(M, \IZ_2)$,~\cite{lm}.

In Section~\ref{sec:incomplete_spin}, starting from the fact that the
second Stiefel-Whitney class vanishes on any bundle admitting a rank
$(n-1)$ trivial subbundle, we construct the stratified space $V$ whose
complement is an incomplete spin manifold.  We then analyze the
geometric structure of $V$ and introduce the set $\TRC(V^{k_b})$ of
good tubular neighborhoods around points in each stratum $V^{k_b}$ of
$V$.

In Section~\ref{sec:pe} we gather Hardy inequalities for spinors on
the asymptotically flat ends of $M$ and near each stratum of $V$. Near
the codimension $2$ stratum of $V$ we have a stronger {\em angular
  estimate}. This is a consequence of the fact that the spin structure
on $M\setminus V$ does not extend over $V^2$ and the resulting
nontrivial holonomy around small normal circles to $V^2$.

We follow with a preliminary study of the Dirac operator $D$ on
$M\setminus V$ in Section~\ref{sec:Do}. We introduce the weighted
Sobolev spaces necessary for our analysis, state the Lichnerowicz
formula on these spaces, and derive a vanishing result in
Proposition~\ref{prop:boch1}. We also introduce the maximal and the
minimal domains of $D$.
We continue with an analysis of the growth condition of the Witten
spinors on the asymptotically flat ends, and conclude this section
with Proposition~\ref{prop:mass}, which recapitulates Bartnik's proof
that the asymptotic boundary contribution to the Lichnerowicz formula
for strongly harmonic spinors does yield the product of a universal
constant and the mass.

In Section~\ref{sec:lbe} we derive estimates near each stratum of
$V$ for spinors $u$ and $v:= Du$, with $u$ in the minimal domain of $D$. 
In Lemma~\ref{lem:conv}, we also derive a sufficient condition for a
spinor to be in the minimal domain of $D$.
In Section~\ref{sec:iie} we sharpen these estimates when
$D^2 u =0$ near $V$. Both these sections rely on the estimates established in
Section~\ref{sec:pe} and a technique of Agmon which we use
repeatedly,~\cite{ag}. 

In Section~\ref{sec:af} we derive estimates for spinors on the
asymptotically flat ends of $M$. 

In Section~\ref{sec:coercive} we prove a coercivity result for the
Dirac operator, which we then use in Section~\ref{sec:main} together
with Corollary~\ref{cor:scws} to prove Theorem A. We then follow with
the proofs of Theorem B and C. We conclude this section with the 
proof of Theorem D.   

\paragraph*{Acknowledgements:}

We would like to thank Hubert Bray for helpful discussions, and 
the referees for giving detailed
constructive suggestions for improving the paper.  The first author
would like to thank Tom Mrowka and Richard Melrose for support and
useful discussions, and also acknowledge the support of NSF
DMS-0505767 and the Max Planck Institute for Gravitational Physics in
Potsdam, Germany during the time this work was carried out.  The work
of the second author was supported partially by NSF grant DMS-1005761.

\tableofcontents

\section{Incomplete Spin Structures}\label{sec:incomplete_spin}
Let $M$ be an oriented $n$-dimensional Riemannian manifold which is
nonspin. Since an oriented manifold is spin if and only if its second
Stiefel-Whitney class vanishes, it follows that $w_2(M) \in
H^2(M,\IZ_2)$ is non-vanishing.  Based on the definition of $w_2(M)$,
we construct a compact stratified subset $V \subset M$ so that $M
\setminus V$ admits a spin structure.
We start by recalling generalities about stratified sets, after which
we construct $V$, describe its properties and those of the spin
structure on $M\setminus V$.

\subsection{Generalities about stratified sets}\label{subsec:str}

Let $M$ be a smooth manifold and let $S\subset M$ be a closed
subset. A {\em stratification} $\Sigma$ of $S$ is a decomposition of
$S$ into disjoint smooth connected submanifolds of $M$, called the
{\em strata} of $M$. The stratification is {\em locally finite} if
each point has a neighborhood intersecting only a finite number of
strata.

Given another manifold $N$, a smooth map $f:N \to M$ is {\em
  transverse to $\Sigma$}, which we write $f \ti \Sigma$, if $f$ is
transverse to each of the strata of $\Sigma$. We need a stratified
version of Thom's transversality theorem,~\cite{hirsch}. A sufficient
condition for this result to apply in the stratified context is that
the stratification $\Sigma$ satisfies {\em Whitney (a)-regularity},
that is, for any pair of strata $X, Y$ of $\Sigma$ with $Y \subset
\bar X \setminus X$, and given a sequence of points $\{x_i \} \subset
X$ such that the $x_i$ converge to $y \in Y $ and the tangent spaces
$T_{x_i} X$ converge to $\tau\subset T_yM$, we have $T_yY \subset
\tau$.

\begin{theorem}[Thom transversality in the stratified setting]\label{thm:Thom}
  Let $M, N$ be smooth manifolds. Let $S$ be a closed subset of $M$
  and let $\Sigma$ be a locally finite stratification of $S$ which
  satisfies Whitney's (a)-regularity condition. Then the set $\{f\in
  \sC^{\infty}(N, M) \mid f \ti \Sigma\}$ is open and dense in
  $\sC^{\infty}(N,M)$.
\end{theorem}

This theorem was proved by Feldman~\cite[Proposition
3.6]{feldman1965}. The necessity of the Whitney condition was proved
by Trotman~\cite{trotman1979}. Once we have this transversality result
in the stratified setting, a consequence of the inverse image theorem
gives:

\begin{corollary}\label{cor:tr-inv}
  Let $M, N$ be smooth manifolds. Let $S$ be a closed subset of $M$
  and let $\Sigma$ be a locally finite stratification of $S$. If $f
  \in \sC^{\infty}(N,M)$ is transverse to $\Sigma$, then $f^{-1}(S)$
  is a locally finite stratified subset of $N$ with stratification
  given by the preimages of the strata of $\Sigma$. Moreover, if $X$
  is a strata of $\Sigma$, the codimension of $f^{-1}(X)$ in $N$ is
  the same as the codimension of $X$ in $M$.
\end{corollary}

\subsection{The construction of $V$}\label{subsec:construct_V}
The second Stiefel-Whitney class $w_2(M)$ of a manifold $M$ is the
topological obstruction to extending $(n-1)$ linearly independent
vector fields from the $1$-skeleton to the $2$-skeleton of
$M$,~\cite[Section 39]{steenrod}.  In particular, it vanishes for any
manifold admitting $(n-1)$ pointwise linearly independent vector
fields.  With this in mind, we consider the vector bundle $\Hom
(\IR^{n-1}, TM)$. The fiber at each point $x$ is identified with the
space of $(n-1)$-tuples of vectors in $T_xM$. In this fiber we take
the subset $\sH_x$ of maps which are not of maximal rank.  Set $\sH
:=\cup_{x\in M}\sH_x$.
\begin{lemma}\label{lem:str-H}
  The subset $\sH \subset \Hom(\IR^{n-1}, TM)$ is a locally finite
  stratified set,
  \begin{equation}\label{eq:str-H}
    \sH = \union_{b\geq 2} \sH^{k_b},
  \end{equation}
  with stratum $\sH^{k_b}$ consisting of all maps of rank $n-b$.
  Moreover, the stratum $\sH^{k_b}$ has codimension $k_b:=b(b-1)$ in
  $\Hom(\IR^{n-1}, TM)$. 
\end{lemma}
\begin{proof}
  We need to show that each $\sH^{k_b}$ is a smooth submanifold of
  $\Hom(\IR^{n-1}, TM)$ of codimension $k_b = b(b-1)$. In fact, it is
  enough to show this claim fiberwise.  This is well-known
  (see~\cite[Lecture 12]{mrowka}, for example), but we include it here
  for completeness.

  Consider $\Hom(\IR^{n-1},\IR^n)$ and let $H^{k_b}\subset
  \Hom(\IR^{n-1}, \IR^n)$ be the subspace of linear maps of rank
  $n-b$. For $L\in H^{k_b}$, $\dim \Ker L = b-1$ and $\dim \Coker L =
  b$. Throughout this section, we identify $\Coker L $ with $(\Im
  L)^{\perp}$ and write
 \[
   \IR^{n-1} = (\Ker L)^{\perp} \oplus \Ker L,
   \qquad 
   \IR^n = \Im L \oplus  \Coker L. 
 \]
 (Without loss of generality, we can equip $\IR^{n-1}$ and $\IR^n$
 with inner products.) With respect to this decomposition, $L$ has the
 form $L = \left[\begin{smallmatrix} \bar L & 0 \\ 0 &
     0\end{smallmatrix}\right]$.

 Let $A \in \Hom(\IR^{n-1}, \IR^n)$ be of the form $A =
 \left[\begin{smallmatrix} e & f \\ g & h\end{smallmatrix}\right]$
 with respect to the above decomposition.  We claim that for $A$
 sufficiently small, $L+A\in H^{k_b}$ if and only if
 \begin{equation}\label{eq:leq}
    h - g (\bar L +e)^{-1} f =0.
 \end{equation}
 To see this, row reduce $\left[\begin{smallmatrix}\bar L+ e & f \\ g
     & h\end{smallmatrix}\right]$ to $\left[\begin{smallmatrix}\bar L+
     e & f \\ 0 & h-g(\bar L+ e)^{-1}f\end{smallmatrix}\right]$. This
 has the desired rank if and only if~\eqref{eq:leq} holds.  From
 here, it is clear that $H^{k_b}$ is smooth near $L$.  Moreover,
 equation~\eqref{eq:leq} shows that the tangent space to the stratum
 $H^{k_b}$ at $L$ is
 \begin{equation}\label{eq:ts}
   T_{L} H^{k_b} = \{ W \in \Hom (\IR^{n-1}, \IR^n) \mid
   \pi_{\Coker L} \circ W \circ \iota_{\Ker L} = 0\},
 \end{equation}
 with $\pi_{\Coker L}: \IR^n \to \Coker L$ the projection onto $\Coker
 L$ and $\iota_{\Ker L}: \Ker L \to \IR^{n-1}$ the inclusion. Hence,
 the normal space to the stratum $H^{k_b}$ is
 \[
   N_L(H^{k_b}) = \Hom (\Ker L, \Coker L),
 \]
 and this has codimension $k_b = b (b-1)$.
\end{proof}

We next show that the stratification~\eqref{eq:str-H} is Whitney
(a)-regular. A general result of Whitney~\cite[Theorem~19.2]{Wh} on
the stratification of algebraic varieties gives the weaker outcome
that our stratification has a Whitney (a)-regular refinement.

\begin{lemma}
 The stratification $\sH = \union_{b \geq 2} \sH^{k_b}$ is Whitney (a)-regular.
\end{lemma} 

\begin{proof}
  It is enough to show that the stratification $\{H^{k_b}\}_{b\geq 2}$
  of the set of linear maps in $\Hom(\IR^{n-1},\IR^n)$ which are not
  of maximal rank is Whitney (a)-regular.

  Let $\Theta \in H^{k_{b+1}}$. With respect to the decompositions
  $\IR^{n-1} = (\Ker \Theta )^{\perp} \oplus \Ker \Theta$ and $\IR^n =
  \Im \Theta \oplus \Coker \Theta$, we have $\Theta =
  \left[\begin{smallmatrix} \bar \Theta & 0 \\ 0 &
      0\end{smallmatrix}\right]$. Let $\sU$ be a neighborhood of
  $\Theta$ in $\Hom(\IR^{n-1}, \IR^n)$ in which each point $L$ has the
  form
 \begin{equation}\label{eq:L}
 L= \left[\begin{smallmatrix} 
            \bar \Theta + e(L) & f(L) \\
             g(L) & h(L)
           \end{smallmatrix}\right] 
   = \left[\begin{smallmatrix} 
             I & 0 \\
             g(L)\left(\bar \Theta + e(L)\right)^{-1} & I
           \end{smallmatrix}\right]
     \left[\begin{smallmatrix} 
             \bar \Theta + e(L) & f(L) \\ 
             0 &  \Upsilon(L)\
           \end{smallmatrix}\right]
 \end{equation}
 with \begin{equation}\label{eq:U} \Upsilon(L): = h(L) - g(L) \left(\bar
     \Theta + e(L)\right)^{-1}f(L)\end{equation} and $|g(L)|$,
 $|e(L)|$, $|f(L)|$ smaller than some constant $\delta \ll \max\{1,
 |(\bar \Theta)^{-1}|\}.$ Then $L\in \sU \cap H^{k_b}$ if and only if $
  \Upsilon(L)$ has rank $1$. 

  Let $H^{k_{b}}_{\Theta}$ denote the space of maps of rank $1 =
 (b+1) - b$ in $\Hom (\Ker \Theta, \Coker\Theta)$. With
  this,   $\sU \cap H^{k_b} = \sU \cap \Upsilon^{-1}(H^{k_b}_{\Theta})$.
  It is easy to see that the map $\Upsilon: \sU \to \Hom(\Ker \Theta,
  \Coker \Theta)$ is a submersion. Therefore a vector 
  $W\in T_L\Hom(\IR^{n-1}, \IR^n)$ is tangent to $\sU\cap
  H^{k_b}$ if and only if $d\;\!\Upsilon_L(W)$ is tangent to
  $H^{k_{b}}_{\Theta}$. By (\ref{eq:ts}) applied to $\Hom (\Ker
  \Theta, \Coker\Theta)$, we have $d\;\!\Upsilon_L(W)$ is tangent to
  $H^{k_{b}}_{\Theta}$ if and only if $\pi_{\Coker \Upsilon(L)} \circ
  d\;\!\Upsilon_L(W) \circ \iota_{\Ker \Upsilon(L)} = 0$. Hence, for $L \in \sU \cap
  H^{k_b}$
 \begin{equation}\label{eq:W} 
   T_LH^{k_b} = \{W \in \Hom(\IR^{n-1}, \IR^n) 
   \mid  \pi_{\Coker \Upsilon(L)} \circ d\;\!\Upsilon_L(W) \circ \iota_{\Ker \Upsilon(L)} = 0\}.
\end{equation}  

 Let $\Lambda \in T_{\Theta} H^{k_{b+1}}$. By~\eqref{eq:ts}, $\Lambda \in
 \Hom(\IR^{n-1}, \IR^n)$ with $\pi_{\Coker \Theta} \circ \Lambda \circ
 \iota_{\Ker \Theta} =0$.  We want to extend $\Lambda$ to a continuous
 vector field in $\sU\cap H^{k_b}$.    For this, we simply take 
\begin{equation}\label{eq:WL}
   \sW_L: = \Lambda - d\;\!\Upsilon_L (\Lambda).
\end{equation}
Since $d\;\!\Upsilon_L (\Lambda)\in \Hom(\Ker \Theta, \Coker \Theta) $,
formula~\eqref{eq:U} gives $d\;\!\Upsilon_L (d\;\!\Upsilon_L(\Lambda)) = dh_{L}
(d\;\!\Upsilon_L(\Lambda))$.  Since $h(L)$ can be written as $\pi_{\Coker
  \Theta}\circ L\circ \iota_{\Ker \Theta} $, we have $dh_{L}(Y) =
\pi_{\Coker \Theta}\circ Y\circ \iota_{\Ker \Theta}$ for all $Y \in
\Hom(\IR^{n-1}, \IR^n)$. Thus $dh_{L}$ acts as the identity
on $\Hom(\Ker \Theta, \Coker \Theta) $, and hence $d\;\!\Upsilon_L(\sW_L) =
0$. Using~\eqref{eq:W} we conclude $\sW_L \in T_{L}H^{k_b}$.
Moreover, $\sW_L \to \Lambda$ as $L \to \Theta$, showing that the
stratification is Whitney (a)-regular.
\end{proof}

Since our stratification~\eqref{eq:str-H} satisfies the Whitney
(a)-regularity condition, it follows that Thom's transversality
Theorem holds (see Theorem~\ref{thm:Thom}) and thus any section of the
bundle $\Hom(\IR^{n-1}, TM)$ can be perturbed to be transverse to the
stratified subset $\sH$.  We choose such a section $s$ of
$\Hom(\IR^{n-1},TM)$, transverse to $\sH$; it corresponds to an
$(n-1)$-tuple of vector fields on $M$.  Let $\Sigma = \Sigma (s) =
s^{-1}(\sH)$ be the set where these vector fields fail to be linearly
independent. Since $s$ was transverse to $\sH$,
Corollary~\ref{cor:tr-inv} gives that this is a stratified space, with
strata $\Sigma^{k_b}$ (possibly empty) only in codimension $k_b$ :
\begin{equation*}\label{eq:initial_strata}
  \Sigma = \, \union_{b\geq 2} \Sigma^{k_b} \, 
         = \,  \Sigma^2 \cup \Sigma^6 \cup \Sigma^{12} \cup \ldots.
\end{equation*}
By construction, $M \setminus \Sigma$ admits a spin structure. Let
$P_{\Spin} (M \setminus \Sigma)$ be the associated lifting of the
bundle of orthonormal frames, $P_{\SO}(M\setminus \Sigma)$. However, it might be
possible to extend this spin structure over certain connected
components of $\Sigma^2$. In the next Lemma we do exactly this: we
extend the spin structure as much as possible over $\Sigma^2$.
\begin{lemma}\label{lem:set_V}
 There exists $V \subset \Sigma$, a closed stratified space
 \begin{equation}\label{eq:strata}
   V = \, \union_{b \geq 2} V^{k_b} \,
     = \, V^2 \cup V^{6} \cup V^{12} \cup \ldots
 \end{equation}
such that $M\setminus V$ is spin, but the spin structure 
 $P_{\Spin} (M \setminus \Sigma)$ cannot be extended over 
 $\Sigma^2 \setminus V^{2}$. 
\end{lemma}
\begin{proof}
  Let $x \in \Sigma^2$. The holonomy of $P_{\Spin} (M \setminus
  \Sigma)$ on infinitesimally small loops in the transverse slice to
  $\Sigma$ at $x$ is either $+1$ or $-1$. If the holonomy is $+1$ then
  it is so for the entire connected component of $\Sigma^2$ in which
  $x$ lies. The spin structure extends over this connected component.

 We take $V^2$ to consist of those connected components of $\Sigma^2$
 around which the holonomy of an infinitesimal loop is $-1$; and then
 define $V^{k_b}: = \Sigma^{k_b}$ for the higher codimension strata.
 The fact that $V$ is closed follows since $\Sigma$ and $\sH$ are
 closed, as locally $\sH$ is given by the vanishing of the
 determinants of $(n-1) \times (n-1)$ minors of a $(n-1)\times n$
 matrix.
\end{proof}
This Lemma allows us to choose a spin structure over $M \setminus
V$ with the property that it does not extend over any component of
$V^2$. We call such a spin structure {\em maximal}. 

Moreover, since $M$, as an asymptotically flat manifold, has a natural
spin structure on the asymptotically flat end, the set $V$ can be
chosen to lie in the compact part of $M$.
\begin{lemma}
 Let $M$ be an asymptotically flat manifold which is not spin. Then 
 the stratified space $V$ can be chosen so that it lies in the compact 
 part $K$ of $M$. The resulting spin bundle is trivial on each end.
\end{lemma}
\begin{proof}
 We show that we can choose a generic section $s$ of the bundle 
 $\Hom (\IR^{n-1}, TM)$ so that $\Sigma (s) \subset K$ and closed.
 Choose asymptotically flat coordinates $x$ on each end. Choose a
 section $\bar s := \langle \frac{\p}{\p x^1},\ldots,\frac{\p}{\p
   x^{n-1}}\rangle$ in this chart at infinity, and then extend $\bar
 s$ to a section $s$ of $\Hom (\IR^{n-1}, TM)$, transverse to $\sH$.
 By construction, $\Sigma (s) \subset K$.

The distinct spin structures on each end are parametrized by
$H^1(\IR^n\setminus B_T(0),\IZ_2) = 0$, for $n>2$. Hence the spin
bundle must agree with the trivial spin bundle on each end.
\end{proof}
\subsection{A remark about higher codimension}\label{sec:hcodim}
We saw in Lemma~\ref{lem:set_V} that the spin structure can be
extended over those connected components of $\Sigma^2$ where the
holonomy of $P_{\spin}(M\setminus \Sigma)$ is $+1$. We show now that
the spin structure can be extended over any of the higher codimension
components of $\Sigma^{k_b}$ with $b>2$ which is a smooth closed
submanifold.  A similar result was obtained in the PhD thesis of
Baldwin~\cite[Section 2.1.2]{baldwin}. 

Presumably, it should also be possible to extend the spin structure
over any of the connected components of $V$ that do not intersect
$V^2$. In particular, if $V^2=\emptyset$ then we expect that $M$ is
spin.  If this is not, in fact the case, then our construction of
Witten spinors and our attendant regularity results show that the
positive mass theorem holds for such  nonspin manifolds with 'small'
singular set $V$.
\begin{proposition}
  Let $X$ be an $n$-dimensional smooth orientable manifold. Let $Y$
  be a smooth connected closed submanifold of codimension $k$. Assume that
  $X\setminus Y$ is spin.  If $k\geq 3$, then $X$ is spin.
\end{proposition}
\begin{proof}
 Choose a Riemannian metric on $X$, and let $  p: D(Y)\to Y$ be the  normal  disk bundle            
 to $Y$ and $S(Y)$  the corresponding sphere bundle. We identify
 $D(Y)$ with a tubular neighborhood of $Y$ with boundary $S(Y)$. 
 Let $w_{k}(N)$ denote the $k$th Stiefel-Whitney class of the normal
 bundle $NY$ to $Y$ in $X$. Consider the Gysin sequence with mod 2
 coefficients for the unoriented sphere bundle $S(Y)$ (see~\cite[p.144]{MS}):
 \begin{equation}\label{eq:gysin}
  \xymatrix{
    \cdots \ar[r]& H^{i-k}(Y,\IZ_2)
           \ar[r]^-{\cup w_k(N)}& H^{i}(Y,\IZ_2)
           \ar[r]^-{p^*}& H^i(S(Y), \IZ_2)
           \ar[r]& H^{i+1-k}(Y, \IZ_2) 
           \ar[r]& \ldots .
  }
 \end{equation}
 Since $k\geq 3$, we see that for $i=2$, the map $p^*:H^{2}(Y,\IZ_2) \rightarrow H^{2}(S(Y),\IZ_2)$ is an injection, while for $i=1$,
 $p^*:H^{1}(Y,\IZ_2) \rightarrow H^{1}(S(Y),\IZ_2)$ is an isomorphism.

 Next, we consider the Mayer-Vietoris sequence
\begin{equation*}\label{eq:mv}
  \xymatrix{
   \cdots \ar[r]& H^1(X,\IZ_2) 
          \ar[r]& H^{1}(D(Y),\IZ_2)\oplus H^{1}(X\setminus Y,\IZ_2)
          \ar[r]& H^{1}(S(Y),\IZ_2) \ar[r]^-{\delta}&\\
          \ar[r]^-{\delta}& H^2(X,\IZ_2)
          \ar[r]& H^{2}(D(Y),\IZ_2)\oplus H^{2}(X\setminus Y,\IZ_2)
          \ar[r]& H^{2}(S(Y),\IZ_2) 
          \ar[r]& \ldots
 }
\end{equation*}
In this sequence, upon identifying $H^{i}(D(Y), \IZ_2) \equiv
H^i(Y,\IZ_2)$, the map $H^i(D(Y),\IZ_2) \to H^i(S(Y), \IZ_2)$ is the
same as the map $p^*$ in the Gysin
sequence~\eqref{eq:gysin}. Hence the map $H^{1}(D(Y),\IZ_2)
\rightarrow H^{1}(S(Y),\IZ_2)$ is surjective, and the map $
H^{2}(D(Y),\IZ_2)\rightarrow H^{2}(S(Y),\IZ_2)$ is injective. This
also implies that the boundary map $\delta$ is identically zero.

By the naturality of the Stiefel-Whitney classes, the image of $w_2(X)$ in $H^2(X\setminus Y, \IZ_2)$ is
$w_2(TX\lvert_{X\setminus Y})$. This vanishes since $X\setminus Y$ is
spin.  On the other hand, the image of $w_2(X)$ in $H^2(D(Y), \IZ_2)$
is zero by the aforementioned injectivity and the fact that
$TX\lvert_{S(Y)}$ is spin, as $S(Y)$ is the boundary of the spin
manifold $X\setminus D(Y)$. Hence $w_2(X) =0$, and thus $X$ is spin.
\end{proof}
\begin{remark}
  Note that in codimension $k =2$ this proof breaks at two stages:
  First of all, for $i=1$ in the Gysin sequence~\eqref{eq:gysin}, the
  map $p^*:H^1(Y,\IZ_2) \to H^1(S(Y),\IZ_2)$ is not necessarily
  surjective, and thus we cannot conclude that $\delta =0$ in the
  Mayer-Vietoris sequence. Secondly, since for $i=2$ the map $p^*:
  H^2(Y,\IZ_2) \to H^2(S(Y), \IZ_2)$ is not injective, we
  cannot conclude that the image of $w_2(X)$ in $H^2(D(Y), \IZ_2)$
  vanishes.
\end{remark}

\subsection{The conical structure of the singularities of $V$}
The singular structure of the set $V$ defined in Lemma~\ref{lem:set_V}
is easily deduced from the geometry of the subset $\sH$ of
$\Hom(\IR^{n-1}, TM)$. Each stratum $\sH^{k_b}$ lies in the closure of
each of the higher dimensional strata (the strata $\sH^{k_a}$ with $a
< b$). For $(x, T) \in \sH^{k_b}$, the normal bundle to this stratum
within $\Hom (\IR^{n-1}, TM)$ can be identified with the stratified
space of maps in $\Hom (\Ker T, \Coker T)$. The elements which are not
of maximal rank can be identified with a subcone of the normal
bundle. Elements in this subcone exponentiate to $\sH$. A choice of
coordinates in a neighborhood $U$ of $x$ allows us to locally
trivialize these structures and to identify, via the exponential map,
a neighborhood of $(x,T)$ in $\sH$ with the product of a neighborhood
of $(x,T)$ in $\sH^{k_b}$ and a small cone of non-maximal rank elements
in $\Hom (\IR^{b-1}, \IR^{b})$.  The cone can be realized as a cone
over a subvariety of non-maximal rank elements in a small sphere in
$\Hom (\IR^{b-1}, \IR^{b})$. This subvariety is again a stratified
space. Therefore, the point $(x, T)$ has a neighborhood which is a
product of a manifold and a cone over the stratified space $S^{k_b -1}
\cap \sH$. Inducting both on the dimension of the manifold and the
dimension of the strata, we see that $\sH$ is  locally 
quasi-isometric to an iterated cone. This is the familiar cone over
cone topological structure of singularities of projective varieties
arising here as a geometric structure.
This geometric cone structure is preserved under pull-back by
transversal maps, and thus inherited by $V$.

We collect the discussion so far in the following theorem.
\begin{theorem}\label{thm:V}
  Let $(M,g)$ be an oriented asymptotically flat Riemannian manifold which is
  nonspin. Then there exists a closed stratified subset $V$,
   locally  quasi-isometric to an iterated cone and lying in
  the compact part $K$ of $M$, so that the spin structure on
  $M\setminus V$ is maximal, in the sense that it does not extend over
  any of the codimension $2$ strata of $V$.  The strata of $V$,
 $$
  V = V^{k_2} \cup  V^{k_3} \cup  \ldots V^{k_{d-1}} \cup V^{k_d},
 $$
 have codimensions $k_b= b(b-1)$ in $M$. Moreover, the maximal spin
 structure on $M\setminus V$ is trivial on each asymptotically flat
 end of $M$.
\end{theorem}
For the remainder of the paper we fix such a $V$ and the associated
maximal spin structure. We denote by $S$ the corresponding spinor
bundle.

\subsection{The geometric structure near $V$}\label{ssec:gV}
To simplify our analysis near $V,$ we introduce a special
set of tubular neighborhoods adapted to the geometry of the stratified set $V$.
 \begin{remark}
   Here we assume that all $V^{k_b}$ are nonempty, $2\leq b\leq d$.
   This is the general case for $M$ and the choice of transversal
   section $s \in \Hom (\IR^{n-1}, TM)$ to $\sH$. However, there could
   be choices of $M$ and $s$ for which the strata of $\sH$ might give
   empty strata of $V$. This affects the following discussion only
   notationally.
 \end{remark}
We employ the convention that all tubular neighborhoods are geodesic,
have constant radius, and do not intersect the higher codimension
strata of $V$.  Unless $V^{k_b}$ is a closed stratum, there
are no tubular neighborhoods over the whole stratum. Hence, we must
restrict attention to tubular neighborhoods over relatively compact
subsets $Y$ of $V^{k_b}$. The larger we take $Y$, the smaller we
must take the radius of the tube.

Consider $V^{k_b}$, one of the strata of $V$.
Let $T^{k_b}$ be a rotationally symmetric neighborhood
of the zero section of the normal bundle $N^{k_b}$ of $V^{k_b}$
in $M$ on which the exponential map is a diffeomorphism. 
Let $\sW^{k_b}: = \exp(T^{k_b})$. On $\sW^{k_b}$ we define a normal distance function 
$r_b$, by setting
\begin{equation}\label{eq:rb-0}
  r_b(\exp_x(v)):= |v|
\end{equation}
for each $x \in V^{k_b}$ and  $v\in N^{k_b}_x\cap T^{k_b}$.

Given any relatively compact subset $ Y\subset V^{k_b}$, there exists
$0<\e(Y)<\frac{1}{2}$ so that the tubular neighborhood of any
radius $\e<\e(Y)$ over $Y$ is contained in $\sW^{k_b}$
and does not intersect $\overline{V^{k_{b+1}}}$. Denote this
neighborhood $B_{\epsilon}(Y)$.  

Set
\begin{equation}\label{eq:TRC}
\TRC(V^{k_b}): = \{B_{\e}(Y) \subset \sW^{k_b} \mid Y \text{ is a
  relatively compact subset of }V^{k_b} \text{ and }\e<\e(Y) \}.
\end{equation}
Note that we have
$W \cap \overline{V^{k_{b+1}}} = \emptyset$ for all $W\in \TRC(V^{k_b})$.
\begin{remark}\label{rem:rb}
On each $W\in \TRC(V^{k_b})$ we have a well-defined
normal distance function $r_b$. Our choice $\epsilon(Y)<\frac{1}{2}$ implies that $r_b$ is always less than $\frac{1}{2}$.
\end{remark}
We now study the metric on these tubular
neighborhoods. The discussion is similar to the discussion
in~\cite[Section 2]{gray}. Consider $W = B_{\e}(Y)$, with $(Y,\phi)$ a
coordinate neighborhood in $V^{k_b}$ with coordinate functions $\phi (y) =
(y^1,\ldots y^{n-k_b})$. Also choose a trivialization of the normal bundle
$N^{k_b}$ on $Y$, with  $\{n_{\alpha}\}_{\alpha=1}^{k_b}$ an orthonormal
frame on $N^{k_b}\lvert_{Y}$. Using this, we have the following
coordinates $(y^i, t^{\alpha})$ on $W$:
\[
  \Psi:\phi(Y)\times B_{\epsilon}(0)\subset
  \IR^{n-k_b}\times\IR^{k_b}
  \to  W \subset M,
\]
with
\begin{equation}\label{normal}
 \Psi(y^i,t^{\alpha}) = \exp_{y}(\sum_{\alpha=1}^{k_b}t^{\alpha} \, n_{\alpha}(y)).
\end{equation}
Using these coordinates, we write the metric near $Y$ as
$$
  g = g_{ij}dy^idy^j + g_{\alpha\beta}dt^{\alpha}dt^{\beta} 
    + g_{i\alpha}dy^i\circ dt^{\alpha},
$$
where 
\begin{equation}\label{eq:ov}
   g_{i\alpha} = \sO(|t|) 
  \quad \text{and} \quad
  g_{\alpha\beta} - \delta_{\alpha\beta} = \sO(|t|^2).
\end{equation}
Moreover, the distance function $r_b$ defined in~\eqref{eq:rb-0} is $r_b
= |t|$.  With this, we define the {\em radial vector} of this tubular
neighborhood
\begin{equation}\label{eq:rb-1}
\frac{\p}{\p r_b}:=
\sum_{\alpha=1}^{k_b}\frac{t^\alpha}{|t|}\frac{\partial}{\partial  t^{\alpha}}.
\end{equation}
It satisfies $\nabla_{\frac{\partial}{\partial r_b}} \frac{\partial}{\partial r_b}
= 0$, and
\begin{equation}\label{drb}
 |dr_b|^2 = 1.
\end{equation}
It also follows that
$$ \frac{\partial}{\partial r_b}\langle \frac{\partial}{\partial y^i},  \frac{\partial}{\partial r_b}\rangle = \langle \nabla_{ \frac{\partial}{\partial r_b}}\frac{\partial}{\partial y^i},  \frac{\partial}{\partial r_b}\rangle = -\langle \nabla_{\frac{\partial}{\partial y^i}} \frac{\partial}{\partial r_b},  \frac{\partial}{\partial r_b}\rangle = 0,$$
and hence 
\begin{equation}\label{eq:rp}
\langle \frac{\partial}{\partial y^i},  \frac{\partial}{\partial
  r_b}\rangle = 0,
\end{equation}
and
$$\sum_{\alpha} t^{\alpha}g_{i\alpha} = 0.$$

\subsection{Behavior of the spin structure near the codimension $2$
  stratum of $V$}\label{ssec:near2}

Let $W \in \TRC(V^2)$.  From~\eqref{eq:TRC}, it follows that there
exists $Y \subset V^{2}$ relatively compact so that $W= B_{\e}(Y)$ for
some $\e<\e(Y)$, and that $W$ does not intersect $\overline{V^{6}}$,
the higher codimension strata of $V$. We denote by $r: = r_2$, the
normal distance function to $V^2$. Let $(r,\theta)$ be the polar
coordinates in each normal disk $D_y \subset W$ to $y\in Y$.

We choose $W$ to be contractible. It admits a trivial spin
structure, with spinor bundle $S_0$. The two spinor bundles
$S\lvert_{W\setminus Y}$ and $S_0$ become trivial when lifted to a
connected double cover of $W\setminus Y$. We use this lifting to
identify sections of $S$ with multivalued sections of $S_0$. In
particular, the maximality of the spin structure near $V$ implies the
following lemma.
\begin{lemma}\label{hol}
  Let $W= B_{\e}(Y)\in \TRC(V^2)$ be contractible. Then with respect
  to the above identification on the double cover of $W\setminus Y$,
  each spinor $\psi$ on $W\setminus Y$ satisfies
 \begin{equation}\label{eq:hol}
   \psi(y, r, \theta + 2\pi) \, = \, - \psi(y, r, \theta),
 \end{equation}
 on each disk $D_y \subset W$, the normal disk to $y \in Y$.
 Therefore the Fourier decomposition of $\psi$ in this transverse disk
 has the form
 \begin{equation}\label{eq:fourier_modes}
   \psi (y, r, \theta) 
  = \sum_{k \in \, \frac{1}{2} + \IZ } \psi^{k}(y,r) \, e^{i k \theta}.
 \end{equation}
\end{lemma}
\begin{proof}
 The first claim follows from the fact that the holonomy of the spin connection is $-1$.

 Let now $\{f^a\}$ be a frame giving a trivialization of the trivial
 spinor bundle $S_0$ over $W$. In this frame
 \[
  \psi(y,r,\theta) = \sum_{a} \psi_{a}(y,r,\theta) f^a 
  =\sum_a \sum_{k \in \frac{1}{2}+\IZ}
   \psi^k_a(y,r) e^{ik\theta} f^a,
 \]
 with $\psi^k_a(y,r): = \frac{1}{2\pi} \int_0^{2\pi} \psi_a(y,r,s) e^{-i k s} ds$.
 Defining 
 \[
  \psi^{k}(y, r): = \sum_{a} \psi^k_a(y,r) f^a
 \]
 we obtain~\eqref{eq:fourier_modes}.
\end{proof}
 The individual Fourier components $\psi^k (y,r)$ in the above
 expansion are dependent on the choice of the trivialization of $S_0$ and 
 so are not globally defined.  
 The next lemma gives the dependence of this expansion on the
 trivialization.
\begin{lemma}\label{arbfour}
 Let $W=B_{\e}(Y)\in\TRC(V^2)$ be contractible. Let $\psi$ be a spinor on
 $W\setminus Y$, and let 
 $\psi^k (y,r)$ and $\hat\psi^k (y ,r)$ be the Fourier components of $\psi$  with respect to two different
 trivializations of $S_0$.
 Then 
  $$
    \|\psi^k -\hat{\psi}^k\|_{L^2(W\setminus
      Y,S\lvert_{W\setminus Y})}\leq C\|r
    \psi\|_{L^2(W\setminus Y, S\lvert_{W\setminus Y})}.
  $$
  Here the constant $C>0$ is independent of $\psi$ but does depend essentially on the two
  trivializations.
\end{lemma}
\begin{proof}
 We continue with the set-up in the proof of Lemma~\ref{hol}.
 Consider another frame $\{h^b\}$ giving a trivialization of the
 spinor bundle $S_0$ over $W$.
 In this other trivialization, we have
 \[
  \psi(y,r,\theta) = \sum_{k\in \frac{1}{2} + \IZ} \hat{\psi}^k (y,r) e^{ik\theta},
 \]
 with 
 \[
   \hat{\psi}^k(y,r): = \sum_b \hat{\psi}_ b^k(y,r) h^b
                     = \frac{1}{2\pi} \sum_b 
                      \left(
                        \int_{0}^{2\pi} \hat{\psi}_b(y,r,s)e^{-iks} ds
                      \right)h^b
 \]
 There exists a matrix valued map $q$ on $W$, so that the two frames
 $\{f^a\}$ and $\{h^b\}$ are related via $h^b = \sum_{a} q^{b}_a f^a$.
 With this, the above becomes
 \begin{align*}
   \hat{\psi}^k(y,r)
   &= \frac{1}{2\pi} \sum_{a,b} 
                      \left(
                        \int_{0}^{2\pi} \hat{\psi}_b(y,r,s)\ e^{-iks} ds
                      \right) q^{b}_a(y,r,\theta) f^a\\
   &= \frac{1}{2\pi} \sum_{a,b,c} 
                      \left(
                        \int_{0}^{2\pi} \psi_c(y,r,s)
                        \left(q^{-1}(y,r,s)\right)^c_b\ e^{-iks} ds
                      \right) q^{b}_a(y,r,\theta) f^a\\
   &= \frac{1}{2\pi} \sum_{a} 
                      \left(
                        \int_{0}^{2\pi} \psi_a(y,r,s) e^{-iks} ds
                      \right) f^a\\
  & \qquad \quad + \frac{1}{2\pi} \sum_{a,b,c} 
                      \left(
                        \int_{0}^{2\pi} \psi_c(y,r,s) 
                         \left(
                           \left(q^{-1}(y,r,s)\right)^{c}_bq(y,r,\theta)^b_a 
                             - \delta_a^c
                          \right)\ e^{-iks}ds
                      \right) f^a.
 \end{align*}
Thus, the last term above is exactly $\hat \psi^k(y,r) -
\psi^k(y,r)$. We rewrite it as
\begin{multline*}
  \hat{\psi}^k(y,r)- \psi^k(y,r) 
  =  \frac{1}{2\pi}  \sum_{a,b,c} 
      \left( \int_{0}^{2\pi} \psi_c(y,r,s) 
         \left(q^{-1}(y,r,s)-q^{-1}(y,0,0)\right)^{c}_b\ 
           q(y,r,\theta)^b_a  e^{-iks}ds
      \right)f^a\\
  + \frac{1}{2\pi} \sum_{a,b,c}
     \left(
       \int_{0}^{2\pi} \psi_c(y,r,s) q^{-1}(y,0,0)^{c}_b\       
             \left(q(y,r,\theta)- q(y,0,0)\right)^b_a e^{-iks}ds
      \right)f^a
\end{multline*}
Using the Taylor expansion for the components of $q(y,r,\theta)$ and
$q^{-1}(y,r,\theta)$ around $(y,0,0)$, it follows that both terms on
the right-hand side are of the form $\sO(r \int_{0}^{2\pi}
|\psi(y,r,\theta)|d\theta)$.  We then have
\[
  \|\psi^k(y,r) - \hat{\psi}^k(y,r)\|^2_{L^2(W\setminus Y,    S\lvert_{W\setminus Y})}
  = \int_0^{\e} \int_{0}^{2\pi} \int_Y |\psi^k(y,r)- \hat\psi^k(y,r) |^2
  m(y,r,\theta)\, dy\, d\theta\, rdr
\]
with $m$ smooth and bounded from above and below by strictly positive
constants.  (Such an $m$ exists since on this neighborhood the metric
is quasi-isometric with the product metric on $W$.) 
Hence
\begin{align*}
  \|\psi^k(y,r) - \hat{\psi}^k(y,r)\|^2_{L^2(W\setminus Y,
    S\lvert_{W\setminus Y})} & \leq C_1 \int_0^{\e} \int_{0}^{2\pi}
  \int_Y \left(r \int_{0}^{2\pi} |\psi(y,r,s)|ds\right)^{\! \! 2}
  m(y,r,\theta) \, dy\, d\theta\, rdr\\
  & \leq 2\pi C_1
  \int_0^{\e} \int_{0}^{2\pi} \int_Y  \left(\int_0^{2\pi}\left(r|\psi(y,r,s)|\right)^2ds\right) m(y,r,\theta)\, dy\, d\theta\, rdr\\
  & \leq 4\pi^2 C_2  \int_0^{\e} \int_Y
  \int_0^{2\pi}\left(r|\psi(y,r,s)|\right)^{2} m(y,r,s) \, dy\, d\theta\, rdr\\
  & = C \|r \psi\|^2_{L^(W\setminus Y, S\lvert_{W\setminus Y})}.
\end{align*}
Here $C_1$ is a positive constant which depends on the two chosen
trivializations of $S_0$, while $C_2$ is a positive constant depending
on the upper and lower bounds of $m$ and on $C_1$.
\end{proof}

\section{Preliminary estimates}\label{sec:pe}
We will need estimates for the decay of various
spinors on the asymptotically flat ends of $M$ and also near the strata
of $V$. We use two types of estimates: radial Hardy inequality
estimates, and angular estimates near the codimension $2$ stratum of
$V$. The angular estimates are a direct consequence of the fact that
the spin structure on $M\setminus V$ has holonomy around small 
circles normal to $V^2$. 

\subsection{Hardy Inequalities}\label{subsec:poincare_est}
In this section we gather some Hardy inequalities, which we
require to control both large scale and small scale behavior of
functions and spinors.  All the estimates arise as simple
perturbations of basic Euclidean Hardy inequalities for
functions, whose proofs we first recall.  The passage from
inequalities for functions to inequalities for spinors follows from
Kato's inequality. All the norms in this section are $L^2$-norms.

\subsubsection{Euclidean Hardy Inequalities}
\begin{proposition}\label{prop:eucl_poinc}
 Let $f\in H^1(\IR^n)$, $n>2$. Let $\rho$ denote the radial distance. Then
 \begin{equation}\label{eq:epoinc} 
   \|d f\|^2\geq \frac{(n-2)^2}{4}\, \|\frac{f}{\rho}\|^2 .
 \end{equation}
 When $n=2$ we have
 \begin{equation}\label{eq:epoinc2}
   \|d f\|^2 \geq \frac{1}{4}\, \|\frac{f}{\rho \ln(\frac{1}{\rho})}\|^2 .
 \end{equation}
\end{proposition}
\begin{proof}
 The subspace $\sC_0^{\infty}(\IR^n)$ is dense in
 $H^1(\IR^n)$. Hence it suffices to prove the estimate for this
 subspace.  Let $f\in C_0^{\infty}(\IR^n)$, and compute
 in spherical coordinates:
 \begin{equation*}
  \begin{split}
    \|\frac{f}{\rho}\|^2 
      & =\frac{1}{n-2}\int_{0}^{\infty}\int_{S^{n-1}}|f|^2\, d\sigma\, d\rho^{n-2}  \\
      & = \frac{-2}{n-2}\int_{0}^{\infty}\int_{S^{n-1}}
               f\,\frac{\partial f}{\partial \rho}\rho^{n-2}\, d\sigma\,d\rho \\
      &  \leq \frac{2}{n-2}\|\frac{f}{\rho}\| \, \|df\|.
  \end{split}
 \end{equation*}
Dividing through by $\|\frac{f}{\rho}\|$ and squaring gives the desired estimate. 

When $n=2$, one computes similarly with $\frac{f}{\rho}$ replaced 
by $\frac{f}{\rho\ln(\frac{1}{\rho})}$. One needs the extra condition $f(0) = 0$ to extend the estimate to $n=1$. 
\end{proof}
Essentially the same proof yields the following variant
of the above proposition.
\begin{proposition}\label{prop:k_eucl_poinc}
 Let $f\in H^1(\IR^n)$, $n>2$. Let $r$ denote the distance to an
 affine subspace of codimension $k>2$. Then
 \begin{equation}\label{eq:k_epoinc}
  \|d f\|^2 \geq \frac{(k-2)^2}{4}\, \|\frac{f}{r}\|^2 .
 \end{equation}
 When $k=2$ we have
 \begin{equation}\label{eq:k_epoinc2}
  \|d f\|^2\geq \frac{1}{4} \|\frac{f}{r\ln(\oner)}\|^2 .
 \end{equation}
\end{proposition}

\subsubsection{Hardy inequalities for the asymptotically flat ends}\label{subsec:af_poinc}
Let $(M,g)$ be an asymptotically flat Riemannian manifold of order
$\tau>0$ as defined in Definition~\ref{def:asy_flat}.
We modify Proposition~\ref{prop:eucl_poinc} to a form suitable to the
asymptotically flat ends $(M_l,Y_l)$ of $M$. Recall that $\rho$ denotes the pullback
by $Y_l$ of the radial coordinate in $\IR^n\setminus B_T(0)$. 
\begin{proposition}\label{prop:af_poinc}
 There is a constant $C_l>0$ so that for all $f\in \sC^{\infty}_0(M_l)$,
 \begin{equation}\label{eq:af_poinc}
  \|\frac{f}{\rho}\|^2 
   \leq 
   \frac{4}{(n-2)^2}\|df\|^2 + C_l \, \|\frac{f}{\rho^{1+\tau/2}}\|^2.
 \end{equation}
\end{proposition}
\begin{proof}
In spherical coordinates we have
 $$
    \|f\|^2 
     = \int_{T}^{\infty}\int_{S^{n-1}}
         |f|^2\,\rho^{n-1}\mu(\rho,\sigma)\, d\sigma\, d\rho,
 $$
 where $\mu$ satisfies $\mu = 1+\sO(\rho^{-\tau})$ and $|d\mu| = \sO(\rho^{-\tau-1})$.
 Then
 \begin{equation*}
  \begin{split}
   \|\frac{f}{\rho}\|^2 
    &= \frac{1}{n-2}\int_{T}^{\infty}\int_{S^{n-1}}
         |f|^2 \ \mu(\rho,\sigma)\,d\sigma\,d\rho^{n-2}\\
    &= \frac{-2}{n-2}\int_{T}^{\infty}\int_{S^{n-1}}
          f \frac{\partial f}{\partial \rho} \
           \mu(\rho,\sigma)\rho^{n-2}\, d\sigma\,d\rho \\
    & \qquad   
        + \frac{-1}{n-2}\int_{T}^{\infty}\int_{S^{n-1}}
          |f|^2\frac{\partial \mu(r,\sigma) }{\partial \rho}\
           \rho^{n-2} d\sigma\,d\rho \\
    & \leq \frac{2}{n-2} \|\frac{f}{\rho |d\rho|}\|\,\|df\| 
         + C\|\frac{f}{\rho^{1+\tau/2}}\|^2,
  \end{split}
 \end{equation*}
 for some constant $C>0$.  
 Asymptotic flatness implies $|d\rho| = 1 + \sO(\rho^{-\tau})$. 
 Hence
 $$
  \|\frac{f}{\rho}\|^2 
  \leq 
   \frac{4}{(n-2)^2}\|df\|^2 
   +C_l\|\frac{f}{\rho^{1+\tau/2}}\|^2.
 $$
 for some $C_l>0$ independent of $f$.
\end{proof}

\subsubsection{Radial estimates near $V$}\label{subsec:V_poinc}
We also need Hardy inequalities near the stratified set
$V$. Together with Kato's inequality, they give growth estimates for
spinors in neighborhoods of each of the strata $V^{k_b}$ of $V$.
We derive all these inequalities on tubular neighborhoods in $\TRC(V^{k_b})$
(defined in~\eqref{eq:TRC}) on which we have a well-behaved radial distance
function $r_b$. Because these estimates use only the radial derivative of a function to control its growth, we will refer to these as {\em radial} estimates.  

\begin{proposition}\label{asympoincv}
  Let $B_\epsilon(Y)\in\TRC(V^{k_b})$.  There is a constant $C_{Y}>0$
  so that for all $f\in \sC_0^{\infty}(B_\epsilon(Y))$,
 \begin{equation}\label{ag1v}
  \|\frac{f}{r_b}\|^2
  \leq 
   \frac{4}{(k_b-2)^2}\|df\|^2
   +C_Y\|\frac{f}{r_b^{1/2}}\|^2
 \quad \text{if\ $k_b>2$},
 \end{equation}
 and 
  \begin{equation}\label{ag1v2}
  \|\frac{f}{r_b\ln(\frac{1}{r_b})}\|^2
  \leq 
    4 \|df\|^2
   +C_Y\|\frac{f}{r_b^{1/2}\ln^{1/2}(\frac{1}{r_b})}\|^2
 \quad \text{if\ $k_b=2$}.
 \end{equation}
\end{proposition}
\begin{proof}
 The proof is similar to the proof of~\eqref{eq:af_poinc}. 
  Write
 \begin{equation*}
  \begin{split}
   \|\frac{f}{r_b}\|^2
    & = \int_{0}^{\epsilon}\int_{S^{k_b-1}\times Y} 
           |f|^2 \,r_b^{n-3}\,
           m(y, r_b,\sigma)\, 
            d\sigma \,dy\,dr_b\\
    & = \frac{1}{k_b-2}\int_{0}^{\epsilon}\int_{S^{k_b-1}\times Y}
           |f|^2 \,m(y, r_b,\sigma)\, d\sigma\, dy\,  d(r_b^{k_b-2}),
  \end{split}
 \end{equation*} 
 where $m$ is smooth and multiplicatively bounded. Integrating by parts gives
 \begin{equation*}
  \begin{split}
   \|\frac{f}{r_b}\|^2
    &=  -\frac{2}{k_b-2}\int_{0}^{\epsilon}\int_{S^{k_b-1}\times Y}
          f\,\frac{\partial f}{\partial r_b} \, 
          m(y,r_b,\sigma)
          r_b^{k_b-2}\,d\sigma\,dy\,dr_b\\\
    & \quad -\frac{1}{k_b-2}\int_{0}^{\epsilon}\int_{S^{k_b-1}\times Y}
         |f|^2 \,\frac{\partial m}{\partial r_b}\,
                r_b^{k_b-2}\,
                d\sigma\,dy\,dr_b\\
    & \leq \frac{2}{k_b-2} \|\frac{f}{r_b}\|\,\|df\|
         + C_Y \|\frac{f}{r_b^{1/2}}\|^2,
  \end{split}
 \end{equation*}
with $C_Y$ a constant depending on the supremum of $\frac{1}{m}$ and
$\frac{\partial m}{\partial r_b}$ on $B_{\e}(Y)$ for all  $\e<\e(Y)$.
From here, the result follows.

The proof when $k_b=2$ is similar.
\end{proof}

\begin{remark}
   The preceding propositions are written as lower bounds for
   $\|df\|^2$. In their proofs, however, all derivatives but the
   (generalized) radial derivatives are discarded. Hence, these
   estimates can be rewritten as lower bounds for $\|\frac{\p f}{\p
   r_b}\|^2$.
 \end{remark}

Kato's inequality for a smooth spinor $\psi$,
\begin{equation}\label{eq:kato}
  | d |\psi| | \leq | \nabla \psi |
\end{equation}
on the support of $\psi$, and the above Hardy inequality give the
following radial estimates for spinors on $M\setminus V$.
\begin{corollary}\label{cor:poincb}
 For every $B_{\epsilon}(Y)\in \TRC(V^{k_b})$  there exists $C_Y>0$ so that for any
 spinor $\psi$ in the closure of $\sC^{\infty}_0(B_{\e}(Y)\setminus V, S)$ with
 respect the $H^1_{\rho}$-norm,  
 \begin{equation}\label{eq:poincb}
  \|\nabla_{\frac{\p}{\p r_b}} \psi \|^2 
   \geq \frac{(k_b-2)^2}{4}\|\frac{\psi}{r_b}\|^2 - C_Y
   \|\frac{\psi}{r_b^{1/2}}\|^2
 \quad \text{if\ $k_b>2$},
 \end{equation}
 and 
  \begin{equation}\label{eq:poinc2}
   \|\nabla_{\frac{\p}{\p r_b}} \psi \|^2 
   \geq \frac{1}{4} \|\frac{\psi}{r_b\ln(\frac{1}{r_b})}\|^2 
 - C_Y
   \|\frac{\psi}{r_b^{1/2}\ln^{1/2}(\frac{1}{r_b})}\|^2
 \quad \text{if\ $k_b=2$}.
 \end{equation}
\end{corollary}

\subsection{Angular estimates near $V^2$}\label{subsec:angular_est}
Near the codimension two stratum of $V$, the absence of a zero mode in
the Fourier decomposition~\eqref{eq:fourier_modes} gives us a sharper
estimate than the radial estimate~\eqref{eq:poinc2}.

We continue with the notation from Section~\ref{ssec:near2}. Let $W =
B_{\e}(Y) \in \TRC(V^2)$ be contractible. Let $(r, \theta)$ be the polar
coordinates on each normal disk $D_y \subset W$ for $y \in Y$.  Let
$e_{\theta}$ denote a unit vector in the
$\frac{\partial}{\partial\theta}$ direction, and $\nabla_{e_{\theta}}$
denote the corresponding covariant derivative on the spinor bundle
induced from the Levi-Civita connection.
Let $\{e_a\}_{a=1}^n$ be an orthonormal frame of the tangent bundle to
$M$ on $W$. Then $\nabla_{e_{\theta}}$ can be expressed as
$$
  \nabla_{e_{\theta}} 
  = e_{\theta} - \frac{1}{4}\omega_{ab}(e_{\theta})c(e_a)c(e_b),
$$
where $\omega_{ab}(e_{\theta}) = g(\nabla_{e_{\theta}}e_a, e_b)$. The
orthonormal frame can be chosen so that the $\omega_{ab}(e_{\theta})$ are bounded on
$W$ for all $1\leq a, b\leq n$.
Hence we find that on a normal disk in a suitable frame, we have
$$
  \nabla_{e_{\theta}} = e_{\theta}  + \sO(1).
$$
Since the absolute value of the Fourier coefficient of a spinor $\psi$
on $M\setminus V$ restricted to $D_y$ has $\frac{1}{2}$ as a lower
bound, we obtain the following {\em angular estimate}:
\begin{proposition}\label{prop:angular}
 For $\psi$ a spinor on $M\setminus V$ and any $W =B_{\e}(Y) \in
\TRC(V^2)$, there
exists a constant $C_Y>0$ independent of $\psi$ so that on each normal
circle $S^1_y(r)$ centered at $y\in Y$ and included in $W$, we have
 \begin{equation}\label{theta}
  \int_{S^1_{y}(r)}\left|\nabla_{e_{\theta}} \psi\right|^2 d\theta
  \geq \frac{1}{4}\int_{S^1_{y}(r)}\left|\frac{\psi}{r}\right|^2 d\theta
  -C_Y \int_{S^1_y(r)}\left|\frac{\psi}{\sqrt{r}}\right|^2 d\theta.
 \end{equation}
 As a consequence,
 \begin{equation}\label{L2-theta}
  \| \nabla_{e_{\theta}}\psi\|^2_{L^2(W)} 
  \geq \frac{1}{4} \|\frac{\psi}{r}\|^2_{L^2(W)} - C_Y \|\frac{\psi}{r^{1/2}}\|^2_{L^2(W)}.
 \end{equation}
\end{proposition}

\section{The Dirac operator on  $M\setminus V$}\label{sec:Do}
In this section we derive the first properties of the Dirac operator
on $M \setminus V$. We introduce weighted Sobolev spaces, and also the maximal and minimal domains of the
Dirac operator viewed as an unbounded operator on $L^2$-spinors. 
Then, using the Lichnerowicz formula we derive various properties of
the Dirac operator on these spaces, which we then use to verify that
certain harmonic spinors satisfy the growth conditions of Witten
spinors. We end this section with the proof that the boundary terms
coming from the asymptotically flat ends in the Lichnerowicz formula
applied to a Witten spinor give exactly the mass of the manifold.

Let $(M, g)$ be an oriented Riemannian manifold which is
asymptotically flat of order $\tau >0$ as defined in
Definition~\ref{def:asy_flat}. 
We assume that $M$ is nonspin, and let $V \subset K$ be a stratified
set given by Theorem~\ref{thm:V}. Without loss of generality we can
assume that the radial coordinate $\rho$ was extended
smoothly to the interior of $M$ so that it is identically $1$ in a
neighborhood of $V$.
Let $S$ be the spin bundle corresponding to the maximal spin structure
on $M\setminus V$, and let $\nabla$ denote the associated spin
connection and $D$ the corresponding Dirac operator. At a point $x \in
M\setminus V$, $D$ has the form $D = \sum_{i=1}^n c(e_i) \nabla_{e_i}$
with $\{e_1, \ldots, e_n\}$ any orthonormal frame of tangent vectors
at $x$, $\nabla$ the spin connection on $S$ determined by the
Levi-Civita connection of the metric $g$, and $c(e_i)$ denoting
Clifford multiplication by the vector $e_i$.

\subsection{Weighted Sobolev spaces of spinors on $M\setminus V$ and
  Rellich's compactness}\label{subsec:functions_g}
First we define the weighted Sobolev spaces for
spinors which will be used to construct the Witten spinors in our main
theorem. We also prove a Rellich-type compactness result.

Each open set in $M\setminus V$ is of the form $W = U \setminus V$
with $U$ an open set in $M$. For each $W$, let $L^2_{\rho}(W, S)$ be
the completion of $\sC^{\infty}_{0}(W, S)$ in the norm
\begin{equation}\label{eq:L2rho}
  \|\psi\|_{L^2_{\rho}}: = \|\frac{\psi}{\rho}\|_{L^2},
\end{equation}
and $H^1_{\rho}(W,S)$ be the completion of $\sC^{\infty}_0(W,S)$ in
the norm
\begin{equation}\label{eq:H1rho}
  \|\psi\|_{H^1_{\rho}}: = \|\nabla \psi\|_{L^2} + \|\psi\|_{L^2_{\rho}}.
\end{equation}
We have the following version of the Rellich compactness
theorem.
\begin{lemma}\label{lem:rellich}
  Let $U$ be a bounded open set  with smooth boundary  in
  $M$ and let $W= U\setminus V$. The inclusion
 \[
   H^1_{\rho}(W, S) \hookrightarrow L^2_{\rho}(W, S)
 \]
 is compact.
\end{lemma}
\begin{proof}
 Let $\{\psi_j\}$ be a sequence of spinors on $W$, bounded in
 $H^1_{\rho}(W,S)$. We need to show that it contains a subsequence
 which is convergent in $L^2_{\rho}(W,S)$. 

 First note that by Kato's inequality $\lvert d|\psi_j| \lvert \leq |
 \nabla \psi_j |$, and the sequence of functions $f_j := |\psi_j|$
 forms a bounded sequence in $H^1_0(U)$, the completion of
 $\sC^{\infty}_0(U)$ in the $H^1$-norm.  By Rellich's compactness
 theorem for bounded open sets  with smooth boundary  in
 complete manifolds, $\{f_j\}$ contains a subsequence, also denoted
 $\{f_j\}$, which is convergent in $L^2(U)$.

 Let $W_1 \subset W_2 \subset \ldots
 \subset W_k \subset \ldots$ be an exhaustion of $W$ by compact sets with smooth boundary.
 We show that we can find a subsequence of $\psi_j$ which converges in
 $L^2(W_k,S)$ for each $k$. 

 Note that the $W_k$ are compact sets in $M$ which do not intersect
 $V$. Let $H^1(W_k,S)$ denote the set of spinors $\psi \in L^2(W_k)$
 with $\nabla \psi \in L^2$. By Rellich's compactness theorem, the
 inclusion $H^1(W_k, S) \hookrightarrow L^2(W_k,S)$ is compact.  Since
 the weight function $\rho$ is bounded on $U$, $\{\psi_j\}\subset
 H^1(W_k,S)$ is a bounded sequence for all $k$.
 Hence we can find a subsequence $\{\psi_{j,1}\}$ of $\{\psi_j\}$
 which converges in $L^2(W_1,S)$. Iterating this we have: given a
 subsequence $\{\psi_{j,k}\}$ which converges in $L^2(W_k, S)$, we can
 pass to a new subsequence $\{\psi_{j, k+1}\}$  which converges in
 $L^2(W_{k+1},S)$. Taking a diagonal subsequence, we produce a
 subsequence, also denoted by $\{\psi_j\}$, which is convergent in
 $L^2(W_k,S)$ for all $k$.

 To conclude the proof, we need to show that this $\{\psi_j\}$ is a Cauchy
 subsequence in $L^2_{\rho}(W, S)$.  We have
 \begin{equation*}
  \begin{split}
    \|\psi_j - \psi_l\|_{L^2_{\rho}(W,S)}
    & = \|\psi_j - \psi_l\|_{L^2_{\rho}(W_k,S)} + \|\psi_j -   \psi_l\|_{L^2_{\rho}(W\setminus W_k,S)}\\
    & \leq \|\psi_j - \psi_l\|_{L^2_{\rho}(W_k,S)} +
    C_U( \|f_j\|_{L^2(U\setminus W_k)} + \|f_l\|_{L^2(U\setminus W_k)}),
  \end{split}
 \end{equation*}
   for $C_U$ a constant depending on the maximum of $\rho$ on $U$. 
 Choose $\epsilon >0$.  Since the sequence $\{f_j\}$ is
 convergent in $L^2(U)$, we can find $k$ large enough and $N_1 >0$, so that
 for all $j\geq N_1$, we have $ C_U \|f_j\|_{L^2(U\setminus W_k)} \leq
 \epsilon/3$. Then, since $\{\psi_{j}\}$ is convergent in
 $L^2(W_k,S)$, we can find $N> N_1$ so that for all $j,k \geq N$ we
 have $\|\psi_j - \psi_l\|_{L^2_{\rho}(W_k,S)} \leq \epsilon/3$.
\end{proof}

\subsection{The Dirac operator  and the Lichnerowicz formula}
The Lichnerowicz formula relates the Dirac Laplacian on $M\setminus V$
to the connection Laplacian on spinors:
\begin{equation}\label{eq:lich}
 D^* D = \nabla^{*}\,\nabla + \frac{R}{4}.
\end{equation}
Here $D^*$ and $\nabla^*$ denote respectively the 
formal adjoints of the Dirac operator and the spin connection.
Since $D$ is a self-adjoint operator, $D^{*} = D$. 

This equality of differential operators gives the {\em pointwise
  Lichnerowicz formula}: For all spinors $\psi \in
\sC^{\infty}(M\setminus V, S)$ we have
\begin{equation}\label{eq:pLich}
  |\nabla \psi|^2 + \frac{1}{4} R |\psi|^2 - |D\psi|^2
  = \div(W),
\end{equation}
where $W$ is the vector field on $M\setminus V$ defined by 
\[
  \<W, e\> = \< \nabla_e \psi + c(e) D\psi, \psi\>
\]
for all $e \in T(M\setminus V)$. 

When integrating formula~\eqref{eq:pLich} on $M\setminus V$, the
divergence is expected to introduce boundary terms from the
asymptotically flat ends of $M$ and from $V$. However, when the spinor
$\psi$ is in $H^1_{\rho}(M\setminus V, S)$, this contribution vanishes.
\begin{proposition}\label{prop:boch1}
 Let $(M,g)$ be a nonspin Riemannian manifold which is asymptotically flat of
 order $\tau>0$. Then the Dirac operator
 \[
   D: H^1_{\rho}(M\setminus V,S) \to L^2(M\setminus V,S)
 \]
  is a bounded linear map which satisfies the integral Lichnerowicz formula
 \begin{equation}\label{boch1}
   \| D \psi\|^2 = \| \nabla \psi\|^2 + \frac{1}{4}(R \psi, \psi).
 \end{equation}
 Moreover, if the scalar curvature is nonnegative, the Dirac operator
 is injective on $H^1_{\rho}(M\setminus V, S)$.
\end{proposition}
\begin{proof}
 The boundedness of the Dirac operator follows immediately from its
 definition.
 To prove the integral Lichnerowicz formula, note that since the metric
 $g$ is asymptotically flat of order $\tau >0$, then in the induced frame
 $\{x_i\}$ on each of the asymptotically flat ends $M_{l}$ of $M$ we have
 the scalar curvature
 \begin{align}\label{eq:sc-decay} 
   R &= g^{jk} \left(\partial_i \Gamma^{i}_{jk}-\partial_k\Gamma^{i}_{ij} +
     \Gamma^{i}_{il}\Gamma^{l}_{jk} -
     \Gamma^{i}_{kl}\Gamma^{l}_{ij}\right)\nonumber \\
     & = \partial_j(\partial_i g_{ij} -\partial_j g_{ii}) + \sO(\rho^{-2\tau-2}),
 \end{align}
 and thus $R = \sO(\rho^{-\tau -2})$. From here it follows that both
 sides of the formula~\eqref{boch1} define continuous functionals on
 $H^1_{\rho}(M\setminus V, S)$ which agree on the dense
 subspace $\sC^{\infty}_0 (M\setminus V, S)$.

 The injectivity statement is a consequence of the Lichnerowicz
 formula.  Let $\psi \in H^1_{\rho}(M\setminus V,S)$ so that $D\psi
 =0$. Since $R \geq 0$, formula~\eqref{boch1} implies that $\psi$ is
 covariantly constant.  Thus it must be identically $0$ in order to be
 in $L^2_{\rho}(M\setminus V,S)$.
\end{proof}

The weighted Sobolev spaces $H^1_{\rho}(M\setminus V, S)$ and
$L^2_{\rho}(M\setminus V, S)$ are well adapted for the coercivity results
which we prove in Section~\ref{sec:coercive}. Because $M\setminus V$ is incomplete, we also need to take care in defining the domain of the Dirac operator. As an operator on the smooth compactly supported
sections of $M\setminus V$, the Dirac operator has two natural
extensions as an unbounded operator on $L^2(M\setminus V,S)$. The {\em
  minimal extension} of $D$ has as domain the {\em minimal domain}
$\Dmin(D)$, the completion of $\sC^{\infty}_0(M\setminus V, S)$ in the
graph norm, $\|\psi\| + \|D\psi\|$.  The {\em maximal extension} has
as domain the {\em maximal domain} $\Dmax(D)$, which consists of those
$\psi \in L^2(M\setminus V, S)$ so that $D\psi \in L^2(M\setminus V,
S)$.

We have the following properties of the minimal extension of the Dirac
operator on $M\setminus V$.
\begin{corollary}\label{cor:Dmin-ns}
  Let $(M,g)$ be a nonspin asymptotically flat Riemannian manifold of
  order $\tau >0$. Then the minimal extension of the Dirac operator $D$
  on $M\setminus V$ satisfies:
 \begin{enumerate}
  \item $\Dmin (D) \subset H_{\rho}^1(M\setminus V,S)$,
  \item Given $\psi \in H_{\rho}^1(M\setminus V, S)$, $\eta \psi \in
    \Dmin(D)$ for all $\eta \in \sC^{\infty}_0(M)$,
  \item If the scalar curvature is nonnegative, then the null-space of
    the Dirac operator on the minimal domain is trivial.
 \end{enumerate}
\end{corollary}

\begin{proof}
 The only claim that needs an argument is the first one. 

Let $\psi \in \Dmin(D)$. Then there exists $\{\psi_j\} \subset
 \sC^{\infty}_0 (M\setminus V, S)$ converging to $\psi$ in the graph
 norm of $D$. Since $\rho \geq 1$, it follows that $\psi_j \to \psi$
 in the $L^2_{\rho}$-norm.  Moreover, the Lichnerowicz
 formula~\eqref{boch1} gives
 \[
   \|D \psi_j \|^2 = \| \nabla \psi_j \|^2 + \frac{1}{4} (R\psi_j, \psi_j).
 \]
 Since the scalar curvature is bounded, the sequence $\{\|\nabla
   \psi_j\|\}$ is convergent.  Thus, $\{\psi_{j}\}$ converges in
 $H^1_{\rho}$-norm to $\psi$ and hence $\psi \in H^1_{\rho}(M\setminus
 V, S)$.
\end{proof}
\begin{remark}\label{rem:Dmaxmin}
 In the case of a complete Riemannian spin manifold, the maximal and
 the minimal extension coincide,~\cite[Theorem 1.17]{gl}.   
\end{remark}

\subsection{Growth conditions for the Witten spinor}
We now show that an asymptotically constant harmonic spinor
constructed as in Section \ref{sec:WittenProof} ({\em assuming
  existence of a solution to (\ref{eq:u})}) satisfies the growth
  conditions of a Witten spinor. This result will be used in the proof
  of our Theorem A. We start with the following  
  regularity result on the asymptotically flat ends.

\begin{proposition}\label{prop:bochend}
  Let $(M,g)$ be a nonspin Riemannian manifold which is asymptotically
  flat of order $\tau>0$. If $\frac{v}{\rho}\in L^2(M\setminus V,S)$
  and $Dv\in L^2(M\setminus V,S)$, then $\nabla v\in
  L^2(M_l,S\lvert_{M_l})$, for every end $(M_l,Y_l)$ of $M$.
\end{proposition}
\begin{proof}
  Fix an asymptotically flat end $M_l$, and let $\eta$ be a smooth
  cutoff function supported in $M_l$ in the region $\rho \geq L$ and
  identically $1$ in the region $\rho \geq 2L$, for some $L>T$.  Next
  choose a sequence $\{\gamma_j\}$ of smooth cutoff functions on $M$,
  compactly supported in the region $\rho \leq 2j$, identically equal
  to $1$ in the region $\rho \leq j$ and so that $|d\gamma_j| \leq
  \frac{2}{\rho}$. Let \( \eta_j: = \eta \gamma_j \).  Since $\eta_jv
  \in H^1_{\rho}(M\setminus V, S)$, the Lichnerowicz
  formula~\eqref{boch1} gives
\begin{align*}
   \| D (\eta_jv)\|^2 
   & = \| \nabla (\eta_jv)\|^2 + \frac{1}{4}(R \eta_jv, \eta_jv)\\
   & =  \| \eta_j\nabla v\|^2+\| |d\eta_j|v\|^2 +2 (\eta_j\nabla
   v,d\eta_j\otimes v ) + \frac{1}{4}(R \eta_jv, \eta_jv)\\
   & \geq \frac{1}{2}\| \eta_j\nabla v\|^2-\| |d\eta_j|v\|^2 + \frac{1}{4}(R \eta_jv, \eta_jv).
\end{align*}
We expand the left-hand side of the above as
$$\| D (\eta_jv)\|^2 = \| \eta_jD v\|^2+\| |d\eta_j|v\|^2+2( \eta_jDv,c(d\eta_j)v),$$ 
and after rearranging, we obtain
\begin{align*}
 \frac{1}{2} \|\eta_j \nabla v\|^2
 &  \leq \| \eta_jD v\|^2 +2( \eta_jDv,c(d\eta_j)v) 
     +2 \| |d\eta_j| v\|^2 
 - \frac{1}{4}(R \eta_jv, \eta_jv)\\
 &  \leq \|D v\|^2 + C \|\frac{v}{\rho}\|^2,
\end{align*}
for $C > 0$ a positive constant independent of $j$ and $v$. Here we used the 
fact that since the manifold is asymptotically flat of order $\tau>0$, 
the scalar curvature $R$ behaves like $\sO(\rho^{-\tau -2})$ on the
asymptotically flat ends (see identity~\eqref{eq:sc-decay}).
Hence we may take the limit as $j\rightarrow\infty$ to deduce that
$\eta \nabla v\in L^2(M_l,S\lvert_{M_l})$.
\end{proof}
 
\begin{corollary}\label{prop:vH1rho}
 Let $(M,g)$ a nonspin Riemannian manifold which is asymptotically
 flat of order $\tau >0$.
 Let $u \in H^1_{\rho} (M\setminus V, S)$ be a smooth spinor so that
 $D^2 u \in L^2 (M\setminus V, S)$.
Then the spinor $v: = Du \in \Dmax(D)$ and $\nabla v \in
 L^2(M_l,S\lvert_{M_l})$ for each end $(M_l,Y_l)$ of $M$.
\end{corollary}

Recall from the introduction that a spinor is called constant near
infinity if it is constant on each end $M_l$ with respect to a frame
induced by the chosen asymptotically flat coordinate chart
$(M_l,Y_l)$. Let $\psi_0$ be a smooth spinor, constant near infinity
and vanishing in a neighborhood of $V$. Since the coefficients of the
spin connection associated to the asymptotically flat metric $g$
differ from the coefficients of the spin connection associated to the
Euclidean metric by terms which decay like $\sO(\rho^{-\tau-1})$, it
follows that
\begin{equation}\label{eq:cs}
  \rho^{\tau+1} |D\psi_0| \quad \text{and} \quad 
  \rho^{\tau+2} |D^2 \psi_0|
\end{equation}
are bounded on $M\setminus V$.

As in the case of asymptotically flat {\em spin} manifolds (see
Section~\ref{sec:WittenProof}), we will construct the Witten spinor in
Theorem A by solving $D^2u = -D\psi_0,$ and setting $\psi = \psi_0
+Du.$ As a corollary to Corollary \ref{prop:vH1rho}, we see that
$\psi$ satisfies the conditions in Definition~\ref{def:ws} of a Witten
spinor:
\begin{corollary}\label{cor:scws}
 Let $(M,g)$ a nonspin Riemannian manifold which is asymptotically
 flat of order $\tau >\frac{n-2}{2}$.   
 Let $\psi_0$ be a smooth spinor on
 $M\setminus V$ which is constant at infinity and supported away from $V$.
 Assume that there exists $u \in H^1_{\rho}(M\setminus
 V, S)$ so that
 \begin{equation*}
   D^2 u = -D \psi_0.
 \end{equation*}
  Let $v = Du$. Then the spinor
 \[
   \psi: = v + \psi_0
 \]
 is a Witten spinor.
\end{corollary}
\begin{proof}
  By construction $\psi$ is in the null-space of $D$. From
  Definition~\ref{def:ws} of Witten spinors, we only need to
  check that $v\in H^1_{\rho}(M\setminus V, S)$.
  Since the spinor $\psi_0$ is constant at infinity, it follows that
  $\rho^{\tau +1} |D\psi_0|$ is bounded. Since $\tau > \frac{n-2}{2}$,
  $\rho^{\tau +1} |D\psi_0| \in L^2 (M\setminus V, S)$, and thus the
  spinor $u$ satisfies the hypothesis of
  Proposition~\ref{prop:vH1rho}. Hence $v \in H^1_{\rho}(M\setminus V,
  S)$.
\end{proof}

\subsection{Obtaining the mass from a Witten spinor}

For spinors that are not in $H^1_{\rho}(M\setminus V, S)$, the
integral Lichnerowicz formula~\eqref{boch1} need not hold,
because the integration by parts introduces boundary terms. For the
Witten spinors, the  boundary terms arising from  the
asymptotically flat ends of $M$ give
  exactly the mass. 
\begin{proposition}\label{prop:mass}
  Let $(M,g)$ be a nonspin Riemannian manifold which is asymptotically
  flat of order $\tau >\frac{n-2}{2}$.  Let $\psi_0$ be a constant
  spinor on the asymptotically flat ends of $M$ with $|\psi_0|\to 1$
  at infinity on each of them. Let $\psi$ be a Witten
  spinor on $M \setminus V$ asymptotic to $\psi_0$.
  Assume there exists a smooth cutoff function $\chi$ supported in a
  neighborhood of $V$ and equal to $1$ in a smaller neighborhood of
  $V$ so that $\chi \psi \in \Dmin(D)$.
 Then
 \begin{equation}\label{eq:mass-pos}
   \int_{M\setminus V} |\nabla \psi|^2 + \frac{R}{4} |\psi|^2 
    = \frac{c(n)}{4} \mass(M,g).
 \end{equation}
  Moreover, if $R\in L^1(M)$ and nonnegative, the mass is
   finite and nonnegative. 
\end{proposition}

\begin{remark}\label{rem:psi}
 When the asymptotic values $\psi_{0l}$ of the Witten spinor $\psi$ in
 Proposition~\ref{prop:mass} on the asymptotically flat ends $M_l$ of
 $M$ have not necessarily norm $1$, we obtain
 \begin{equation}\label{eq:mass-pos-gen}
   \int_{M\setminus V} |\nabla \psi|^2 + \frac{R}{4} |\psi|^2 
    = \frac{c(n)}{4} \sum_{l}\|\psi_{0l}\|^2\mass(M_l,g).
 \end{equation}
\end{remark}

\begin{proof}[Proof of Proposition~\ref{prop:mass}]
  The condition $\chi \psi \in \Dmin(D)$ %
  implies that, as an ideal boundary, $V$ makes no contribution when
  integrating the divergence term on the right-hand side of the
  pointwise Lichnerowicz formula~\eqref{eq:pLich}. Since $\psi$ is a
  Witten spinor, $D\psi =0$. Integrating~\eqref{eq:pLich} on
  $M\setminus V$ for this $\psi$, and applying the divergence theorem
  to its right-hand side, we obtain
 \begin{equation}\label{eq:mint}
  \int_{M\setminus V} |\nabla \psi|^2 + \frac{R}{4}|\psi|^2 
  = \sum_{l=1}^L \lim_{\rho\to \infty} \int_{S^{n-1}_{\rho,l}}
  \<\nabla_{\nu}\psi +c(\nu)D\psi, \psi\> d\sigma.
 \end{equation}            
 Here $S^{n-1}_{\rho,l}$ is the sphere of radius $\rho$ in the
 asymptotically flat coordinate chart $(M_l, Y_l)$ with outward normal
 vector $\nu$ and volume form $d\sigma$.  Let $v = \psi - \psi_0$. We
 split the integral on the right-hand side of~\eqref{eq:mint} into
 three integrals
\begin{align}\label{eq:b-1}
  \int_{S^{n-1}_{\rho,l}} \<\nabla_{\nu} \psi_0 + c(\nu)D\psi_0,\psi_0\> d\sigma
     + \int_{S^{n-1}_{\rho,l}} \<  \nabla_{\nu} v + c(\nu)\nabla_{\nu}v,\psi_0\> d\sigma
     + \int_{S^{n-1}_{\rho,l}} \<\nabla_{\nu} \psi,v\> d\sigma.
\end{align}
As follows from Witten's proof (see~\cite{bartnik, pt, lp}), the first
integral converges to $\frac{c(n)}{4}\mass(M_l,g)$ as $\rho \to
\infty$. It remains to prove that the other two converge to $0$ as
$\rho \to \infty$.

In the second integral, we rewrite the integrand in an orthonormal
frame $\{e_i\}_{1\leq i\leq n}$ on the end $M_l$,  with $e_1 = \nu$ as 
\begin{align}\label{eq:b-12}
\sum_{j=2}^n \<\frac{1}{2}[c(\nu),c(e_j)]\nabla_{e_j}v ,\psi_0\>  
= &\sum_{j=2}^ne_j\<\frac{1}{2}[c(\nu),c(e_j)] v ,\psi_0\>
-  \sum_{j=2}^n\< \frac{1}{2}[c(\nu)c(\nabla_{e_j}e_j)] v ,\psi_0\>\nonumber \\
&-  \sum_{j=2}^n\<\frac{1}{2}[c(\nabla_{e_j}\nu),c(e_j)]v ,\psi_0\> 
 - \sum_{j=2}^n\<\frac{1}{2}[c(\nu),c(e_j)]v ,\nabla_{e_j}\psi_0\>.
\end{align}
We recognize the first two sums on the right-hand side of~\eqref{eq:b-12} as the divergence (on the sphere) of the vector field $$U:=\sum_{j=2}^n\< \frac{1}{2}[c(\nu),c(e_j)]v ,\psi_0\>e_j.$$ Hence these terms integrate to zero on the sphere. 
The third term on the right-hand side of~\eqref{eq:b-12}
vanishes due to the symmetry of the second fundamental
form $A$ of $S^{n-1}_{\rho,l} \subset M_l$. Explicitly, we have
\begin{equation*}
  \sum_{j=2}^n [c(\nabla_{e_j}\nu),c(e_j)]
   =  \sum_{i,j=2}^n A(e_{j}, e_{i})[  c(e_{i}) ,c(e_{j})]  =0.
\end{equation*}
Thus, it remains to evaluate
\begin{multline}\label{eq:b-3}
 \lim_{\rho\to \infty}\int_{S^{n-1}_{\rho,l}}  \sum_{j=2}^n\< \frac{1}{2} [c(\nu),
 c(e_j)]v, \nabla_{e_j}\psi_0\> d\sigma 
+ \lim_{\rho\to \infty}\int_{S^{n-1}_{\rho,l}} \<\nabla_{\nu} \psi,v\> d\sigma\\
=\lim_{\rho\to \infty}\int_{S^{n-1}_{\rho,l}} 
   \left( \< v,\nabla_{\nu}\psi_0+c(\nu)D\psi_0\>
+  \<\nabla_{\nu} \psi,v\>\right) d\sigma.
\end{multline}
Since $\psi_0$ is asymptotically constant,
$\nabla \psi_0 = \sO(\rho^{-\tau-1})$, and since $\tau>
\frac{n-2}{2}$, $\nabla \psi_0 \in L^2(M\setminus V, S)$. Moreover since $\psi$ is a
Witten spinor, $\frac{v}{\rho} \in L^2(M\setminus V, S)$ and $\nabla
\psi \in L^2 (M_l)$ for all asymptotically flat ends $(M_l,Y_l)$.
Hence, as a function of $\rho$
\[
  \int_{S^{n-1}_{\rho}} \left(\<  v, \nabla_{\nu}\psi_0+c(\nu)D\psi_0\> +
  \<\nabla_{\nu} \psi, v\> \right) d\sigma  \in L^1([T, \infty),\frac{d\rho}{\rho}),
\]
which implies that the integral converges to $0$ as $\rho \to \infty$. 
\end{proof}

\section{Growth estimates near $V$ for spinors in the minimal
  domain}\label{sec:lbe}

We now turn to the study of spinors near $V$. This will occupy the
next two sections. We assume that $(M,g)$ is a nonspin asymptotically
flat Riemannian manifold with nonnegative scalar curvature.  In this
section, we derive growth estimates near each stratum of $V$ for
spinors in the minimal domain. For this we use the Lichnerowicz
formula together with the radial and angular estimates near $V$
obtained in Section~\ref{sec:pe}.  We also give conditions
guaranteeing a spinor is in the minimal domain.

\begin{lemma}
 Let $u\in \Dmin(D)$ with $u$ supported in $W=B_{\e}(Y) \in
 \TRC(V^2)$. Then 
 \begin{equation}\label{eq:Lich}
  \|Du\|^2 \geq \frac{1}{4} \|\frac{u}{r}\|^2 + \frac{1}{4}
  \|\frac{u}{r \ln(\oner)}\|^2 - C_Y\|\frac{u}{ r^{1/2} }\|^2,
 \end{equation}
 for some constant $C_Y >0$, depending on $Y$ and independent of $u$.
\end{lemma}
\begin{proof}
  Since $u \in \Dmin(D)$ and since the scalar curvature is nonnegative,
 the Lichnerowicz formula~\eqref{boch1} gives
 \[
  \|D u\|^2 \geq \|\nabla u\|^2 .
 \]
 The estimate~\eqref{L2-theta} for the angular
 derivative and the estimate~\eqref{eq:poinc2} for the radial
 derivative give
 \begin{equation*}
   \|D u\|^2 \geq \frac{1}{4} \|\frac{u}{r}\|^2 
- C_1   \|\frac{u}{r^{1/2}}\|^2 
                  + \frac{1}{4}\|\frac{u}{r\ln(\frac{1}{r})}\|^2
                  - C_2 \|\frac{u}{r^{1/2}\ln^{1/2}(\frac{1}{r})}\|^2.
   \end{equation*}
   Here the constant $C_1$ and $C_2$ depend on the relatively compact
   subset $Y \subset V^2$. Taking $C_Y: =
   C_1 + C_2$, the result follows.
 \end{proof}

\begin{lemma}
  Let $u\in \Dmin(D)$ be supported in $W = B_{\e}(Y) \in
  \TRC(V^{k_b})$ with $k_b>2$.  Then
 \begin{equation}\label{eq:Lichkb}
   \|Du\|^2 \geq \frac{(k_b-2)^2}{4} \|\frac{u}{r_b}\|^2 -C_Y\|\frac{u}{r_b^{1/2}}\|^2,
 \end{equation}
 where $C_Y>0$ is a positive constant depending on $Y$ and independent
 of $u$.
\end{lemma}
\begin{proof}
  Since $u$ is supported in $W\in \TRC(V^{k_b})$ with
  $k_b>2$, the Lichnerowicz formula and the radial
  estimate~\eqref{eq:poincb} give the desired estimate.
 \end{proof}

As a consequence of the previous two lemmas, we have
\begin{corollary}\label{cor:Dmin}
  Let $u \in \Dmin(D)$. Then for
  any $b$ so that $V^{k_b}$ is nonempty,
 \begin{equation}\label{eq:Dmin}
   \frac{u}{r_b} \in L^2(W\setminus V,  S\lvert_{W\setminus V}) \quad \text{for all} \quad W\in \TRC(V^{k_b}).
 \end{equation}
\end{corollary}
\begin{proof}
 Let $W \in \TRC(V^{k_b})$. From definition~\eqref{eq:TRC} there
 exists a relatively compact subset $Y \subset V^{k_b}$ and $\e <
 \e(Y)$ so that $W= B_{\e}(Y)$. Without loss of generality we can
 assume that $u$ is supported in an open subset $W' = B_{2\e} (Y')$ with
 $Y'$ a slightly bigger relatively compact subset of $V^{k_b}$
 containing $Y$. 

 We consider the estimates~\eqref{eq:Lich} and~\eqref{eq:Lichkb}
 applied to $u$ supported in $W'$. Without loss of generality, we can
 choose $\e$ very small so that in the case $b=2$, the negative term
 in~\eqref{eq:Lich} is absorbed by the second positive term, while in
 the case $b>2$ it is absorbed by a small amount of the positive
 term. In both cases, it follows that $\|\frac{u}{r_b}\|_{L^2}$ is
 finite. In particular, $\frac{u}{r_b} \in L^2 (W\setminus V,
 S\lvert_{W\setminus V})$.
\end{proof}

We also have a useful partial converse result. 

\begin{lemma}\label{lem:conv}
  Let $u \in \Dmax (D)$. If for all $b$ so that $V^{k_b}$ is nonempty
  and for all $W\in \TRC(V^{k_b})$, $\frac{ u }{r_b} \in
  L^2(W\setminus V, S\lvert_{W\setminus V})$, then
  $u\in \Dmin(D).$
\end{lemma}
\begin{proof}
  It suffices to construct a sequence of spinors in $\Dmin(D)$ that
  converge to $u$ in the graph norm of $D$.  Note that since the
  closures of the asymptotically flat ends of the manifold are
  complete, $u \in \Dmax(D)$ implies that $\chi u \in \Dmin(D)$ for
  any smooth, $C^1$-bounded $\chi$ supported away from $V$.
  Therefore, we can construct our desired sequence, by using cutoff
  functions supported in the complement of smaller and smaller
  neighborhoods of $V$.

  Let $\{\gamma_m \}_m$ be a sequence of smooth functions
  $\gamma_m:(0,1)\rightarrow\IR$ satisfying
\begin{enumerate}
\item $0\leq \gamma_m(t)\leq 1$ vanishes in a neighborhood of $0$,
\item $\gamma_m(t)= 1 $ for  $t>2e^{-e^{m}}$, and 
\item $|d\gamma_m|<\frac{1}{t\ln(\frac{1}{t})}$ for all $m$.
\end{enumerate}
Piecewise differentiable $\gamma_m$ satisfying these properties are
easily constructed by letting $\gamma'_m(t)$ be the product of
$\frac{1}{ t\ln(\frac{1}{t})}$ 
and the characteristic function of the
interval $[e^{-e^{m+1}},e^{-e^{m}}]$. Smoothing the characteristic
function of the interval suffices to construct smooth $\gamma_m$. 

Let $2=k_2<k_3<\cdots <k_{p}$ denote the codimensions for which
$V^{k_{j}}$ is nonempty. The highest codimension stratum, $V^{k_p},$
is a compact subset of $M$.  Take $W^{k_p}: = B_{\e_p}(V^{k_p}) \in
\TRC(V^2)$ for some $\e_p>0$.  Extend $r_p$ from $W^{k_p}$ to $M$ as a
smooth, nonnegative eventually constant function, bounded above by
$\frac{1}{2}$. Then $\{\gamma_m(r_{p})u\}_m$ converges in $L^2$ to
$u$. By hypothesis, $\frac{|u|}{r_{p}}\in L^2(W^{k_p})$, and therefore
$|[D,\gamma_m(r_p)] u|\leq \frac{|u|}{r_p\ln(\frac{1}{r_p})}\in L^2.$
Since
\[
  D (\gamma_m(r_p) u) = [D,\gamma_m(r_p)] u + \gamma_m(r_p) Du,
\]
and $\gamma_m\rightarrow 1$ pointwise, Lebesgue's dominated convergence theorem implies $[D,
\gamma_m(r_p)]u$ converges to $0$ in $L^2$-norm, and
$D(\gamma_m(r_{p})u)$ converges in $L^2$ to $Du$ as
$m\rightarrow\infty$.

Consider now $\gamma_m(r_p) u$ for $m$ fixed. By construction, this
vanishes on $B_{\epsilon(m,p)} (V^{k_p})\subset W^{k_p}$ for some
$\epsilon(m,p)>0$. Observe that for any $\epsilon >0$ such that
$B_{\epsilon }( V^{k_{p-1}}\setminus
B_{\frac{1}{2}\epsilon(m,p)}(V^{k_p}))\in \TRC(V^{k_{p-1}})$,  the set
$B_{\epsilon }( V^{k_{p-1}}\setminus B_{\frac{1}{2}\epsilon(m,p)}(V^{k_p}))
\cup B_{\epsilon(m,p)} (V^{k_p})$ is a neighborhood of $V^{k_{p-1}}$.
Extend $r_{p-1}$ from $B_{\epsilon }( V^{k_{p-1}}\setminus
B_{\frac{1}{2}\epsilon(m,p)}(V^{k_p})) $ as a smooth, nonnegative,
eventually constant function on $M$, bounded above by $\frac{1}{2}$.
Then each of the spinors
$\{\gamma_{\mu}(r_{p-1})\gamma_m(r_p)u\}_{\mu}$ is supported outside a
neighborhood of $\overline{V^{k_{p-1}}}$. For $\epsilon$ fixed and
$\mu$ sufficiently large, $[D,\gamma_\mu(r_{p-1})]\gamma_m(r_p) u$ is
supported in $B_{\epsilon }( V^{k_{p-1}}\setminus
B_{\frac{1}{2}\epsilon(m,b)} (V^{k_p}))$ and satisfies
$$\left|[D,\gamma_\mu(r_{p-1})]\gamma_m(r_p) u\right|\leq \frac{|\gamma_m(r_p) u|}{r_{p-1}\ln(\frac{1}{r_{p-1} })}
 \in L^2(B_{\epsilon }( V^{k_{p-1}}\setminus B_{\frac{1}{2}\epsilon(m,b)}(V^{k_p}))).$$
 Hence we may again apply the Lebesgue dominated convergence theorem to conclude that  
$\gamma_\mu(r_{p-1})\gamma_m(r_p) u\stackrel{L^2}{\rightarrow}\gamma_m(r_p) u$ and 
$D(\gamma_\mu(r_{p-1})\gamma_m(r_p)u)\stackrel{L^2}{\rightarrow}D(\gamma_m(r_p) u)$
as $\mu \to \infty$.  
 
Proceeding inductively backwards on $k_b$, we construct a sequence of
 spinors  in the minimal domain of
  $D$ which converge to $u$ in the graph norm of $D$. Hence $u \in
\Dmin(D)$.
\end{proof}

\section{Improved integral estimates for harmonic spinors near $V$}\label{sec:iie}
For harmonic spinors, we can improve the growth estimates of Section \ref{sec:lbe}.  
In this section, we use the lower bound estimates derived in
Section~\ref{sec:lbe} to obtain weighted integral estimates for
spinors $u$ in the minimal domain of $D$ satisfying $D^2 u =0$ in a
neighborhood of $V$, and for the spinors $v = Du$. We conclude with a
result which gives a sufficient condition for $v$ to be in the minimal
domain of $D$.

Our main tool is a technique of Agmon,~\cite{ag}, which we
  now present.  If $u \in \Dmin
(D)$, then $f u \in \Dmin(D)$ for any bounded, piecewise
differentiable function $f$, with bounded derivative. Hence we have
the following identity
\begin{equation}\label{eq:u1}
   (D^2 u, f^2 u)
    = \|D(f u)\|^2 -\| [D, f] u\|^2.
\end{equation}
Combining this identity with the Lichnerowicz formula, we have 
\begin{equation}\label{eq:u01}
   (D^2 u, f^2 u)
    = \|\nabla(f u)\|^2 +\frac{1}{4}(R fu, fu) -\| [D,f] u\|^2.
\end{equation}
We refer to either of the expressions (\ref{eq:u1}) or (\ref{eq:u01})
as the {\em Agmon identity}. In this section we apply this identity,
together with the radial and angular estimates which we derived in
Section~\ref{sec:pe}, to obtain our desired estimates.

\subsection{Estimates near the codimension $2$ stratum of $V$}
We prove first estimates near the codimension $2$ stratum $V^2$.
\begin{lemma}\label{lem:u_est}
 If $u \in \Dmin (D)$ with $D^2 u =0$ in a neighborhood  of $\overline{V^2}$, then
\begin{equation}\label{eq:u_est}
 \frac{u}{r^{3/2} \ln^{1/2+a}(\frac{1}{r})} \in L^2(W\setminus V,
 S\lvert_{W\setminus V})
\end{equation}
 for all $a>0$ and for all $W \in \TRC(V^2)$.
\end{lemma}
\begin{proof}
 From the definition~\eqref{eq:TRC}, given $W\in \TRC(V^2)$,  there exists a
 relatively compact subset $Y$ of $V^2$ and $\e< \e(Y)$ so that $W=
 B_{\e}(Y)$.
 Without loss of generality we can assume that $D^2 u=0$ in an open
 set $W'\in \TRC(V^2)$ containing $\widebar W$.

Let $\zeta\in C_0^\infty(W') $ with $0\leq \zeta \leq 1$,
$\zeta\equiv 1$ on $W$.
%
 Given $a>0$, we define the sequence of functions
 \[
   \mu_m(r): = 
   \begin{cases}
      r^{-1/2}\ln^{-a}(\frac{1}{r})& \text{for $r> \frac{1}{m}$}\\
      m^{1/2}\ln^{-a}(m)  & \text{for $r\leq \frac{1}{m}$}.
   \end{cases}
 \]
 Then, the Agmon identity~\eqref{eq:u1} applied to $\mu_m\zeta$,
\[
  0= (D^2u, \mu_m^2\zeta^2 u) = \|D(\mu_m\zeta u)\|^2- \|[D,\mu_m\zeta]u\|^2,
\]
together with the estimate~\eqref{eq:Lich} give
 \begin{equation}\label{qe2}
    0  \geq \frac{1}{4} \|\frac{\mu_m \zeta u}{r}\|^2 
         + \frac{1}{4}\|\frac{\mu_m \zeta u}{r\ln (\frac{1}{r})}\|^2 
         - \|[D,\mu_m\zeta] u\|^2 
         - C_{W'}\|\frac{\mu_m\zeta u}{r^{1/2}}\|^2,
 \end{equation}
 with the constant $C_{W'}$ depending on $W'$.
Expanding $[D,\mu_m\zeta]u$, we see that for large enough $m$,
\begin{multline*}
\|[D,\mu_m\zeta]u\|^2 
\leq 
\frac{1}{4}\|\frac{\chi_m\mu_m\zeta u}{r}\|^2
-a\|\frac{\chi_m\mu_m\zeta u}{r\ln^{1/2}(\frac{1}{r})}\|^2
+a^2\|\frac{\chi_m\mu_m\zeta u}{r\ln(\frac{1}{r})}\|^2\\
 + \|d\zeta\|_{L^{\infty}(W')}^2\|\frac{u}{r^{1/2}}\|_{L^2(W')}^2 + C_{\zeta,\e}\|\frac{u}{r}\|_{L^2(W')}^2,
\end{multline*}
with $\chi_m$ the characteristic function of the interval
$r>\frac{1}{m}$ and $C_{\zeta, \e}$ a constant depending on
$\|d\zeta\|_{L^{\infty}(W')}$.
 By Corollary \ref{cor:Dmin}, $\|\frac{u}{r}\|_{L^2(W')}^2$ is
 finite. Since $\frac{\mu_m\zeta}{r^{1/2}} < \frac{1}{r}$, plugging this
 expansion back into (\ref{qe2}) gives
\begin{multline*}
  (C_{\zeta,\e}+C_{W'})\|\frac{u}{r}\|_{L^2(W')}^2 + \|d\zeta\|_{L^{\infty}(W')}^2\|\frac{u}{r^{1/2}}\|_{L^2(W')}^2\\
\geq 
a\|\frac{\chi_m\mu_m\zeta u}{r\ln^{1/2}(\frac{1}{r})}\|^2 
+ \frac{1}{4}  \|\frac{\chi_m\mu_m \zeta u}{r\ln(\frac{1}{r})}\|^2
 -a^2\|\frac{\chi_m\mu_m\zeta u}{r\ln(\frac{1}{r})}\|^2.
\end{multline*}
 Taking the limit as $m \to \infty$, it follows that
 $$\frac{u}{r^{3/2} \ln^{1/2+a}(\frac{1}{r})} \in L^2(W\ \setminus V,
  S\lvert_{W\setminus V})$$
for all $a>0$.
\end{proof}
\begin{lemma}\label{lem:v_est}
 Let $v = Du$ with $u$ as in the previous lemma. Then
 \begin{equation}\label{eq:v_est}
  \frac{v}{r^{1/2}\ln^{1/2+a}(\frac{1}{r})} \in L^2(W \setminus V,
   S\lvert_{V\setminus W})
 \end{equation}
 for all $a>0$ and for all $W \in \TRC(V^2)$.
\end{lemma}
\begin{proof}
 Let $W'$ and $\zeta$ be as in the proof of Lemma~\ref{lem:u_est}.
 We modify the sequence of functions in the proof of the preceding lemma to 
 \[
   \mu_m (r):
   = \begin{cases}
        r^{-1/2} \ln^{-1/2-a} (\frac{1}{r}) & \text{for $r>\frac{1}{m}$}\\
        m^{1/2} \ln^{-1/2-a} (m) & \text{for $r \leq \frac{1}{m}$}.
     \end{cases}
 \]
The Agmon identity~\eqref{eq:u1} applied to $\mu_m\zeta$ gives
 \begin{equation}\label{eq:09.29.2012}
   \| D(\mu_m\zeta u)\|^2 = \|[D,\mu_m\zeta] u\|^2.
 \end{equation}
The right-hand side is bounded by
 \begin{multline*}
   \|[D,\mu_m\zeta]u\|^2 
  \leq 
  \frac{1}{4}\|\frac{\chi_m\mu_m\zeta u}{r}\|^2
   - (\frac{1}{2} + a) \|\frac{\chi_m\mu_m\zeta u}{r\ln^{1/2}(\frac{1}{r})}\|^2
   + (\frac{1}{2}+a)^2\|\frac{\chi_m\mu_m\zeta u}{r\ln(\frac{1}{r})}\|^2\\
 + \|d\zeta\|_{L^{\infty}(W')}^2\|\frac{u}{r^{1/2}}\|_{L^2(W')}^2 
 + C_{\zeta,\e}\|\frac{u}{r}\|_{L^2(W')}^2,
 \end{multline*}
 with $\chi_m$ the characteristic function of the interval
 $r>\frac{1}{m}$ and $C_{\zeta,\e}$ a constant depending on
 $\|d\zeta\|_{L^2(W')}$ and $\e$. Since by Lemma~\ref{lem:u_est} we
 have $\frac{|u|}{r^{3/2} \ln^{1/2 +a} (\frac{1}{r})} \in L^2(W')$, it
 follows that $\| [D, \mu_m\zeta] u\|^2$ is uniformly bounded as
 $m\rightarrow\infty$. Hence, taking the limit as $m\rightarrow\infty$
 in~\eqref{eq:09.29.2012} we obtain
 $$
   D (\frac{\zeta u}{ r^{1/2}\ln^{1/2+a}(\frac{1}{r})}) \in
   L^2(W\setminus V, S\lvert_{W\setminus V}).
 $$
 Since the function $\mu$ defined by $\mu(r): = r^{-1/2}
 \ln^{-1/2-a}(\frac{1}{r})$ satisfies $|d\mu| \leq \frac{C(a)}{r^{3/2}
 \ln^{1/2+a}(\frac{1}{r})}$ for a positive constant $C(a)$ depending
 on $a$, it follows also from Lemma~\ref{lem:u_est} that 
 \[
   [D, \frac{\zeta}{r^{1/2} \ln^{1/2+a}(\frac{1}{r})}]u \in
   L^2(W\setminus V, S\lvert_{W\setminus V}).
 \]
 Writing
 $$
   D (\frac{\zeta u}{r^{1/2}\ln^{1/2+a}(\frac{1}{r})}) 
   =  \frac{\zeta Du}{r^{1/2}\ln^{1/2+a}(\frac{1}{r})} 
+ [D, \frac{\zeta(r)}{r^{1/2} \ln^{1/2+a}(\frac{1}{r})}] u,
 $$ expresses $\frac{\zeta
v}{r^{1/2}\ln^{1/2+a}(\frac{1}{r})}$ as the difference of
$L^2$-sections. Hence it is square integrable.
\end{proof}

\subsection{Estimates near the higher codimension strata of $V$}
Now we prove similar estimates to~\eqref{eq:u_est}
and~\eqref{eq:v_est} near the higher
codimension strata of $V$. These estimates will be in terms of $r_b$,
the distance to the stratum $V^{k_b}$. 

\begin{lemma}\label{lem:u_estb}
  If $u \in \Dmin (D)$ and $D^2 u =0$ in a neighborhood of $\overline{V^{k_b}}$ in $M\setminus V$
  with $k_b >2$, then 
\begin{equation}\label{eq:u_estb}
 \frac{u}{r_b^{k_b/2}\ln^{1/2+a}(\frac{1}{r_b})} \in L^2(W \setminus V ,  S\lvert_{W\setminus V})
\end{equation}
 for all $W \in \TRC(V^{k_b})$ and for all $a>0$.
\end{lemma}
\begin{proof}
The proof is similar to the proof of Lemma \ref{lem:u_est},
with~\eqref{eq:Lichkb} replacing \eqref{eq:Lich}.
It consists of two steps. In the first step we show that 
\begin{equation}\label{eq:step1}
 \frac{u}{r_b^{\alpha}}  \in L^2(W \setminus V ,  S\lvert_{W\setminus V})
\end{equation}
for all $\alpha < k_b/2$ and for all $W \in \TRC(V^{k_{b}}).$ In
the second step we prove~\eqref{eq:u_estb}.

{\em Step 1:}
Let $W \in \TRC(V^{k_b})$. By~\eqref{eq:TRC}, this means that there exists $Y$ a
relatively compact subset of $V^{k_b}$ and $\e<\e(Y)$ so that $W= B_{\e}(Y)$.
Without loss of generality, we can assume that $D^2 u =0$ in an open
set $W' \in \TRC(V^{k_b})$ containing $\widebar{W}$.

Let $\zeta\in C_0^\infty(W') $ with $0\leq \zeta \leq 1$ and
$\zeta\equiv 1$ on $W$.
Then the Agmon identity~\eqref{eq:u1} applied to the sequence of functions
$\mu_m\zeta$, with $\mu_m$ defined as
\[
  \mu_m(r_b):
   = \begin{cases}
         r_b^{-\alpha} & \text{for
         $r_b>\frac{1}{m}$}\\
         m^{\alpha} & \text{for $r_b \leq \frac{1}{m}$},
     \end{cases}
\]
gives
\[
  0 = \|D(\mu_m\zeta u)\|^2 - \|[D,\mu_m\zeta]u\|^2.
\]
Applying~\eqref{eq:Lichkb} to the first term, we obtain
\begin{equation}\label{eq:ukb}
  0 \geq \frac{(k_b-2)^2}{4}
  \|\frac{\mu_m \zeta u}{r_b}\|^2 - C_{W'}\|\frac{\mu_m \zeta u}{r_b^{1/2}}\|^2 
   - \|[D,\mu_m\zeta ] u\|^2
\end{equation}
with $C_{W'}$ a constant depending on $W'$.
Expanding $[D,\mu_m\zeta]u$, we have
\[
  \|[D,\mu_m\zeta]u\|^2\leq 
   \alpha^2 \|\frac{\mu_m \zeta u}{r_b}\|^2 
   + \|d\zeta\|_{L^{\infty}(W')}^2 \|\mu_m   u\|_{L^2(W')}^2 
   + 2\alpha \|d\zeta\|_{L^{\infty}(W')} \|\frac{\chi_m \chi_{d\zeta} u}{r_b^{\alpha+1/2}}\|^2.
\] 
Here $\chi_m$ denotes the characteristic function of $r> \frac{1}{m}$
and $\chi_{d\zeta}$ the characteristic function of the support of
$d\zeta$.
Plugging this into~\eqref{eq:ukb}, we obtain
\begin{multline*}
  \|d\zeta\|_{L^{\infty}(W')}^2 \|\mu_m   u\|_{L^2(W')}^2 
 + 2\alpha  \|d\zeta\|_{L^{\infty}(W')} \|\frac{\chi_m \chi_{d\zeta} u}{r_b^{\alpha+\frac{1}{2}}}\|^2 
 + C_{W'}\|\frac{\mu_m u}{r_b^{1/2}}\|_{L^2(W')}^2 \\
 \geq  \frac{(k_b-2)^2}{4}
  \|\frac{\mu_m \zeta u}{r_b}\|^2 - \alpha^2 \|\frac{\mu_m\zeta  u}{r_b}\|^2 .
\end{multline*}
Assuming that $\frac{u}{r_b^{\alpha + 1/2}} \in
L^2(W' \setminus V, S\lvert_{W'\setminus V} )$, it
follows that the left-hand side above is bounded by a constant
independent of $m$. Hence we can take the limit as $m\to \infty$ and conclude
\[
  \frac{u}{r_b^{\alpha+1}}  \in L^2(W \setminus V , S\lvert_{W\setminus V})
\]
as long as $\alpha < (k_b-2)/2$.

By Corollary~\ref{cor:Dmin}, we know that $\frac{u}{r_b} \in L^2(W', S\lvert_{W'\setminus V})$.
Starting from here and bootstrapping using the above argument, we
obtain~\eqref{eq:step1}.  

{\em  Step 2:}
We have the same argument as in the proof of the first step, only that
now we take
\[
  \mu_m (r_b):
   = \begin{cases}
        r_b^{-(k_b-2)/2} \ln^{-a}(\frac{1}{r_b}) & \text{for
         $r_b>\frac{1}{m}$}\\
        m^{(k_b-2)/2} \ln^{-a}(m)  & \text{for $r_b \leq \frac{1}{m}$},
     \end{cases}
\]
for $a>0$.
This time, expanding $[D,\mu_m\zeta]u$, we obtain
\begin{align*}
 \|[D,\mu_m\zeta]u\|^2
 &= \|[D,\mu_m]\zeta u\|^2 + 2([D,\mu_m]\zeta u, \mu_m[D,\zeta] u)
 + \|\mu_m[D,\zeta] u\|^2\\
 & \leq
\frac{(k_b-2)^2}{4} \|\frac{ \mu_m \zeta u}{r_b}\|^2 
- (k_b-2) a \|\frac{\chi_m\mu_m \zeta u}{r_b \ln^{1/2}(\frac{1}{r_b})}\|^2
+ a^2 \|\frac{\mu_m \zeta u}{r_b \ln(\frac{1}{r_b})}\|^2 \\
 & \quad
+ \frac{k_b-2}{2}\|d\zeta\|_{L^{\infty}(W')} \|\frac{\mu_m
 \chi_m\zeta^{1/2}u}{r_b^{1/2}}\|^2
+ a \|d\zeta\|_{L^{\infty}(W')} \|\frac{\mu_m
 \chi_m\zeta^{1/2}u}{r_b^{1/2}\ln^{1/2}(\frac{1}{r_b})}\|^2\\
& \quad
+  \|d\zeta\|_{L^{\infty}(W')}^2 \|\mu_m \chi_{d\zeta} u\|^2 .
\end{align*}
Plugging this into~\eqref{eq:ukb}, we obtain
\begin{align*}
(k_b-2) a\  &\|\frac{\chi_m\mu_m \zeta u}{r_b \ln^{1/2}(\frac{1}{r_b})}\|^2
 - a^2 \|\frac{\mu_m \zeta u}{r_b  \ln(\frac{1}{r_b})}\|^2 \\
 & \leq  
\frac{k_b-2}{2}\|d\zeta\|_{L^{\infty}(W')} \|\frac{\mu_m
 \chi_m\zeta^{1/2}u}{r_b^{1/2}}\|^2
+ a \|d\zeta\|_{L^{\infty}(W')} \|\frac{\mu_m
 \chi_m\zeta^{1/2}u}{r_b^{1/2}\ln^{1/2}(\frac{1}{r_b})}\|^2\\
 & \quad +  \|d\zeta\|_{L^{\infty}(W')}^2 \|\mu_m \chi_{d\zeta} u\|^2 
+ C_{W'}\|\frac{\mu_m\zeta u}{r_b^{1/2}}\|^2 .
\end{align*}
By the first step, we know that the right-hand side is bounded by a
constant independent of $m$. Therefore we can take limit as $m\to
\infty$ and conclude
\[
   \frac{u}{r_b^{k_b/2} \ln^{1/2+a}(\frac{1}{r_b})} \in L^2(W \setminus V , S\lvert_{W\setminus V})
\]
for all $a>0$.
\end{proof}
For $v = Du$ we have now the following estimate near the higher
codimension strata.
\begin{lemma}\label{lem:v_estb}
 Let $v = Du$, with $u$ as in Lemma~\ref{lem:u_estb}. Then
 \begin{equation}\label{eq:v_estb}
  \frac{v}{r_b^{(k_b-2)/2}\ln^{1/2+a}(\frac{1}{r_b})} \in L^2(W \setminus V , S\lvert_{W\setminus V})
 \end{equation}
 for all $W \in \TRC(V^{k_b})$ and all $a>0$.
\end{lemma}
The proof is similar to the proof of Lemma~\ref{lem:v_est}, using the
result of Lemma~\ref{lem:u_estb}.

\subsection{A sufficient condition for $v$ to be in $\Dmin(D)$}

The following result gives a sufficient condition for the spinor $v =
Du$ to be in the minimal domain of $D$.
\begin{proposition}\label{prop:v_min}
 Let $v = Du$ with $u \in \Dmin(D)$ and $D^2 u = 0$ in a
 neighborhood of $V$. 
If $\frac{v}{r^{1/2}\ln^{1/2}(\frac{1}{r})}\in L^2(W \setminus V , S\lvert_{W\setminus V})$ for all  $W\in TRC(V^2),$
then  $\eta v \in \Dmin(D)$, for all
  $\eta \in \sC^{\infty}_0(M)$. 
\end{proposition}
Note that the hypothesis is stronger than the estimate we obtained in
Lemma~\ref{lem:v_est}; in particular,  it rules out $|v|$ growing like
$r^{-1/2}$ near $V^2$.
\begin{proof}
We show that $v$ satisfies the hypotheses of Lemma~\ref{lem:conv}.  By
Lemma \ref{lem:v_estb} we have the desired estimates near the higher
codimension strata $V^{k_b}$ with $k_b >2$. The only difficult stratum
is $V^2$.
Let $W\in \TRC(V^2),$ and let $W' \in \TRC(V^{2})$ containing
$\widebar{W}$.  Let $\zeta\in C_0^\infty(W') $ with $0\leq \zeta \leq
1$ and $\zeta\equiv 1$ on $W$. Let $\{\gamma_m(r)\}_m$ be the sequence of
cutoff functions described in Lemma~\ref{lem:conv}. Define another 
sequence of bounded, piecewise differentiable functions by
\[
  \mu_j(r): 
   = \begin{cases}
       j^{1/2}r^{1/2} & \text{for    $r<\frac{1}{j}$}\\
       1  & \text{for $r \geq \frac{1}{j}$},
     \end{cases}
\]
and let  $\chi_j$ denote the characteristic function of the support of
$d\mu_j$. 
The estimate~(\ref{eq:Lich}) applied to $\gamma_m\mu_j \zeta v\in
\Dmin(D)$ gives
\begin{equation}\label{recap}
\|D(\gamma_m\mu_j\zeta v)\|^2  \geq
\frac{1}{4}\|\frac{\gamma_m\mu_j\zeta v}{r}\|^2 
- C_{W'}\|\frac{\gamma_m\mu_j\zeta  v}{r^{1/2}}\|^2.
\end{equation}
On the other hand, since $Dv =0$ on the support of $\zeta$, 
\begin{align*}
  \|D(\gamma_m\mu_j \zeta  v)\|^2 
  & = \|[D,\gamma_m]\mu_j \zeta v +\frac{\chi_j}{2r}c(e_r)\gamma_m\mu_j\zeta v+ \gamma_m\mu_j[D,\zeta] v\|^2\\
  & = \|[D,\gamma_m]\mu_j\zeta v \|^2 
      +\frac{1}{4}\|\frac{\gamma_m\chi_j\mu_j\zeta  v}{r}\|^2
      +\| \gamma_m\mu_j[D,\zeta] v\|^2 \\
  & \quad 
      +2([D,\gamma_m]\mu_j\zeta  v , \gamma_m\mu_j[D,\zeta] v+\frac{\chi_j}{2r}c(e_r)\gamma_m\mu_j\zeta v)
      +2( \gamma_m\mu_j[D,\zeta]v,\frac{\chi_j}{2r}c(e_r)\gamma_m\mu_j\zeta v). 
\end{align*}
Since $\frac{v}{r^{1/2}\ln^{1/2}(\frac{1}{r})}\in L^2(W')$ by hypothesis, the Lebesgue
dominated convergence theorem implies that
\[
 \lim_{m\to \infty}\left(
  ( [D,\gamma_m]\mu_j\zeta v , \gamma_m\mu_j[D,\zeta]v+\frac{\chi_j}{2r}c(e_r)\gamma_m\mu_j\zeta v)
+ \|[D,\gamma_m]\mu_j\zeta v \|^2
 \right) =0.
\]
Hence, substituting the above expression for $  \|D(\gamma_m\mu_j
\zeta  v)\|^2$ into~\eqref{recap} 
and taking the limit as $m\rightarrow\infty$ yields
\[
 \|\mu_j [D,\zeta] v\|^2
 +2( \mu_j[D,\zeta] v,\frac{\chi_j}{2r}c(e_r)\mu_j\zeta v)
 \geq 
\frac{1}{4}\|\frac{(1-\chi_j)\mu_j \zeta v}{r}\|^2 - C_{W'}\|\frac{ \mu_j\zeta v}{r^{1/2}}\|^2.
\]
Taking now the limit as $j\rightarrow\infty$ gives 
$$\frac{1}{4}\|\frac{ \zeta  v}{r}\|^2 \leq C_{W'}\|\frac{ \zeta v}{r^{1/2}}\|^2
 +\|[D,\zeta]  v\|^2, $$
and therefore $\frac{ v}{r}\in L^2(W,  S\lvert_{W\setminus V}).$
  We may now apply Lemma~\ref{lem:conv} to deduce $v\in \Dmin(D)$. 
\end{proof}

\section{Estimates on the asymptotically flat ends of $M$}\label{sec:af}
In this section, we show that if $u \in H^{1}_{\rho}(M\setminus V, S)$
is a smooth spinor with $D^2u$ satisfying certain decay estimates,
then $u$ satisfies weighted integral estimates on the
asymptotically flat ends of $M$. These results are used to
prove Theorem A. Such estimates are not new (see~\cite{bartnik, pt}), but we choose to obtain them here using
Agmon's identity in a manner similar to the proof in Section~\ref{sec:iie} of our weighted integral estimates near $V$.

We first record the extension of Agmon's
identity~\eqref{eq:u1} from $\Dmin(D)$ to $H^1_{\rho}(M\setminus V,S)$.
\begin{lemma}\label{lem:agmon_af}
  Let $u \in H^1_{\rho}(M\setminus V, S)$ be a smooth spinor. Then for
  any bounded, piecewise differentiable function $f$, with bounded
  derivatives, satisfying $|df| u  \in L^2(M\setminus V,S)$  and the pointwise inner-product $\langle D^2 u, f^2 u \rangle(x) \in L^1
  (M\setminus V)$, we have the
  following identities in $L^2$-norm on $M\setminus V$:
 \begin{equation}\label{eq:u2}
   (D^2 u, f^2 u) = \|D(fu)\|^2 - \|[D,f]u\|^2 = \|\nabla(fu)\|^2
   +\frac{1}{4}(Rfu, fu)- \|[D,f]u\|^2.
  \end{equation}
\end{lemma}
\begin{proof}
Since $u \in H^1_{\rho}(M\setminus V, S)$,  $\eta u \in  \Dmin(D)$ for all $\eta \in \sC_0^{\infty}(M).$
Choose a sequence of smooth compactly supported functions $\eta_j \in
\sC^{\infty}_0(M)$ so that $0\leq \eta_j \leq 1$ is supported in the region where
$\rho(x)\leq 2j$, is identically $1$ where
$\rho(x) \leq j$, and satisfies $|d\eta_j|\leq \frac{2}{\rho}$.
Then by Agmon's identity~\eqref{eq:u1}, we have 
\begin{align*}
(D^2 u, f^2 \eta_j^2u) 
  & = \|D(\eta_jfu)\|^2 - \|[D,\eta_jf]u\|^2 \\
  & = \|\eta_j D(fu) + [D, \eta_j] fu\|^2 - \|[D,\eta_j f] u\|^2.
\end{align*}
 Since $\frac{u}{\rho}$ and $D(fu)$ are both in $L^2(M\setminus V,S)$,
 and since $|d\eta_j|< \frac{2}{\rho}$, the Lebesgue dominated
 convergence theorem gives
\[
 2 (\eta_j D(fu), [D,\eta_j]fu) + \|[D,\eta_j] fu\|^2 \to 0 
 \quad \text{as $j\to \infty$}.
\]
Hence, we may thus take the limit as $j\to\infty$ of the preceding equalities to get 
$$(D^2 u, f^2  u) = \|D( fu)\|^2 - \|[D, f]u\|^2.$$
The second part of the identity~\eqref{eq:u2} follows similarly, using the Lichnerowicz formula~\eqref{boch1}.
\end{proof}
We first illustrate the use of this Agmon identity to derive, in our 
context, the standard result that $\rho^{a+1} D^2 u \in L^2$ implies
$\rho^{a-1} u \in L^2$.
\begin{proposition}\label{udecay}
 Let $(M^n,g)$ be a nonspin Riemannian manifold which is asymptotically
 flat of order $\tau >0$ and has nonnegative scalar curvature.
 Let $u \in H^1_{\rho} (M\setminus V, S)$ be a smooth spinor.  If
 there exists a positive real number $0\leq a < \frac{ n-2 }{2}$ so that
 $\rho^{a+1} D^2u \in L^2(M_l, S\lvert_{M_l})$ for each asymptotically
 flat end $M_l$ of $M$, then
 \begin{equation}\label{eq:udecay} 
  \rho^{a-1} u \in L^2(M\setminus V, S).
 \end{equation}
\end{proposition}
\begin{proof}
 Fix an asymptotically flat end $M_l$ of $M$. 
 Let $\eta$ be a cutoff function, $0\leq\eta\leq 1$,  supported in the region of $M_l$
 where $\rho(x) \geq L$ and identically $1$ in the region
 where $\rho(x) \geq 2L$, for some $L >T$ to be fixed later.
For each positive integer $m$ consider 
\[
  \mu_m(\rho): =
  \begin{cases}
     \rho^a & \text{if $\rho \leq m$},\\
     m^a & \text{if $\rho >m$}.
  \end{cases}
\]
Then~\eqref{eq:u2} applied to the functions $\mu_m \eta$ gives
\[
  (D^2u,\mu_m^2\eta^2u) = \|D(\mu_m\eta u)\|^2 - \||[D,\mu_m\eta]u\|^2.
\]
Since we assume that the scalar curvature is nonnegative 
\[
  ( D^2u,\mu_m^2\eta^2u) \geq \|\nabla(\mu_m\eta u)\|^2 - \|[D,\mu_m\eta]u\|^2.
\]
Applying the Hardy inequality~\eqref{eq:af_poinc} to the term
$\|\nabla (\mu_m \eta u)\|$ on the
asymptotically flat end $M_l$ of
$M$, we obtain
\begin{equation}\label{eq:udecay-1}
 (D^2u,\mu_m^2\eta^2u)
  \geq \frac{(n-2)^2}{4}  \|\frac{\mu_m\eta u}{\rho}\|^2 
   - \|[D,\mu_m\eta]u\|^2
   - C_l\|\frac{\mu_m\eta  u}{\rho^{1+\frac{\tau}{2}}}\|^2.
\end{equation}
We expand the term $\|[D,\mu_m\eta]u\|^2$ into
\begin{align*}
 \|[D,\mu_m\eta]u\|^2
  & = \|[D,\mu_m]\eta u\|^2+ 2 ( [D,\mu_m]\eta u, \mu_m[D, \eta] u) +
  \|\mu_m [D,\eta]  u\|^2\\
  & \leq \|[D,\mu_m]\eta u\|^2 + C_{\eta, a} \|\chi_{d\eta}u\|^2\\
  & =  a^2\|\chi_m \frac{\mu_m\eta u}{\rho}\|^2 
     + C_{\eta, a} \|\chi_{d\eta}u\|^2.
\end{align*}
for $m$ large enough. Here $\chi_m$ is the characteristic function of
the set $\rho(x) \leq m$, $\chi_{d\eta}$ is the characteristic
function of the support of $d\eta$, while $C_{\eta,a}>0$ is a constant
depending on $\eta$ and $a$, and independent of $u$ and $m$.
Plugging this into~\eqref{eq:udecay-1}
we have
\begin{align*}\label{eq:udecay-2}
  &\|\rho^{1+a} D^2u\|\  \|\eta\frac{\mu_m}{\rho}u\|
    \geq ( \mu_m\rho D^2u,\frac{\mu_m}{\rho} \eta^2u) \\
  & \geq \left(\frac{(n-2)^2}{4} -a^2\right) \|\chi_m\frac{\mu_m\eta u}{\rho}\|^2 
        + \frac{(n-2)^2}{4} \|(1-\chi_m)\frac{\mu_m\eta u}{\rho}\|^2 
        - C_{\eta, a} \|\chi_{d\eta}u\|^2- C_l\|\frac{\mu_m\eta  u}{\rho^{1+\frac{\tau}{2}}}\|^2.
\end{align*}
Hence  for $a\in [0,\frac{ n-2 }{2})$  fixed, $  \|\chi_m\frac{\mu_m\eta u}{\rho}\| $ is uniformly bounded as $m\rightarrow\infty$. Therefore, $\rho^{a-1} \eta u \in L^2 (M\setminus V, S)$.
\end{proof}
In the case when $D^2 u$ vanishes on the asymptotically flat ends, the
previous result is simply the integral version of the
 fact that $L^2$ spinors $u$ that satisfy $D^2u =0$ at infinity, decay pointwise like
$\sO(\rho^{-n+2})$ on the ends, and thus 
$\rho^{\frac{n}{2}-2-\e}u\in L^2(M\setminus V, S)$ for all $\e>0$.
This result can be sharpened for spinors that satisfy $Du =0$ at infinity; such spinors decay
like $\sO(\rho^{-n+1})$ on the ends, and thus $\rho^{\frac{n}{2} -
  1-\e}u \in L^2(M\setminus V, S)$.
 We provide the integral version of this estimate too.

\begin{proposition}\label{udecay2}
 Let $(M,g)$ be a nonspin Riemannian manifold which is asymptotically
 flat of order $\tau >0 $ and has nonnegative scalar curvature.
 Assume $u$ is a smooth spinor so that $Du$ is compactly supported and
 $\rho^{-b} u \in L^2(M\setminus V, S)$ for some $b\in [0, \frac{n-2}{2})$.
 Then,  
 \[
   \rho^{\frac{n}{2}- 1-\e}u \in L^2 (M\setminus V, S)
 \]
 for all $\e>0$.
\end{proposition}
\begin{proof}
  Let $\eta$ be a radial cutoff function supported in the region of
  $M_l$  where   $\rho(x) \geq L$ and identically $1$ in the region  where $\rho(x) > 2L$.
  We choose $L >T$ large enough so that on the support of $\eta$ we have $Du=0$. 
 For a fixed $a > b$ and for each positive integer $m$, we
  consider the sequence of functions
 \[
  \mu_m(\rho): =
  \begin{cases}
     \rho^a & \text{if $\rho \leq m$},\\
    \frac{m^{b+a}}{\rho^b} & \text{if $\rho >m$}.
  \end{cases}
 \]
 Since $Du=0$ on the support of $\eta$, it follows that 
 \begin{equation*}
   D(\mu_m\eta u) = [D,\mu_m\eta]u \in L^2(M\setminus V, S).
 \end{equation*}
 Since by construction $\mu_m \eta u \in L^2(M\setminus V, S)$, it
 follows that $\mu_m \eta u \in \Dmin(D)\subset H^1_{\rho}(M\setminus
 V)$.  Therefore we can apply both the Lichnerowicz
 formula~\eqref{boch1} and the Hardy
 inequality~\eqref{eq:af_poinc} to this sequence of spinors.

  The Lichnerowicz formula gives
 \begin{equation}\label{eq:ud2-lich}
  \| [D,\mu_m\eta] u\|^2 \geq \|\nabla (\mu_m\eta u)\|^2.
\end{equation}
Now we want to make use of the so-far unused orthogonal directions
to $\nu$, the unit vector in the direction of
$\frac{\p}{\p\rho}$, in the term $||\nabla (\mu_m\eta u)||^2$. For this,
let $\nabla^0$ denote the covariant derivative in directions
orthogonal to $\nu$. The equation $Du=0$ implies that we have the
pointwise inequality
\begin{align*}
 \left|\mu_m\eta c(\nu)\nabla_{\nu}u\right|^2 
  & = \left|\mu_m \eta  \left(D- c(\nu)\nabla_{\nu}\right)u\right|^2\\
  & \leq (n-1) \left|\mu_m\eta\nabla^0u\right|^2.
\end{align*}
With this,~\eqref{eq:ud2-lich} becomes
\begin{equation}\label{eq:bRad}
  \| [D,\mu_m\eta] u\|^2 \geq \|\nabla_{\nu} (\mu_m\eta u)\|^2 + \frac{1}{n-1}\|\mu_m\eta\nabla_{\nu}u\|^2
\end{equation}
For sufficiently large $m$, the last term on the right-hand side is
greater than $\frac{1}{n-1} \|\chi_m\rho^a \eta \nabla_{\nu} u\|^2$, with
$\chi_m$ the characteristic function of the set $\rho \leq m$. For
this we have the following weighted Hardy inequality:
$$
\|\chi_m \rho^a\eta\nabla_{\nu}u\|^2\geq \frac{(n+2a-2)^2}{4}\|
\frac{\chi_m\rho^a\eta u}{\rho}\|^2 - C_1\|\frac{\chi_m\rho^a \eta
  u}{\rho^{1+\frac{\tau}{2}}}\|^2.$$ This inequality is proved in the
same fashion as the Hardy inequality in
Proposition~\ref{prop:af_poinc}.  We remark that the boundary term
corresponding to $\rho=m$ arising in the bounded domain $1 \leq \rho
\leq m$ of this Hardy inequality may be discarded because its
sign is fixed and helps rather than hurts the estimate.  The other
boundary term does not contribute, since $\eta$ vanishes there.

Expanding $[D,\mu_m\eta]u$ in~\eqref{eq:bRad}, and using the
Hardy inequality~\eqref{eq:af_poinc} for the term
$\|\nabla_{\nu}(\mu_m\eta u)\|^2$ and the above weighted Hardy inequality for $\frac{1}{n-1}\|\mu_m\eta\nabla_{\nu}u\|^2$,
we obtain
\begin{align*}
 \|\mu_m|d\eta|u\|^2 & + 2(|u|\eta d \mu_m ,|u|\mu_md \eta) \\
  & \geq
  \frac{(n-2)^2}{4}\|\frac{\mu_m\eta u}{\rho}\|^2 +
  \frac{(n+2a-2)^2}{4(n-1)}\| \frac{\chi_m\mu_m\eta u}{\rho}\|^2 
 - \|\eta d\mu_m u\|^2
 -C_2\|\frac{\mu_m\eta   u}{\rho^{1+\frac{\tau}{2}}}\|^2 \\
  &\geq    
  \left(\frac{(n-2)^2}{4} +   \frac{(n+2a-2)^2}{4(n-1)} -a^2\right)
    \| \frac{\chi_m\mu_m\eta u}{\rho}\|^2 \\
   & \quad \quad \quad 
+   \left(\frac{(n-2)^2}{4} -b^2\right)
    \| \frac{(1-\chi_m)\mu_m\eta u}{\rho}\|^2 
 -C_2\|\frac{\mu_m\eta   u}{\rho^{1+\frac{\tau}{2}}}\|^2.
\end{align*}
 Since $d\eta$ is compactly supported, the  left-hand side is uniformly bounded  as $m\to\infty$.  Moreover, since $b< \frac{n-2}{2}$, after
eventually shrinking the support of $\eta$ by choosing $L$ larger, the
negative term on the right-hand side can be absorbed into the first
two terms, as long as $a$ satisfies
\begin{equation}\label{eq:a}
   \frac{(n-2)^2}{4} + \frac{(n+2a -2)^2}{4(n-1)}- a^2 >0.
\end{equation}
Equivalently for $n>2$,
\begin{equation}\label{eq:ab}
   \frac{n^2}{4} - \frac{n}{2} >a^2-a,
\end{equation}
which (for $n>2$) holds for all $a\in [0,\frac{n}{2}).$
Hence taking the limit as $m\to \infty$, it follows that 
\begin{equation}\label{eq:udec}
  \rho^{a-1}\eta u \in L^2(M\setminus V, S)
\end{equation}
as long as $a>\frac{n}{2}$. Thus $\rho^{\frac{n}{2}-1-\e} u \in
  L^2(M\setminus V, S)$ for all $\e >0$.
\end{proof}

For the proof of Theorem D, we also need information on the behaviour
of the eigenfunctions of $D^2$ with negative eigenvalues on the
asymptotically flat ends of $M$. The following proposition which gives
exponential decay for these eigenfunctions follows immediately from
the proof of~\cite[Theorem 4.1, p. 52]{ag}.

\begin{proposition}\label{udecay3}
 Let $(M,g)$ be a nonspin Riemannian manifold which is asymptotically
 flat of order $\tau >0 $ and has nonnegative scalar curvature.
 Assume $u\in \Dmax(D)$ satisfies  $D^2u = -\lambda^2u,$ for some
 $\lambda > 0$. Then
 \begin{equation}\label{eq:udecay3}
   e^{(\lambda -a)\rho} u \in L^2(M\setminus V, S)
 \end{equation}
 for all $a >0$.
\end{proposition}
\begin{proof} The proof is exactly like the proof of
  Proposition~\ref{udecay} with the functions $\mu_m$ defined as
\[
  \mu_m(\rho): =
  \begin{cases}
     e^{(\lambda - a)\rho} & \text{if $\rho \leq m$},\\
    e^{(\lambda - a)m} & \text{if $\rho >m$}.
  \end{cases}
 \]
\end{proof}

\subsection{More estimates on the asymptotically flat ends of $M$}

As a further application of Agmon's
identity~\eqref{eq:u2}, we use it to derive
$L^p$-estimates for spinors on the asymptotically flat ends of
$M$. These results will not be used for the proof of our main
theorems but are useful in the proof of pointwise estimates.

\begin{proposition}\label{prop:ulp}
 Let $(M,g)$ be a nonspin Riemannian manifold which is asymptotically
 flat of order $\tau>0$ and has nonnegative scalar curvature. Let $u \in
 H^1_{\rho}(M\setminus V, S)$ be a smooth spinor so that $\rho^{b+1}
 |D^2 u|$ is bounded for some $b\geq 1$. Then 
 \begin{equation}\label{eq:ulp}
   \frac{|u|^p}{\rho} \in L^2(M_l)
 \end{equation}
 for all $p\geq 1$ and for all asymptotically flat ends $(M_l, Y_l)$
 of $M$.
\end{proposition}
\begin{proof}
  Fix an asymptotically flat end $M_l$, and let $0\leq \eta\leq 1$
  be a smooth cutoff function supported on $M_l$ in the region
  $\rho(x) \geq L$ and equal to $1$ in the region $\rho(x) \geq 2L$.
  We choose a sequence $\{\gamma_j\}$ of smooth cutoff functions, $0\leq \gamma_j\leq 1$,
  compactly supported on $M$, supported in the region $\rho(x) \leq
  2j$, identically $1$ in the region $\rho(x) \leq j$, and so that
  $|d\gamma_j| \leq \frac{2}{\rho}$.  Let $\eta_j: = \eta
  \gamma_j$. 

  We consider the bounded positive smooth function $f: (0, + \infty) \to (0,+\infty)$
  \begin{equation}\label{eq:f-1}
     f(t) = \frac{t^p}{1 + at^p},
  \end{equation}
  for some positive constant $a>0$ and $p \geq 0$. Observe that this function satisfies
  \begin{equation}\label{eq:f-2}
    \left(\frac{d}{dt} (t f(t))\right)^2 - \left(t \frac{d}{dt}
      f(t)\right)^2 \geq \frac{1}{(p+1)^2} \left(\frac{d}{dt} (tf(t))\right)^2.
  \end{equation}
   Define the sequence of bounded functions on $M$,
 \[
   f_j (x): = \eta_j(x) f(|u(x)|),
 \]
 to which we apply Agmon's identity~\eqref{eq:u2}:
 \begin{align*}
   (D^2 u, f_j^2 u)
    & \geq \|\nabla (f_j u)\|^2 - \|[D,f_j] u\|^2\\
    & \geq \|\eta_j \nabla (f(|u|)u) \|^2 
         + 2 (\eta_j \nabla (f(|u|) u), d\eta_j \otimes f(|u|) u)
         + \|f(|u|) |u|d \eta_j \|^2\\
    & \quad
         - \| f(|u|) |u|d\eta_j\|^2 
         - 2 ( d\eta_j \otimes f(|u|) u, \eta_j d(f(|u|))\otimes u)
         - \|\eta_j |u|d(f(|u|))\|^2\\
   & \geq \|\eta_j d(f(|u|) |u|)\|^2 
         + 2 ( \eta_j \nabla (f(|u|) u), d\eta_j \otimes f(|u|) u)
         - 2 (d\eta_j \otimes f(|u|) u, \eta_j d(f(|u|))\otimes u)\\
&\quad
         - \|\eta_j |u|d(f(|u|))\|^2,
 \end{align*}
 where on the last line we have used Kato's
 inequality~\eqref{eq:kato}. To estimate the left-hand side we use the
 boundedness of $\rho^{b+1} |D^2 u|$ from the hypothesis, while on the
 right-hand side we use the inequality~\eqref{eq:f-2} to estimate the first and the
 fourth term, and we group together the second and the third, 
 to get
 \begin{align*}
  \|\rho^{ b+1}D^2u\|_{L^{\infty}}\ \|\rho^{-b-1}\eta_j^2 f(|u|)^2 |u|\|_{L^{1}}
   & \geq \frac{1}{(p+1)^2} \|\eta_j d(f(|u|) |u| )\|^2
          + \frac{1}{2} \< d \eta_j^2 , f(|u|)^2 d  |u|^2 \>.
 \end{align*}
 Observe that 
 \[
  \eta_j(x)^2 f(|u|(x))^2 |u|(x) \leq 
  \begin{cases}
    |u|(x)^2 & \text{if $|u(x)| \leq 1$}\\
    a^{-2}|u|(x)^2 & \text{if $|u(x)| \geq 1$}.
  \end{cases}
 \]
 Note that since $b\geq 1$, $\rho(x)^{-b-1} |u(x)|^2 \in L^1(M)$.
 Hence the term $\|\rho^{-b-1}\eta_j^2 f(|u|)^2 u\|_{L^{1}}
 $ converges to $\|\rho^{-b-1} \eta^2 f(|u|)^2 u\|_{L^1}$
 as $j\to \infty$. Moreover, for $j\geq 2L$ we have $d\eta_j^2 =
 d\eta^2 + d\gamma_j^2$, and since all the terms containing
 $d\gamma_j$ converge to $0$ as $j\to \infty$, we can take the limit
 as $j\to \infty$ to obtain
 \[
  \|\rho^{b+1}D^2u\|_{L^{\infty}}\ \|\rho^{-b-1}\eta^2 f(|u|)^2 u\|_{L^{1}}
    \geq \frac{1}{(p+1)^2} \|\eta  d(f(|u|) |u| )\|^2
          + \frac{1}{2} (d \eta^2 , f(|u|)^2 d  |u|^2 ).
 \]
 To apply the Hardy inequality~\eqref{eq:af_poinc}, we
 rewrite the above as
 \begin{align*}
  \|\rho^{ b+1}D^2u\|_{L^{\infty}}\ & \|\rho^{-b-1}\eta^2 f(|u|)^2
  |u|\|_{L^{1}}\\
  &  \geq \frac{1}{(p+1)^2} \|d(\eta f(|u|)|u|)\|^2 
        - \frac{2}{(p+1)^2} ( |u|d(\eta f(|u|)) ,  f(|u|)  |u|d\eta)\\
   &     \quad \quad \quad
         + \frac{1}{(p+1)^2} \|  f(|u|) |u|d\eta\|^2
          + \frac{1}{2} ( d \eta^2 , f(|u|)^2 d  |u|^2 )\\
  &\geq \frac{(n-2)^2}{4(p+1)^2} \|\frac{\eta f(|u|) |u|}{\rho}\|^2
          - C_l \|\frac{\eta f(|u|) |u|}{\rho^{1+\tau/2}}\|^2 
   - \frac{1}{(p+1)^2} \| f(|u|) |u|d\eta\|^2 \\
  &   \quad \quad \quad 
   + \frac{1}{2} (1- \frac{1}{(p+1)^2}) \<d \eta^2 , f(|u|)^2 d |u|^2 \>   
   - \frac{1}{2(p+1)^2} ( d(\eta^2), |u|^2 d  f^2(|u|)).
 \end{align*}
 Since $d\eta$ is
 compactly supported, all the terms containing $d\eta$ are bounded by
 a constant $C_1 = C_1(p, \|\chi_{d\eta} |u|\|_{L^{\infty}},
 \|\chi_{d\eta} d(|u|)\|_{L^{\infty}}) >0$ which is independent of
 $j$ and of $a$. Here
 $\chi_{d\eta}$ denotes the characteristic function of the support of $d\eta$.
 Thus
 \begin{align*}\label{eq:ulp-2}
 \|\rho^{ b+1}D^2u\|_{L^{\infty}}\ \|\rho^{-b-1}\eta^2 f(|u|)^2 u\|_{L^{1}} + C_1
 & \geq \frac{(n-2)^2}{4(p+1)^2} \|\frac{\eta f(|u|) u}{\rho}\|^2
        - C_l \|\frac{\eta f(|u|) u}{\rho^{1+\tau/2}}\|^2
\end{align*}
 Since $\tau>0$, we can choose $L$ large enough so that the negative term
 on the right-hand side is absorbed into the positive term. Therefore, 
 \[
   C_2 \|\rho^{-b-1}\eta^2 f(|u|)^2 u\|_{L^{1}} + C_3 \geq \|\frac{\eta f(|u|) u}{\rho}\|^2
 \]
 for positive constants $C_2, C_3$ independent of $a$.  Note that
 since $tf(t)^2 \leq t^{2p+1}$, applying this to the left-hand side of
 the above, we obtain
 \[
   C_2 \| \frac{\eta |u|^{p +\frac{1}{2}}}{\rho}\|^2 + C_3 \geq    \|\frac{\eta f(|u|) u}{\rho}\|^2
 \]
 for all $a>0$. Assuming that $\frac{\eta |u|^{p+\frac{1}{2}}}{\rho}
 \in L^2(M_l)$ and taking the limit as $a\to 0$, it follows that 
 $\frac{\eta |u|^{p+1}}{\rho} \in L^2(M_l)$.  Now the argument follows
 by induction, since we know that $\frac{u}{\rho} \in L^2(M\setminus
 V, S)$.
\end{proof}
 
\section{Coercivity for the Dirac operator}\label{sec:coercive}
In this section we prove 
 two coercivity results. The first  one is 
 for the Dirac operator on $H^1_{\rho}(M\setminus V, S)$:
 
\begin{theorem}\label{thm:D-coercive}
  Let $(M,g)$ be a nonspin Riemannian manifold which is asymptotically
  flat of order $\tau >0$ and has nonnegative scalar curvature. Let
  $S$ be the spinor bundle of a maximal spin structure on $M\setminus
  V$, with $V$ a stratified space given by Theorem~\ref{thm:V}, and
  $D$ be the corresponding Dirac operator.
 Then, there exists a constant $\lambda >0$ so that 
 \begin{equation}\label{eq:coercive}
   \| D u\| \geq \lambda \|\frac{u}{\rho}\|
 \end{equation}
 for all $u$ in $H^1_{\rho}(M\setminus V,S)$.
\end{theorem}

 As a consequence we derive an invertibility result for the Dirac
Laplacian. 

\begin{corollary}\label{cor:invt}
  Assume the hypotheses of Theorem~\ref{thm:D-coercive} hold. Then for
  each smooth spinor $\Psi$ on $M\setminus V$ so that $\rho \Psi \in
  L^2(M\setminus V, S)$, there exists a unique spinor $\Phi \in
  H^1_{\rho}(M\setminus V, S)$ so that $D^2 \Phi = \Psi$.
\end{corollary}
\begin{proof}[Proof of Corollary~\ref{cor:invt}]
 Let $B$ be the bilinear form  on $H^1_{\rho}(M\setminus V, S)$
  defined by
  \[
    B(u,v): = (Du, Dv).
  \]
  Clearly $B$ is bounded on $H^1_{\rho}(M\setminus V, S)$. 
  By~\eqref{eq:coercive}  and the
  Lichnerowicz formula~\eqref{boch1}, we have
   \[
    B(u,u) \geq  \frac{1}{2}\|\nabla u\|^2 + \frac{\lambda^2}{2} \|\frac{u}{\rho}\|^2
           \geq  C  \|u\|_{H^1_{\rho}}^2,
  \]
  with $C = \min\{\frac{1}{2}, \frac{\lambda^2}{2}\}$. Thus, 
   $B$ satisfies the conditions of the
  Lax-Milgram Lemma. 
   We  apply this 
 lemma 
   to the linear functional on $H^1_{\rho}(M\setminus V, S)$,
   \[
     L(v): = (\Psi, v),
   \]
    which is bounded, since $\rho \Psi \in L^2(M\setminus V, S)$. 
     Hence,  
 there exists a unique $\Phi \in H^1_{\rho}(M\setminus V, S)$
   so that
   \[
   B(\Phi,v) = L(v).
   \]
 In particular,
  \[
    (D^2\Phi, v) = (\Psi, v)
  \]
   for all $v \in \sC^{\infty}_0(M\setminus V, S)$. This implies that
   $\Phi$ is a weak solution to $D^2 \Phi = \Psi$. Since $\Psi$ is
   smooth, elliptic regularity implies that $\Phi$ is smooth and
  is thus a strong solution.
\end{proof}

\subsection*{Proof of  Theorem~\ref{thm:D-coercive}}
Because $\Dmin(D)$ is dense in $H^1_{\rho}(M\setminus V,
S)$, it suffices to prove inequality (\ref{eq:coercive}) for
$u\in \Dmin(D)$. By  Corollary~\ref{cor:Dmin-ns}, 
the null-space of
$D$ on its minimal domain is trivial.  We need to show that $0$ is
not in the essential spectrum of $D$. We prove this by contradiction.

Let $\{u_j\}$ be an infinite $L^2_{\rho}$-orthonormal sequence of
sections with $u_j\in \Dmin(D)$ satisfying $\|D u_j \|_{L^2}\to 0$.
Since $R\geq 0$,  the Lichnerowicz formula~\eqref{boch1} implies that
$\{u_j\}$ is a bounded sequence in $H^1_{\rho}(M\setminus V, S)$.
By the compactness Lemma~\ref{lem:rellich}, we may pass to a
subsequence (still denoted $\{u_j\}$) which converges strongly on
compacta in $L_{\rho}^2$ and weakly in $H^1_{\rho}$ to a section $u
\in H^1_{\rho}(M\setminus V, S)$. Since $D\colon H^1_{\rho}(M\setminus V,S)
\to L^2(M\setminus V,S)$ is bounded, weak $H^1_{\rho}$-convergence
implies that $u$ lies in the null-space of $D$. By
 Corollary~\ref{cor:Dmin-ns}, 
 it follows that $u =0.$ We show that our
hypotheses prohibit this and arrive at a contradiction.

Consider the sequence $\{u_j\}\subset \Dmin(D)$ which converges to zero
strongly on compacta in $L^2_{\rho}$ and weakly in $H^1_{\rho}$.
Next, we observe that the sequence must also converge
to zero in $L^2_{\rho}$-norm on the asymptotically flat ends. To see this, choose an end, $M_l$, and let $\eta$ be a
cutoff function, which is supported in $M_l$ and identically $1$ in a neighborhood of infinity in $M_l$.
Clearly $\|D(\eta u_j)\|_{L^2} \to 0$. Moreover the Lichnerowicz
formula and Kato's inequality combined with the Hardy inequality~\eqref{eq:af_poinc} on this
asymptotically flat end give
\[
  \|D (\eta u_j)\|^2\geq \|\nabla (\eta u_j)\|^2 \geq \|d|\eta u_j|\|^2\geq
  \frac{(n-2)^2}{4} \|\frac{\eta u_j}{\rho}\|^2 - C_l \|\frac{\eta u_j}{\rho^{1+\tau/2}}\|^2.
\]
Shrinking the support of $\eta$, we can absorb the negative term
above and conclude that $\|\frac{\eta u_j}{\rho}\|_{L^2} \to 0$,
as claimed.
 Since $\rho=1$ in a neighborhood of $V$, it follows
  that the $L^2$-mass of the sequence accumulates in an arbitrarily
small neighborhood of $V$.  We show that this cannot happen. 

We first show that for any stratum $V^{k_b}$ of $V$ and any $W\in
\TRC(V^{k_b})$, the sequence $\{\|\frac{u_j}{r_b}\|_{L^2(W)}\}$ is
bounded. To see this, 
let $W'\in \TRC(V^{k_b})$ be an open set
containing $W$, and $\zeta \in \sC^{\infty}_0(W')$ with $0\leq \zeta
\leq 1$ so that $\zeta \equiv 1$ on
$W$. 
Since 
\[
  \|D(\zeta u_j)\|^2\leq 2 \|[D,\zeta]u_j\|^2 + 2
   \|Du_j\|^2,
\]
our hypothesis implies that $\{D(\zeta u_j)\}$ is uniformly bounded in
$L^2$-norm.  Since $\zeta u_j \in \Dmin(D)$, when $k_b>2$ the
estimate~\eqref{eq:Lichkb} gives
\begin{equation*}\label{eq:closebnd-b}
  \frac{(k_b-2)^2}{4}\|\frac{\zeta u_j}{r_b}\|^2 
\leq \| D(\zeta u_j)\|^2
+ C_{W'}\|\frac{\zeta u}{r_b^{1/2}}\|^2,
\end{equation*}
while for $k_b =2$ the estimate~\eqref{eq:Lich} gives 
\begin{equation*}\label{eq:closebnd-2}
 \frac{1}{4} \|\frac{\zeta u_j}{r}\|^2 \leq \|D(\zeta u_j)\|^2 +
 C_{W'}\|\frac{\zeta u}{r^{1/2} }\|^2,
\end{equation*}
with $C_{W'}$ a positive constant depending on $W'$.  Shrinking the
radii of the tubular neighborhoods $W$ and $W'$, we can absorb the
last terms into the left-hand side terms of each of the above
formulas, and conclude that the sequence $\{\|\frac{\zeta
  u_j}{r_b}\|_{L^2}\}$ is uniformly bounded.

Since in $L^2$-norm the sequence $\{u_j\}$ converges to zero on any
compact subset of $M\setminus V$, while
$\{\|\frac{u_j}{r_b}\|_{L^2(W)}\}$ is uniformly bounded 
 for  
any $W \in \TRC(V^{k_b})$, it follows that $\{u_j\}$ converges to zero
in $L^2$-norm on any $W \in \TRC(V^{k_b})$. Thus the $L^2$-mass
of the sequence cannot accumulate in arbitrarily small neighborhoods
of $V$, contradicting the above.  \qed

\subsection{A second coercivity result}\label{sec:2coerc}
As we will see in the next section, the invertibility result of
Corollary~\ref{cor:invt} suffices to prove Theorem A, the existence of
Witten spinors, in all the cases except the case when $n=4$ and $\tau
\in (\frac{n-2}{2}, \frac{n}{2}]$. The corollary is insufficient for
this exceptional case because the spinor $\Psi$ for which we need to
apply Corollary~\ref{cor:invt} is only in $L^2(M\setminus V, S)$ and
not in $\rho L^2(M\setminus V, S)$. To cover the exceptional case, we
prove a coercivity result on a weighted Hilbert space with weight
shifted from that of $H^1_{\rho}(M\setminus V, S)$.

Define $\sH(M\setminus V, S)$ to be 
 the closure of $\sC^{\infty}_0(M\setminus V, S)$ in the norm
 \begin{equation}\label{eq:2Hilb}
   \|u\|^2 + \|\rho \nabla u\|^2.
 \end{equation}
 Note that  \( \sH(M\setminus V, S) \subset \Dmin(D) \subset
  H^1_{\rho}(M\setminus V, S), \) and therefore $D$
  has trivial null-space on $\sH$. 
\begin{theorem}\label{thm:2-coerc}
 Assume that $(M,g)$ satisfies the hypotheses of
 Theorem~\ref{thm:D-coercive}. Then there exists a constant $\lambda
 >0$ so that 
 \begin{equation}\label{eq:2-coerc}
  \|\rho Du\| \geq \lambda \|u\|
 \end{equation}
 for all $u$ in $\sH(M\setminus V, S)$.
\end{theorem}
\begin{proof}
  The proof is similar to the proof of Theorem~\ref{thm:D-coercive}
  except that we need new estimates on the asymptotically flat
  ends of $M$, where we have modified the norms of our Hilbert space.

  Assume there exists an infinite $L^2$-orthonormal sequence $\{u_j\}$
  in $\sH(M\setminus V,S)$ so that $\|\rho Du_j\|_{L^2} \to 0$. The sequence is clearly a bounded sequence in
  $H^1_{\rho}(M\setminus V, S)$. The same argument as in
  Theorem~\ref{thm:D-coercive} then gives that the sequence converges
  strongly to zero in $L^2$-norm on compacta in $M\setminus V$ and on
  any compact neighborhood of $V$. Therefore the $L^2$-mass of the
  sequence must accumulate on the asymptotically flat ends of $M$.

  We show that there exist constants $A>0$, $L>0$ so that 
  \begin{equation}\label{eq:af_coerc}
    \|\rho Du \| \geq A \|u\|
  \quad \text{for all $u \in \sH(M\setminus V, S)$ with $\supp (u)
    \subset \{x\in M_l:\rho(x)>L\}$,}
  \end{equation}
  where $M_l$ is an asymptotically flat end of $M$.

  Assuming this for the moment, let $\eta$ be a cutoff function
  supported on one of the ends. Then $\|\rho D(\eta
  u_j)\| \to 0$, and then~\eqref{eq:af_coerc} shows that the $L^2$-mass of
  the sequence $\{u_j\}$ cannot accumulate on the asymptotically flat
  ends either. Thus we reach a contradiction.

  It remains to show~\eqref{eq:af_coerc}.  Since $\sC^{\infty}_{0} (M\setminus V,
  S)$ is dense in $\sH(M\setminus V, S)$, it suffices to
  prove~\eqref{eq:af_coerc} for $u \in \sC^{\infty}_{0} (M\setminus V,
  S)$. We write
  \[
    \|\rho Du\|^2 
    = \|D (\rho u)\|^2 - 2 (D(\rho u), [D, \rho] u) + \|[D,\rho] u\|^2.
  \]
  Let $\nu$ denote the unit vector in the direction of
  $\frac{\p}{\p \rho}$, and let $\nabla^0$ denote the covariant
  derivative in directions $\{e_{\sigma}\}_{\sigma = 2, \ldots ,  n}$ orthogonal to $\nu$. 
  We apply the Lichnerowicz formula to the first term on the
  right-hand side and expand the cross-term, to 
  \begin{align*}
    \|\rho Du\|^2 
 &  \geq \|\nabla_{\nu} (\rho u)\|^2 + \|\nabla^0 (\rho u)\|^2 
       - 2 ( c(\nu)\nabla_{\nu} (\rho u), [D,\rho] u)\\
 & \quad \quad \quad
     - 2 ( (D- c(\nu)\nabla_{\nu})(\rho u), [D, \rho] u) 
       + \|[D,\rho] u\|^2.
  \end{align*}
Observe that we can group 
\[
 \|\nabla^0 (\rho u)\|^2 
- 2 ((D- c(\nu)\nabla_{\nu})(\rho u), [D, \rho] u)
 + (n-1) \|[D,\rho] u\|^2 
  = \sum_{\sigma=2}^{n} \|\nabla_{e_{\sigma}}  (\rho u) +  c(e_{\sigma})[D,\rho]u\|^2,
\]
and obtain
  \begin{align*}
    \|\rho Du\|^2 
 &  \geq \|\nabla_{\nu} (\rho u)\|^2 + \sum_{\sigma=2}^{n}
 \|\nabla_{e_{\sigma}}  (\rho u) +  c(e_{\sigma})[D,\rho]u\|^2\\
 & \quad \quad \quad
       - 2 ( c(\nu)\nabla_{\nu} (\rho u), [D,\rho] u)
       - (n-2) \|[D,\rho] u\|^2
  \end{align*}
Since the metric is asymptotically flat of order $\tau$,
$|[D, \rho] | = 1 + \sO(\rho^{-\tau})$, and thus
\[
  \|\rho Du\|^2 
  \geq \|\nabla_{\nu} (\rho u)\|^2
   - 2 (\rho \nabla_{\nu} u, u)
   - n \|u\|^2 - C_1 \|\frac{u}{\rho^{\tau/2}}\|^2,
\]
for some constant $C_1 >0$ independent of $u$. To handle the term  $-2\<
\rho \nabla_{\nu} u, u\>$, we integrate by parts to rewrite it as
\begin{align*}
  - ( \nabla_{\nu} |u|^2, \rho)
  & = (|u^2|, \rho^{1-n}\nabla_{\nu}(\rho^n) )- C_2 \|\frac{u}{\rho^{\tau/2}}\|^2 \\
  & = n \|u\|^2 - C_3\|\frac{u}{\rho^{\tau/2}}\|^2 .
\end{align*}
The error terms
arise from the deviation of the metric from the Euclidean metric.  Thus, we obtain
\[
  \|\rho Du \|^2 \geq \|\nabla_{\nu}(\rho u)\|^2 - C \|\frac{u}{\rho^{\tau/2}}\|^2.
\]
 with $C >0$ a constant independent of $u$. Now the desired inequality
 follows using the Hardy inequality~\eqref{eq:af_poinc} on the asymptotically flat
 end and choosing $L$ sufficiently large so that
 the lower order term can be absorbed into  $\frac{(n-2)^2}{4} \|u\|^2$.
\end{proof}
As a consequence, 
 we have the following invertibility result, analogous to Corollary~\ref{cor:invt}.
\begin{corollary}\label{cor:2invt}
 Assume that the hypotheses of Theorem~\ref{thm:2-coerc} hold. Then
 for each smooth spinor $\Psi \in L^2(M\setminus V, S)$, there exists
 a unique spinor $\Phi$ in $\sH(M\setminus V, S)$ so that 
 $D(\rho^2 D \Phi) = \Psi$.
\end{corollary}
\begin{proof}
 The only difference from the proof of Corollary~\ref{cor:invt} is
 that we now take $B$ to be the bilinear form on $\sH(M\setminus V, S)$
 defined as 
 \[
   B(u, v): = (\rho Du, \rho Dv),
 \]
 and apply the Lax-Milgram Lemma to the bounded linear  
  functional  $L$ on $\sH(M\setminus V, S)$ defined to be
 $L(v): = (\Psi, v)$.

 The only argument which requires a slightly different 
  justification is showing that the bilinear form $B$ is
 coercive. For this, let $\e >0$ small to be chosen later, and bound
 \begin{equation}\label{eq:2-coerc-1}
   \|\rho Du \|^2 \geq  \e \|\rho Du\|^2 + (1-\e) \lambda^2 \|u\|^2
 \end{equation}
 using~\eqref{eq:2-coerc}. To estimate the first term on the
 right-hand side, note that since $u\in \sH(M\setminus V, S)$, it
 follows that $\rho u \in H^1_{\rho}(M\setminus V,S)$.  Thus using
 the Lichnerowicz formula~\eqref{boch1}, we have 
\begin{align*}
     \|\rho Du \|^2
    &= \|D(\rho u) \|^2 +\| |d\rho| u\|^2 -2( D(\rho u),c(d\rho)u) \\
    & \geq \frac{1}{2}\|D(\rho u)\|^2 -\| |d\rho| u\|^2  \\
    & \geq \frac{1}{2}\|\nabla(\rho u)\|^2 -\| |d\rho| u\|^2  \\
    & \geq  \frac{1}{2}\|\rho\nabla u\|^2+\frac{1}{2}\| |d \rho|  u\|^2
+(\rho\nabla u,d\rho\otimes u)-\| |d\rho| u\|^2  \\
    & \geq  \frac{1}{4}\|\rho\nabla  u\|^2-\frac{3}{2}\| |d \rho| u\|^2 
  \end{align*}
  Choosing $\e$ so that 
   $\frac{3\e}{2}\|d \rho\|^2_{L^{\infty} (M) }<\frac{1}{2}(1-\e)
    \lambda^2$ gives the coercivity of the bilinear form $B$ on the
  Hilbert space $\sH(M\setminus V, S)$.
\end{proof}

\section{Proof of our main results}\label{sec:main}

In this section we prove our main theorems stated in the
Introduction. 
Since we are assuming that the manifold $M$ is
orientable and nonspin, its dimension must be $n\geq 4$ (as every
orientable $3$-manifolds is automatically spin).  
For the proof of Theorem A, the existence and construction of the
Witten spinor is separated into two cases depending on the order of
convergence, $\tau,$ of the asymptotically flat metric to a Euclidean
metric.  The reason for this is that, for $\tau > \frac{n}{2}$, a
spinor $\psi_0$ supported on an end and constant in a frame induced
from an asymptotically flat coordinate system satisfies $\rho D\psi_0
\in L^2(M\setminus V, S)$; the existence of the Witten spinor is then
an immediate consequence of Corollary~\ref{cor:invt}. However, if
$\tau \in (\frac{n-2}{2}, \frac{n}{2}]$, then $\rho D\psi_0$ need not
be $L^2$, but $\rho D^2\psi_0$ is still square
integrable. Establishing the existence of the Witten spinor from this
weaker hypothesis is a two step procedure, provided that $n\geq 5$. In
the case $n=4$ and $\tau \in (\frac{n-2}{2}, \frac{n}{2}]$ the proof
requires further refinement.

The proofs of Theorem B and Theorem C are based on the form of the
Witten spinor derived in Theorem A. As a consequence, we separate
these proofs into cases, according to the construction we use for the
Witten spinor.

We conclude with the proof of Theorem D.

Without loss of generality, we can assume that the radius defining function
$\rho$ in Definition~\ref{def:asy_flat} is identically $1$ in a
compact neighborhood of $V$.

\subsection{Proof of Theorem A}\label{sec:ThmA}
Let $\psi_0$ be a smooth spinor which is constant on the
asymptotically flat ends of $M$ and supported outside a
neighborhood of $V$. It follows (see~\eqref{eq:cs}) that
$\rho^{\tau+1} |D\psi_0|$ is bounded on $M\setminus V$. We
separate the construction into two cases, according to whether
$\frac{n-2}{2} < \tau \leq \frac{n}{2}$ or $\tau > \frac{n}{2}$.

If $\tau >\frac{n}{2}$, then $\rho D\psi_0 \in L^2(M\setminus V, S)$,
and thus the spinor $D\psi_0$ satisfies the hypothesis of
Corollary~\ref{cor:invt}. Hence, there exists a unique $u\in
H^1_{\rho}(M\setminus V,S)$ so that
\[
  D^2 u = - D\psi_0.
\]
From Corollary~\ref{cor:scws}, it follows that the spinor $\psi: = Du +
\psi_0$ is a Witten spinor.

If $\tau \in (\frac{n-2}{2}, \frac{n}{2}]$, more work is required to
construct the desired Witten spinor. In this case,
$D^2\psi_0$ satisfies the hypothesis of
Corollary~\ref{cor:invt}. Hence there exists a unique $w\in
H^1_{\rho}(M\setminus V,S)$ so that 
\[
  D^2 w = - D^2\psi_0.
\]
Let 
\[
  W: = w + \psi_0.
\]
Then $D^2W = 0,$ $DW \in L^2(M\setminus V, S)$, and $DW$ is in the
null-space of the maximal extension of the Dirac operator. If in fact
$DW =0$, then $W$ is the desired Witten spinor. If $DW\not = 0$, then
we modify $W$ further.  Let $\eta$ be a smooth cutoff function,
vanishing in a compact neighborhood of $V$ where $\rho = 1$ and
identically $1$ outside a bigger compact neighborhood of $V$. Without
loss of generality, we can assume that $\eta$ is $1$ on the support of
$\psi_0$.  Since $DW \in L^2(M\setminus V, S)$ is strongly harmonic,
it follows that $\eta DW \in H^1_{\rho}(M\setminus V, S)$ and $D(\eta
DW)$ is compactly supported. By Proposition~\ref{udecay2},
$\rho^{\frac{n}{2}-1-\e} \left(\eta DW\right) \in L^2(M\setminus V,
S)$ for all $\e>0$. In particular, when $n\geq 5$ we have $\rho D(\eta
W) \in L^2(M\setminus V, S)$ and by Corollary~\ref{cor:invt} there
exists a unique $u \in H^1_{\rho} (M\setminus V, S)$ so that
\[
  D^2 u = - D(\eta W).
\]
Then, as in the previous case, Corollary~\ref{cor:scws} gives that the spinor $\psi: = Du +
\eta W = Du + \eta w + \psi_0$ is a Witten spinor.

We are left to analyze the case $\tau \in (\frac{n-2}{2},
\frac{n}{2}]$ and $n=4$. In this case, we use our second coercivity
result in Section~\ref{sec:2coerc} to construct the Witten
spinor. Since $D\psi_0 \in L^2(M\setminus V, S)$, by
Corollary~\ref{cor:2invt} there exists a unique $u \in L^2(M\setminus
V, S)$ with $\rho \nabla u \in L^2(M\setminus V, S)$ so that 
\begin{equation}\label{eq:4-1}
  D (\rho^2 Du) = - D\psi_0.
\end{equation}
 We set $v:= \rho^2 Du$ and let $\psi:= v + \psi_0$. Since
  $\rho Du \in L^2(M\setminus V, S)$, then $\frac{v}{\rho} \in
  L^2(M\setminus V, S)$. Moreover $Dv =- D\psi_0 \in L^2(M\setminus V,
  S)$ and then Proposition~\ref{prop:bochend} gives $\nabla v \in
  L^2(M_l,S\lvert_{S_l})$ for all asymptotically flat ends $(M_l,Y_l)$
  of $M$. Hence $\psi$ is a Witten spinor.   
\qed

\begin{remark}
 Note that the spinor $\psi - \psi_0$ satisfies the hypothesis of
 Proposition~\ref{udecay2} with $b = 1$. Thus $\rho^{\frac{n}{2}-1-\e}
 (\psi-\psi_0) \in L^2(M\setminus V, S)$  for all $\e >0$ .
\end{remark}

\begin{remark}\label{adumbration}
  Note that if we set $\Phi = w + \psi_0$, with $w\in
  H^1_{\rho}(M\setminus V,S)$ a solution to $D^2 w = - D^2 \psi_0$,
  $\Phi$ is in the minimal domain of $D$ near $V$ and is weakly
  harmonic. If $\Phi$ is also strongly harmonic, then by
  Proposition~\ref{prop:mass}, the positive mass theorem holds for
  $M$. If $\Phi$ is not strongly harmonic, then $D\Phi$ is a nonzero
  strongly harmonic $L^2$-spinor, necessarily in $\Dmax(D)
  \setminus\Dmin(D).$
\end{remark} 

\subsection{Proof of Theorem B}\label{sec:ThmB}
The proof of this theorem is a consequence of the estimates we derived
in Section~\ref{sec:iie}. 

Consider the Witten spinor $\psi$ given by the proof of Theorem A.
Thus $\psi = Du + \psi_0$ in the case when $\tau > \frac{n}{2}$, $\psi
= Du + \eta w + \psi_0$ in the case $\tau \in (\frac{n-2}{2},
\frac{n}{2}]$ and $n\geq 5$; while $\psi = \rho^2 Du + \psi_0$ in the
case $\tau \in (\frac{n-2}{2}, \frac{n}{2}]$ and $n=4$.  Both $\psi_0$
and $\eta w$ vanish in a neighborhood of $V$ (where $\rho
=1$). Hence $\psi = Du$ in this neighborhood, and therefore also $D^2 u =0$ there.  Moreover, by construction we have have $\chi u \in
\Dmin(D)$ for all $\chi \in \sC^{\infty}_0(M)$. Thus after multiplying
by a cutoff function supported in the region where $D^2 u=0$
and which is identically $1$ in a smaller neighborhood of $V$, $u$
satisfies the hypothesis of Lemma~\ref{lem:u_est} and
Lemma~\ref{lem:u_estb}. 
 Then 
Lemma~\ref{lem:v_est} gives the desired estimate 
 for $\psi$ 
near $V^2$, while
Lemma~\ref{lem:v_estb} gives
 the 
estimates near the higher codimension
strata $V^{k_b}$.  \qed

\subsection{Proof of Theorem C}\label{sec:CorC}
The main ingredients for this proof are Proposition~\ref{prop:mass}
and Proposition~\ref{prop:v_min}

Let $\psi_0$ be a smooth spinor which is constant on the
asymptotically flat ends of $M$ and supported outside a neighborhood
of $V$. We assume that $|\psi_0| \to 1$ on each of the asymptotically
flat ends of $M$. (The case when $|\psi_0| \to 1$ only on one of the
ends $M_l$ of $M$ and it converges to $0$ on all the others, follows
similarly.)

Let $\psi$ be the Witten spinor constructed in
Theorem A and which satisfies~\eqref{eq:wish}. Thus $\psi = Du +
\psi_0$ in the case when $\tau > \frac{n}{2}$, $\psi = Du + \eta w +
\psi_0$ in the case $\tau \in (\frac{n-2}{2}, \frac{n}{2}]$ and $n\geq
5$, while $\psi = \rho^2 Du + \psi_0$ in the case $\tau \in
(\frac{n-2}{2}, \frac{n}{2}]$ and $n=4$.
Recall that $\eta$ is a smooth function vanishing in a neighborhood of
$V$ where $\rho =1$.  Moreover, in the first two cases $u \in
H^1_{\rho}(M\setminus V, S)$, while in the third case $u \in
L^2(M\setminus V, S)$ and $\rho \nabla u \in L^2(M\setminus V, S)$.

We want to show that the spinor $\psi$ satisfies the conditions of
Proposition~\ref{prop:mass}. Then, since the scalar curvature is
nonnegative, the nonnegativity of the mass follows from
formula~\eqref{eq:mass-pos}.  

Let $\chi$ be a smooth cutoff function on $M$ which is supported in a
neighborhood of $V$ where $\rho =1$ and $\psi_0 =0$.  We need to show
that $\chi \psi\in \Dmin(D)$. Note that $\chi \psi = \chi Du$ in all
the cases. 
From the properties of $\chi$ and $u$, it is clear that $\chi u \in
H^1_{\rho}(M\setminus V, S)$. Then, from
Corollary~\ref{cor:Dmin-ns}(2) it follows that $\chi u \in \Dmin(D)$,
while from the construction of $\psi$ it follows that $\chi D^2 u
=0$. Since $\chi \psi = D(\chi u) - [D, \chi]u$, it is enough to show
that $D(\chi u) \in \Dmin(D)$. This follows as a consequence of
Proposition~\ref{prop:v_min} applied to $\bar{u}: = \chi u$ and
$\bar{v}: = D \bar u$, since $\bar{u} \in \Dmin(D)$, $D^2 \bar{u} =0$
in a neighborhood of $V^2$, and $\frac{|\bar
  v|}{r^{1/2}\ln^{1/2}(\frac{1}{r})}\in L^2 (W)$ for all $W \in
\TRC(V^2)$ by the assumption~\eqref{eq:wish}.

It remains to show that the mass cannot be zero. If the mass were
zero, then from~\eqref{eq:mass-pos} it would follow that the
Witten spinor $\psi$ is covariantly constant, i.e. $\nabla \psi
=0$. This implies that $d|\psi|^2 =0$, and thus $|\psi|$ is constant
on $M\setminus V$. Since $\psi$ is asymptotic to $\psi_0$ near the
asymptotically flat ends, it follows that $|\psi| >0$. On the other
hand, $\psi$ is in minimal domain of the Dirac operator near $V^2$,
condition which, via Corollary~\ref{cor:Dmin}, forces $|\psi|$ to be
arbitrarily small near $V^2$. Therefore we obtain a contradiction.
\qed

\begin{remark}
  Note that from the positivity of the total mass of the manifold we
  cannot conclude the positivity of the individual asymptotically
  flat ends.  The positivity of the mass of each end
  follows exactly as above, once we know that for each asymptotically flat end
  $M_l$ there exists a Witten spinor  which is asymptotic to
  $\psi_{0l}$ with $|\psi_{0l}| \to 1$ on $M_l$, vanishes at infinity
  on all the other ends of $M$, and satisfies~\eqref{eq:wish}.
\end{remark}

\subsection{Proof  of Theorem D}\label{sec:proofTD}

We analyze the space of strongly harmonic $L^2$-spinors on $M\setminus
V$.  From Corollary~\ref{cor:Dmin-ns}(3), we know that such a nonzero
spinor cannot be in $\Dmin (D)$. Hence we introduce the following
definition:
\begin{definition}
 A harmonic spinor $\psi \in L^2(M\setminus V,
 S)$  is called {\em singular} if $\psi \in \Dmax(D)
 \setminus \Dmin(D)$. Let $H_{\sing}\subset L^2(M\setminus V,
 S)$ denote the space of singular
 harmonic spinors.
\end{definition}
We have seen in Remark~\ref{adumbration} that singular harmonic
spinors are the obstruction to extending Witten's proof of the
positive mass conjecture to nonspin manifolds using weakly harmonic
spinors. In this proof we show that not only do singular harmonic
spinors exist but that, in fact, there is an infinite dimensional
space of them.

\paragraph{Step 1:}
We first reduce the problem to showing that it is enough to construct
infinitely many spinors $\psi \in \Dmax(D)$ with
$\nabla \psi \notin L^2(M\setminus V, S)$ which are linearly
independent in $\Dmax(D)/\Dmin(D)$.

On $\Dmax(D)$ we have the inner-product
\(
  \<\!\< \psi, \phi\>\!\>_1: = \<\psi,\phi\>_{L^2} + \<D\psi,
  D\phi\>_{L^2}
\)
and the degenerate inner-product
\(
  \<\!\< \psi, \phi\>\!\>_2: =  \<D\psi, D\phi\>_{L^2}.
\)
Let $\sA$ and $\sB$ be the orthogonal complements of $\Dmin(D)$ in
$\Dmax(D)$ with respect each of them
\[
  \Dmax(D) = \Dmin(D) \oplus \sA
  \quad \text{and} \quad
  \Dmax(D) = \Dmin(D) \oplus \sB.
\]
It is clear that $\sA \,\iso \,\Dmax(D)/\Dmin(D)$. Moreover
\begin{equation}\label{eq:oc}
   \sA = \{\psi \in \Dmax(D) \mid D^2\psi = - \psi\}
 \quad \text{and} \quad
   \sB = \{\psi \in \Dmax(D) \mid D^2 \psi =0\}.
\end{equation}
Each element $\psi$ in $\sA$ is an eigenfunction of $D^2$ with
negative eigenvalue $-1$. From Proposition~\ref{udecay3} it follows that
$\psi$ decays exponentially on each of the asymptotically flat ends of
$M$.  Then, by Corollary~\ref{cor:invt} there exists a unique
$u_{\psi} \in \Dmin(D)$ so that $D^2 u_{\psi} = \psi$.  We use this to
construct an isomorphism from $\sA$ to $\sB$.

\begin{lemma}\label{lem:T}
 The map $T:\sA \to \sB$ defined as $T(\psi) = \psi+ u_{\psi}$ is 
 an isomorphism.
\end{lemma}

\begin{proof}
 Since $D^2 (T\psi) = D^2 \psi + D^2 u_{\psi} =0$, the map $T$ is
 well-defined.  The uniqueness of $u_{\psi}$ gives that $T$ is linear.
 To show injectivity, note that if $\psi + u_{\psi} =0$ then it
 follows that $u_{\psi} \in \sA$, as $D^2 u_{\psi} = -
 u_{\psi}$. Since $u_{\psi} \in \Dmin(D)$,~\eqref{eq:oc} gives
 $u_{\psi}=0$ and thus $\psi=0$.

 To show surjectivity, let $\phi \in \sB$. From the orthogonal
 decomposition of $\Dmax(D) = \Dmin(D) \oplus \sA$, there exist
 unique $u\in \Dmin(D)$ and $\psi \in \sA$ so that 
 \[
   \phi = u + \psi.
 \]
 Since $D^2 \phi =0$, it follows that
 $D^2 u =  \psi$. By Corollary~\ref{cor:invt} there exists a
 unique solution $u_{\psi}$ to $D^2 u= \psi$ in the minimal domain of
 $D$. Thus we must have $u = u_{\psi}$ and hence $\phi = T(\psi)$.
\end{proof}

We consider now the Dirac operator restricted to $\sB$,
\[
  D\lvert_{\sB}: \sB \to L^2(M\setminus V, S).
\]
From~\eqref{eq:oc} it follows that $\Range(D\lvert_{\sB}) \subset H_{\sing}$.
Moreover, $\Ker (D\lvert_{\sB}) \subset H_{\sing}$. We write
\begin{equation}
 \sB \ \iso \ \Ker (D\lvert_{\sB}) \oplus \Range (D\lvert_\sB).
\end{equation}
Therefore to show that $H_{\sing}$ is infinite dimensional, it is
enough to show that either $\Ker(D\lvert_{\sB})$ or
$\Range(D\lvert_{\sB})$ is infinite dimensional.  In other words, we
need to show that $\sB$ is infinite dimensional. Via Lemma~\ref{lem:T}
this is equivalent to showing that $\sA$ is infinite dimensional.  To
show this, it is enough to construct infinitely many spinors $\psi \in
\Dmax(D)$ with $\nabla \psi \notin L^2(M\setminus V, S)$ which are
linearly independent in $\Dmax(D)/\Dmin(D)$.

Since, as we mentioned above, for spinors supported on the asymptotically flat ends of $M$
the maximal and minimal domain of $D$ coincide, the construction
reduces to a construction near $V$. To perform this construction, we
need to understand a bit better the geometry near $V^2$, and thus
expand on the material in Sections~\ref{ssec:gV} and~\ref{ssec:near2}.

\paragraph{Step 2:} We now take a closer look at the geometry near $V^2$.
We continue with the set-up in Section~\ref{ssec:gV}: Let $W =
B_{\e}(Y) \in \TRC(V^2)$ be contractible and so that $(Y,\phi)$ is a
coordinate neighborhood in $V^2$ with coordinates $\phi(y) = (y^1,
\ldots, y^{n-2})$ and also a trivializing chart for the normal bundle
$N^2$ to $V^2$ in $M$. This gives the coordinates $(t^1, t^2, y^1,
\ldots, y^{n-2})$ on $W$, and upon introducing the normal distance
function $r$ and the angular function $\theta$ in the normal disks to
$Y$, it gives the polar coordinates $(r, \theta, y^{1}, \ldots,
y^{n-2})$ on $W\setminus V$. We consider the corresponding frame $\{
\frac{\partial}{\partial r},\frac{1}{r} \frac{\partial}{\partial
  \theta}, \frac{\partial}{\partial y^1}, \ldots,
\frac{\partial}{\partial y^{n-2}}\}$, to which we apply the
Gram-Schmidt procedure, to obtain the orthonormal frame $\{e_r,
e_{\theta}, e_{3}, \ldots, e_n\}$. From~\eqref{drb} we have $e_r =
\frac{\partial}{\partial r}$, and then the definition
of $\theta$, Gau{\ss} Lemma, and~\eqref{eq:ov} give $e_{\theta} = \frac{1}{r}
\frac{\partial}{\partial\theta} + \sO(r^2 e_\theta)$ and $\< e_{\theta},
\frac{\partial}{\partial y^j}\> = \sO(r)$. Since from~\eqref{eq:rp}
$e_r$ is perpendicular to $\frac{\partial}{\partial y^j}$, it follows
$e_j = \frac{\partial}{\partial y^j} + \sum_{3\leq i \leq j-1}
\sO(e_i) + \sO(re_{\theta})$ for all $3 \leq j \leq n$. Moreover
\begin{equation}\label{eq:dr}
 e_r (r) =1, \quad e_{\theta}(r) = 0, \quad \text{and}\quad e_j(r)
 = 0 \quad \text{for all $3 \leq j \leq n$},
 \end{equation}
while
\begin{equation}\label{eq:dtheta}
 e_r(\theta) =0, \quad e_{\theta}(\theta) = \frac{1}{r}+ O(r) \quad
 \text{and}\quad e_j(\theta) = \sO(1)
 \quad \text{for all $3 \leq j \leq n$}.
\end{equation}
Concerning the various covariant derivatives, we
have
\begin{equation}\label{eq:cd}
 \nabla_{e_r} e_r=0, \quad
 \nabla_{e_r} e_{\theta} = \sO(re_{\theta}),\quad
 \nabla_{e_{\theta}} e_r = -\frac{1}{r} e_{\theta} + \sO(re_{\theta}).
\end{equation}
Writing $D_Y = \sum_{j=3}^n c(e_j) \nabla_{e_j}$, the Dirac operator
on $W\setminus V$ is
\[
  D = c(e_r) \nabla_{e_r} + c(e_{\theta}) \nabla_{e_{\theta}} + D_Y.
\]
As discussed in Section~\ref{ssec:near2}, on $W$ we have an
identification of sections of $S$ with multivalued sections of $S_0$,
the trivial spinor bundle on $W$.  Moreover, $ic(e_r)c(e_{\theta})$
squares to $1$ and is radially covariant constant. It thus gives a decomposition
of the spin bundle $S$ into $\pm 1$-eigenspaces
\[
  S\lvert_{W\setminus V} = S^{+} \oplus S^{-}
\]
preserved under parallel transport.

\paragraph{Step 3:} We proceed now to construct infinitely many
spinors $\psi \in \Dmax(D)$ with $\nabla \psi \notin L^2(M\setminus V,
S)$ which are linearly independent in $\Dmax(D)/\Dmin(D)$.

Let $\psi^{\pm \frac{1}{2}}(y)$ be two radially covariant constant 
extensions of smooth sections of $S_0\lvert_Y$ in $S^{\pm}$ on
$W\setminus Y$. Let also $\chi(r)$ be a smooth function supported in
$W$ and equal to $1$ in a smaller tubular neighborhood of $Y$.  Using
this input data, we consider the smooth spinor $\psi$ on $S$ supported
in $W$, whose Fourier mode decomposition~\eqref{eq:fourier_modes} is
\begin{align}
\psi(y,r,\theta) 
= \chi(r) &
   \left(
     r^{-1/2} e^{i\theta/2}  \psi^{\frac{1}{2}}(y) + 
     r^{-1/2}   e^{-i\theta/2} \psi^{-\onehalf}(y)
    \right.\nonumber \\
&\quad \left.  +
      c(e_r) r^{1/2} e^{i\theta/2}D_Y\psi^{\frac{1}{2}}(y)
     + c(e_r) r^{1/2}e^{-i\theta/2}
     D_Y\psi^{-\onehalf}(y)
     \right).
\end{align}
It is clear that $\psi \notin \Dmin(D)$ because $|\nabla \psi| \geq c
r^{-3/2}$ for small $r$, $r^{-3/2} \notin L^2((0,1), rdr)$,
and thus does not satisfy Corollary~\ref{cor:Dmin-ns}(1).
We now show that $\psi$ is in $\Dmax(D)$. We have
\begin{align*}
D \psi (y,r,\theta)& =
    \quad\left(-\frac{1}{2r} c(e_r) + \frac{i}{2r} c(e_{\theta})\right)
    \chi(r)  r^{-1/2} e^{i\theta/2} \psi^{\frac{1}{2}}(y)  \\
& \quad
   +  \left(-\frac{1}{2r} c(e_r) - \frac{i}{2r}
      c(e_{\theta})\right)  \chi(r)  r^{-1/2} e^{i\theta/2}
    \psi^{-\frac{1}{2}}(y) \\
 &\quad
     +\left(-\frac{1}{2} - \frac{i}{2} c(e_r)c(e_{\theta})\right)
        \chi(r) r^{-1/2} e^{i\theta/2}D_Y\psi^{\frac{1}{2}}(y)\\ 
 & \quad
    +\left(-\frac{1}{2} + \frac{i}{2}c(e_r) c(e_{\theta})\right) 
       \chi(r) r^{-1/2}e^{-i\theta/2}D_Y\psi^{-\onehalf}(y)\\
&\quad 
   + \chi(r) r^{1/2} e^{i\theta/2}D_Y(c(e_r)D_Y\psi^{\frac{1}{2}})(y) 
   +  \chi(r) r^{1/2}D_Y(c(e_r)e^{-i\theta/2} D_Y\psi^{-\onehalf})(y)
    +  \sO(r^{-1/2}),
\end{align*}
where the error term $\sO(r^{-1/2})$ arises from the difference
between the Riemannian metric on $W$ and the product metric on the
tubular neighborhood (see~\eqref{eq:dr},~\eqref{eq:dtheta}
and~\eqref{eq:cd}), as well as from the term containing the
derivatives of $\chi$.

Since $(-\frac{1}{2} - \frac{i}{2} c(e_r) c(e_{\theta})$ and
$(-\frac{1}{2} + \frac{i}{2} c(e_r) c(e_{\theta}))$ are the
projections onto the $\pm 1$-eigenspace of $i c(e_r) c(e_{\theta})$
respectively, the first two terms in the expression of $D\psi$ vanish,
giving $D\psi \in L^2(M\setminus V, S)$. Hence $\psi \in \Dmax(D)$.

Clearly we can construct infinitely many linearly independent elements
of $\Dmax(D)/\Dmin(D)$ in this fashion, just by choosing them to have
disjoint supports. 

\begin{remark}
  Note that if $\psi$ is the Witten spinor on $M\setminus V$
  asymptotic to $\psi_0$ constructed in Theorem A, then all the
  elements in $\psi + H_{\sing}$ are also Witten spinors asymptotic to
  $\psi_0$. Thus, the space of Witten spinors asymptotic to $\psi_0$
  is infinite dimensional.
\end{remark}

%
\newcommand{\etalchar}[1]{$^{#1}$}

\end{document}